%% file: 2CYpurity_v5.tex
\documentclass[a4paper,reqno,10pt]{amsart}
\usepackage[utf8]{inputenc}
\usepackage{geometry}
\usepackage{amssymb,amsmath,dsfont,tikz-cd}
\usepackage{mathtools}
\usepackage{hyperref}
\usepackage[all]{xy}
\usepackage{upgreek}
\usepackage{colonequals}
\usepackage{todonotes}
\usepackage{amsthm}
\usepackage{mathrsfs}
\usepackage{stmaryrd}
\usepackage{bm}

\usepackage{etoolbox}

\input{macros}

\title{Purity and 2-Calabi--Yau categories}
\author{Ben Davison}

\begin{document}

\setcounter{tocdepth}{2}% to get subsubsections in toc

\let\oldtocsection=\tocsection

\let\oldtocsubsection=\tocsubsection

\let\oldtocsubsubsection=\tocsubsubsection

\renewcommand{\tocsection}[2]{\hspace{0em}\oldtocsection{#1}{#2}}
\renewcommand{\tocsubsection}[2]{\hspace{3em}\oldtocsubsection{#1}{#2}}
\renewcommand{\tocsubsubsection}[2]{\hspace{2em}\oldtocsubsubsection{#1}{#2}}

\makeatletter
\patchcmd{\@tocline}
  {\hfil}
  {\leaders\hbox{\,.\,}\hfil}
  {}{}
\makeatother

\maketitle

\vspace{-0.2in}
\begin{center} 
{\it In memory of Tom Nevins}
\end{center}

\begin{abstract}
    For various 2-Calabi--Yau categories $\mathscr{C}$ for which the classical stack of objects $\mathfrak{M}$ has a good moduli space $p\colon\mathfrak{M}\rightarrow \mathcal{M}$, we establish purity of the mixed Hodge module complex $p_{!}\ul{\BQ}_{\FM}$.  We do this by using formality in 2CY categories, along with \'etale neighbourhood theorems for stacks, to prove that the morphism $p$ is modelled \'{e}tale-locally by the semisimplification morphism from the stack of modules of a preprojective algebra.   Via the integrality theorem in cohomological Donaldson--Thomas theory we then prove purity of $p_{!}\ul{\BQ}_{\FM}$.  It follows that the Beilinson--Bernstein--Deligne--Gabber decomposition theorem for the constant sheaf holds for the morphism $p$, despite the possibly singular and stacky nature of $\FM$, and the fact that $p$ is not proper.  We use this to define cuspidal cohomology for $\FM$, which conjecturally provides a complete space of generators for the BPS algebra associated to $\mathscr{C}$.

    We prove purity of the Borel--Moore homology of the moduli stack $\mathfrak{M}$, provided its good moduli space $\mathcal{M}$ is projective, or admits a suitable contracting $\BC^*$-action.  In particular, when $\mathfrak{M}$ is the moduli stack of Gieseker semistable sheaves on a K3 surface, this proves a conjecture of Halpern-Leistner.  We use these results to moreover prove purity for several stacks of coherent sheaves that do not admit a good moduli space.  
    
    Without the usual assumption that $r$ and $d$ are coprime, we prove that the Borel--Moore homology of the stack of semistable degree $d$ rank $r$ Higgs sheaves is pure and carries a perverse filtration with respect to the Hitchin base, generalising the usual perverse filtration for the Hitchin system to the case of singular stacks of Higgs sheaves.
    
\end{abstract}

\tableofcontents

\section{Introduction}
\label{sec:introduction}
\subsection{2-Calabi--Yau categories}
\label{glob_intro}
2-Calabi--Yau (2CY) categories are central objects of study in several areas of mathematics, including nonabelian Hodge theory, geometric representation theory, and the algebraic geometry of K3 and Abelian surfaces.  In this paper we study moduli stacks of objects in 2CY categories, and in particular the Hodge theory of such stacks, and perverse filtrations of their Borel--Moore homology arising from cohomological Donaldson--Thomas theory.
\smallbreak
Using work of Deligne \cite{DelII,DelIII} the Borel--Moore homology of each of the stacks $\FM$ that we will consider in this introduction carries a mixed Hodge structure, generalising the classical construction of the Hodge structure of weight $n$ on the $n$th singular cohomology of a smooth projective variety.  This Hodge structure can indeed be mixed for some of the stacks that we will consider: this is no surprise, since these stacks will generally be highly singular, stacky, with good moduli spaces that furthermore may not be projective.  Nonetheless, we prove that in many cases of interest this mixed Hodge structure is in fact pure.  For example, we prove in this way Halpern--Leistner's conjecture on the purity of the compactly supported cohomology of the stack of semistable coherent sheaves on a K3 surface.
\smallbreak
We tackle the subject through Saito's theory of mixed Hodge modules.  In this introduction we only
 consider stacks that come with a good moduli space $p\colon \FM\rightarrow \CM$, i.e. a morphism to a scheme enjoying good properties, that is universal (see Definition~\ref{def:good_moduli_space}, and see \S \ref{gen_sheaves} for more general results in the absence of a good moduli space).  We prove \textit{global} results on mixed Hodge structures, which we can loosely think of as the study of the derived direct image of the constant complex\footnote{The looseness here is that we never define this complex, and only define its direct images.} of mixed Hodge modules $\ul{\BQ}_{\FM}$ along the structure morphism, by instead studying the \textit{local} theory of $p$.  In other words, we study\footnote{Throughout the paper, all direct and inverse image functors are assumed to be derived.} $(\FM\rightarrow \CM)_!\ul{\BQ}_{\FM}$ as a route to understanding $(\FM\rightarrow \pt)_!\ul{\BQ}_{\FM}\cong (\CM\rightarrow \pt)_!(\FM\rightarrow \CM)_!\ul{\BQ}_{\FM}$.  While the purity result for Borel--Moore homology applies to some 2CY categories, and not to others, we show that for \textit{all} 2CY categories for which there is a good moduli space $p$ of objects in $\mathscr{C}$, the mixed Hodge module $\Ho^i\!p_!\ul{\BQ}_{\FM}$ is pure of weight $i$ for all $i$.  We take cohomology objects with respect to the natural t-structure on the derived category of mixed Hodge modules, which is a lift of the \emph{perverse} t-structure on the underlying constructible complex, and we deduce (from a result of Deligne \cite{Del87}) that the underlying perverse cohomology sheaves ${}^{\mathfrak{p}}\!\Ho^i\!p_!\BQ_{\FM}$ are semisimple.
\smallbreak
As a means of reducing to the case of stacks of representations of preprojective algebras, we first show that objects in left 2CY categories with semisimple endomorphism algebras have formal Yoneda algebras.  This generalises a catalogue of formality results in the literature (see \S \ref{prev_work_sec} for a summary).  In particular, this reproves the Kaledin--Lehn conjecture \cite{KaLe07,MR3942159} for coherent sheaves on K3 surfaces, as well as its generalisation to Kuznetsov components in the derived category of cubic fourfolds or Gushel--Mukai varieties.  We combine formality and derived deformation theory with fundamental results of Alper, Hall and Rydh on the \'etale local structure of stacks to deduce that the morphism from the moduli stack to the coarse moduli space of objects in a 2CY category is \'{e}tale locally isomorphic to the analogous morphism from the moduli stack of representations of a preprojective algebra to its affinization.  Local purity then follows from results in cohomological Donaldson--Thomas theory, i.e. dimensional reduction combined with the cohomological integrality theorem.  Putting all of this together, we establish that $\Ho^i\!p_!\ul{\BQ}_{\FM}$ is pure of weight $i$.  

\sssct
\label{exam_2CY}
Before stating the main results a little more exactly, we introduce our motivating set of examples.

\begin{enumerate}
    \item 
    \label{preproj_class}
    Let $\ol{Q}$ be the double of a quiver $Q$, and let $\Pi_Q \coloneqq \BC \ol{Q}/\langle \sum_{a\in Q_1}[a,a^*]\rangle$ be the preprojective algebra associated to $Q$.  We denote by $\FM^{\zeta\sstab}_{\dd}(\Pi_Q)$ the moduli stack of $\dd$-dimensional $\zeta$-semistable $\Pi_Q$-modules, where $\zeta\in\BQ^{Q_0}$ is a King stability condition, and by $\CM_{\dd}^{\zeta\sstab}(\Pi_Q)$ its coarse moduli space constructed via GIT by King~\cite{MR1315461}.
        There is a semisimplification morphism
        \begin{equation}
            \JH^{\zeta}_{\dd\colon}\FM_{\dd}^{\zeta\sstab}(\Pi_Q) \to \CM^{\zeta\sstab}_{\dd}(\Pi_Q)
        \end{equation}
        taking a semistable module to the associated polystable module.  Restricting to the special case in which $\zeta$ is degenerate, i.e. decreeing that all $\Pi_Q$-modules are semistable and dropping $\zeta\sstab$ from the notation, this is the example that we show governs the local situation in more general 2CY categories.
    \item 
    \label{K3_class}
    Let $S$ be a smooth complex projective Abelian or K3 surface, let $H$ be an ample class, let $\nu$ be a Mukai vector, and let $\FM^{H}_{\nu}(S)$ be the moduli stack of $H$-Gieseker semistable coherent sheaves of Mukai vector $\nu$.
        This stack admits a projective good moduli space:
        \begin{equation}\label{eq:2CY_category_1}
             \pi^{H}_{\nu} \colon \FM^{H}_{\nu}(S) \to \CM^{H}_{\nu}(S).
        \end{equation}
        More generally, we may assume merely that $S$ is a smooth quasiprojective surface satisfying $\mathcal{O}_S\cong \omega_S$, with $H$ an ample class on a compactification $\overline{S}$ of $S$, and consider moduli spaces of compactly supported $H$-Gieseker semistable coherent sheaves on $S$.
    \item 
    \label{Higgs_class}
    Let $C$ be a smooth genus $g$ complex projective curve and let $\FM^{\Dol}_{r,d}(C)$ denote the complex Artin stack of semistable rank $r$ and degree $d$ Higgs sheaves on $C$, i.e. semistable pairs $(\mathcal{F},\phi\colon\CF\rightarrow \CF\otimes\omega_C)$ of a rank $r$ degree $d$ coherent sheaf $\CF$ on $C$ and a Higgs field $\phi$.  This stack admits a universal morphism $p$ to its good moduli space, and by postcomposing with the Hitchin map $\Hit$, a morphism to the Hitchin base:
        \begin{equation}\label{eq:2CY_category_2}
        \xymatrix{     \FM^{\Dol}_{r,d}(C) \ar[r]^p \ar[dr]&\CM^{\Dol}_{r,d}(C)\ar[d]^{\Hit}\\
        &\Lambda_r\coloneqq\bigoplus_{m= 1}^r \HO^0(C,\omega_C^{\otimes m}).}
        \end{equation}
Since we do not assume that $d$ and $r$ are coprime, both $\FM^{\Dol}_{r,d}(C)$ and $\CM^{\Dol}_{r,d}(C)$ may be singular, and stabiliser groups of complex points in $\FM^{\Dol}_{r,d}(C)$ can be more complicated than $\mathbb{C}^*$.
    \item 
    \label{Betti_class}
    Let $\Sigma_{g}$ be a genus $g$ orientable closed Riemann surface without boundary and let $\FM^{\Betti}_{g,r}$ denote the moduli stack of $r$-dimensional $\pi_1(\Sigma_g)$-modules. 
        This stack admits a semisimplification morphism to its coarse moduli space,
        \begin{equation}\label{eq:2CY_category_3}
            \JH \colon \FM^{\Betti}_{g,r} \to \CM^{\Betti}_{g,r}.
        \end{equation}
        The coarse moduli space of this and the previous example form two sides of the nonabelian Hodge correspondence.  In contrast with the case of Higgs bundles, it is already known (just as for the smooth character varieties in \S \ref{smooth_case}) that $\HO^{\BM}(\FM^{\Betti}_{g,r},\BQ)$ is \textit{not} pure if $g\geq 1$.
    \item
    \label{mpreproj_class}
    Let $Q$ be a quiver, let $q\in(\mathbb{C}^*)^{Q_0}$, and let $\Lambda^q(Q)$ be the multiplicative preprojective algebra (see \S\ref{sec:mult_pp} for definitions and background).  Fix a stability condition $\zeta\in\mathbb{Q}^{Q_0}$.  Then as in (1) and (4) there is a semisimplification morphism
    \begin{equation}\label{eq:2CY_category_4}
            \JH^{\zeta}_{\dd}\colon\FM_{\dd}^{\zeta\sstab}(\Lambda^q(Q)) \to \CM^{\zeta\sstab}_{\dd}(\Lambda^q(Q)).
        \end{equation}
        Again, the Borel--Moore homology of the moduli stack $\FM_{\dd}^{\zeta\sstab}(\Lambda^q(Q))$ can be shown to be impure in basic examples.
\end{enumerate}

\subsection{Hodge theoretic results}

The following theorem is our first main result.
\begin{thmx}\label{thm:purity_absolute}(Theorems \ref{1a}, \ref{HL_con}, \ref{Higgs_pure})
    Let $\FM$ be one of the moduli stacks of objects in the 2CY categories of type (\ref{preproj_class}), (\ref{K3_class}) or (\ref{Higgs_class}) described above, where in case (\ref{K3_class}) we assume that $S$ is projective.
    The natural mixed Hodge structure on the Borel--Moore homology $\HOBM(\FM,\BQ)$ is \textbf{pure}.
\end{thmx}
The purity in case (\ref{preproj_class}) is one of the main results of \cite{2016arXiv160202110D}.
The purity of the Borel--Moore homology of the moduli stack of Gieseker semistable sheaves on a smooth projective K3 surface, part of case (\ref{K3_class}), was conjectured by Halpern--Leistner~\cite[Conj.~4.4]{HL16}, and is new.  In the special case of coprime $(r,d)$, the purity in case (\ref{Higgs_class}) is a result of \cite{HT03}, and can also be deduced from a standard ``sandwich'' argument: if a smooth quasiprojective scheme is contracted by a $\mathbb{C}^*$-action onto a projective scheme, the mixed Hodge structure on its cohomology is pure.  In the non-coprime case the stack $\FM^{\Dol}_{r,d}(C)$ is highly singular, but we prove that the mixed Hodge structure on the Borel--Moore homology is nonetheless pure.  We do this by melding the standard sandwich argument with our second main result, on local purity, which we come to next.
\sssct
Let $p \colon \FM \to \CM$ be the coarse moduli space of objects for one of the 2CY categories in \S \ref{exam_2CY}.  Recall that a complex of mixed Hodge modules $\CF$ is called pure if for every $i$ the $i$th cohomology mixed Hodge module $\CH^{i}(\CF)$ is pure of weight $i$.  We refer the reader to \S \ref{sec:algebraic_MHMs} for details and references regarding mixed Hodge modules.  To deduce Theorem~\ref{thm:purity_absolute} we first establish purity (of weight $i$) of the complex of mixed Hodge modules $\Ho^i\!p_{!}\ul\BQ_{\FM}$ on $\CM$, the derived direct images with compact support of the locally constant mixed Hodge module complex $\ul\BQ_{\FM}$.
\smallbreak
The following theorem is our second main result, which is most cleanly expressed in the language of derived moduli stacks.
\begin{thmx}\label{thm:purity_relative}(Theorem \ref{loc_pur_thm})
Let $\bm{\FM}\subset \bm{\FM}_{\mathscr{A}}$ be an open substack of the derived moduli stack of objects in an Abelian category $\HO^0(\mathscr{A})$, with $\mathscr{A}$ a full dg subcategory of a left 2CY dg category $\mathscr{C}$.  Let $\FM=t_0(\bm{\FM})$ be the classical truncation, and assume that closed points of $\FM$ represent semisimple objects of $\mathscr{A}$.  Let $p\colon \FM\rightarrow \CM$ be a good moduli space.  Then the mixed Hodge module $\Ho^i\!p_!\underline{\BQ}_{\FM}$ is pure of weight $i$, and 
\[
\left(\Ho^i\!p_!\underline{\BQ}_{\FM}\right)\lvert_{p(x)}=0
\]
for $x$ representing an object $\mathcal{F}$ of $\mathscr{C}$ satisfying $i>\chi_{\mathscr{C}}(\CF,\CF)$, or for $i$ odd.
\end{thmx}
The condition on the closed points of $\FM$ holds, for example, if $\HO^0(\mathscr{A})$ is a finite length Abelian category.  This explains the restriction to categories of semistable objects in our examples concerning coherent sheaves and Higgs bundles.
\sssct
Since Saito's theory of mixed Hodge modules is not yet fully formulated to cover algebraic stacks, one of the difficulties that we will encounter is actually defining the derived direct image $p_!\underline{\BQ}_{\FM}$, as a mixed Hodge module complex, at the level of generality demanded by Theorem \ref{thm:purity_relative}.  Indeed, part of the content of the theorem is that at least $\Ho^i\!p_!\underline{\BQ}_{\FM}$ can be defined as a mixed Hodge module (see Definition/Proposition \ref{dfp}).  This is why the theorem is expressed at the level of the $i$th derived functors, as opposed to objects in the derived category.  Once we pass to concrete applications, it becomes easier to define the complex $p_!\underline{\BQ}_{\FM}$, so that we can use Theorem \ref{thm:purity_relative}, as per the strategy of \S \ref{glob_intro}, to study $(\FM\rightarrow \pt)_!\underline{\BQ}_{\FM}$ via an understanding of $(\FM\rightarrow \CM)_!\underline{\BQ}_{\FM}$.  In particular, we deduce the following from Theorem \ref{thm:purity_relative}.
\begin{thmx}
\label{ex_pur_thm}
    Let $p \colon \FM \to \CM$ be the morphism to the coarse moduli space in any of the examples of categories (1)-(5) above.
    The complex of mixed Hodge modules $p_{!}\ul\BQ_{\FM}$ is \textbf{pure}.  
    
    The Borel--Moore homology $\HOBM(\FM,\BQ_{\vir})\coloneqq \HO_c(\FM,\BD(\BQ[\chi_{\mathscr{C}}(\cdot,\cdot)]))$ carries a natural associated perverse filtration by mixed Hodge structures, beginning in perverse degree zero, with associated graded object
    \[
    \bigoplus_{n\geq 0}\HO\!\left(\CM,\BD \Ho^{-n}\!\left(p_!\underline{\BQ}_{\FM}[\chi_{\mathscr{C}}(\cdot,\cdot)]\right)\right)[-n]
    \]
    where $\BD$ denotes the Verdier duality functor.
\end{thmx}
\sssct
Forgetting the extra structure that goes into a mixed Hodge module complex (i.e. remembering only the underlying complexes of constructible sheaves), Theorem \ref{ex_pur_thm} states that the Beilinson--Bernstein--Deligne--Gabber decomposition theorem \cite{MR751966} holds for the pair $(p\colon \FM\rightarrow \CM,\BQ_{\FM})$, despite the messy topology of $\FM$ and $p$.  In other words, there is a decomposition in the bounded above derived category of constructible sheaves on $\CM$,
\[
p_!\BQ_{\FM}\cong \bigoplus_{n\in\BZ}\bigoplus_{j\in J_{n}}\IC_{\ol{Z_j}}(\CL_j)[-n],
\]
where $Z_j$ are irreducible locally closed subvarieties of $\CM$ indexed by $J$, and $\CL_j$ are local systems on the varieties $Z_j$.  The theorem moreover states that this decomposition can be upgraded to a decomposition in the derived category of mixed Hodge modules, for which the summands appearing in the decomposition are simple.
\sssct
\label{phold}
We prove Theorem \ref{thm:purity_relative} by showing that the morphism to the good moduli space $p \colon \FM \to \CM$ is \'etale-locally modelled by the morphism $\JH_{\dd} \colon \FM_{\dd}(\Pi_{Q'}) \to \CM_{\dd}(\Pi_{Q'})$ for some quiver $Q'$ and dimension vector $\dd$.  
Since purity may be checked locally (see Lemma \ref{loc_pur_lem}), Theorem~\ref{thm:purity_relative} follows from purity in case (\ref{preproj_class}) of \S \ref{exam_2CY}:
\begin{theorem}[{\cite[Thm.~A]{preproj3}}]\label{preprojy}
    Let $Q$ be a quiver, and let $\mathbf{d}$ be a dimension vector.
    The complex of mixed Hodge modules,
    \begin{equation}
        \JH_{\dd,!} \ul\BQ_{\FM_{\dd}(\Pi_{Q})}
    \end{equation}
    is \textbf{pure}, i.e. $\Ho^i\!\JH_{\dd,!} \ul\BQ_{\FM_{\dd}(\Pi_{Q})}$ is pure of weight $i$.
\end{theorem}
The statement regarding nonvanishing of cohomological degrees of Theorem \ref{thm:purity_relative} is likewise local, and is proved by comparison with the category of representations of preprojective algebras.
\smallbreak
The direct image of a pure complex of mixed Hodge modules along a projective morphism is pure, by Saito's version of the decomposition theorem.  So if the moduli space of objects $\CM$ is projective, and we can define $p_!\underline{\BQ}_{\FM}$ in the derived category of mixed Hodge modules, then Theorem~\ref{thm:purity_relative} implies purity of $\HO_c(\FM,\BQ)$.  This concludes case (\ref{K3_class}) of Theorem \ref{thm:purity_absolute} in the case of projective $S$.
For case (\ref{Higgs_class}), we employ the natural $\BC^*$-action scaling the Higgs field to complete the argument.
\subsection{Formality}
The statement in \S \ref{phold} regarding \'etale neighbourhoods requires the following formality statement.  Firstly, given an object $\mathcal{F}$ in a category $\mathscr{C}$, we call $\mathcal{F}$ a $\Sigma$\textit{-object} if 
\[
\dim(\Ext^i(\mathcal{F},\mathcal{F}))=\begin{cases} 1& \textrm{if }i=0,2\\
2g& \textrm{for some }g\in\mathbb{Z}_{\geq 0}\textrm{ if }i=1
\\
0&\textrm{otherwise.}\end{cases}
\]
We call a collection $\mathcal{F}_1,\ldots,\mathcal{F}_r$ a $\Sigma$\textit{-collection} if each $\mathcal{F}_t$ is a $\Sigma$-object, and $\Hom(\mathcal{F}_m,\mathcal{F}_n)=0$ for $m\neq n$.  The terminology is inspired by \cite{ST04,HT06}: i.e. if we modify our conditions on the dimensions of the Yoneda algebras of the $\mathcal{F}_i$ we would instead have the definition of a spherical/$\mathbb{P}^n$-collection, but because of the degree 1 extensions the Yoneda algebra of a single $\Sigma$-object resembles the cohomology ring of a Riemann surface $\Sigma_g$, possibly with $g>0$, as opposed to that of $\mathbb{P}^n$ or a sphere.
\begin{theorem}
\label{formality_thm}
Let $\mathcal{F}_1,\ldots,\mathcal{F}_r$ be a $\Sigma$-collection in a left 2CY category $\mathscr{C}$.  Then the full dg-subcategory of $\mathscr{C}$ containing $\mathcal{F}_1,\ldots,\mathcal{F}_r$ is formal.
\end{theorem}
We refer the reader to \S \ref{LeftRight} for the exact definition of the ``left'' Calabi--Yau condition, and for now just remark that all of the examples introduced in \S \ref{glob_intro} are left 2CY categories.
\smallbreak
Theorem \ref{formality_thm} has many antecedents (at least if $\CF_1,\ldots,\CF_r$ belong to the natural heart of a given 2CY category) going back more than forty years in the case of nonabelian Hodge theory (e.g. see \cite{Del75,Go88}).  For each of the areas of mathematics that 2CY categories appear in, it seems that there is a parallel literature devoted to this problem.  We relate Theorem \ref{formality_thm} to previous work in \S \ref{prev_work_sec}.  The above formality theorem shows that all of these formality results are special cases of a very general statement in 2CY categories, and in each instance can be extended to arbitrary $\Sigma$-collections of complexes in these categories.

\subsection{New perverse filtrations in nonabelian Hodge theory}
\label{perv_intro}
One of our motivations for extending the BBDG decomposition theorem to 2CY categories comes from its prominent role in the study of moduli spaces of Higgs bundles, and in particular the construction there of perverse filtrations.  For the purposes of this discussion, let us assume that $g\geq 2$, where $g$ is the genus of the smooth projective connected complex curve $C$.  We will also work in the derived category of constructible sheaves, as opposed to mixed Hodge modules, for simplicity of exposition.  
\sssct
\label{smooth_case}
We assume throughout \S \ref{perv_intro} that $r\geq 1$, and for a warmup we assume in \S \ref{smooth_case} that $(r,d)=1$.  Recall that there is a projective morphism defined by Hitchin \cite{Hit87}:
\[
\Hit\colon \CM^{\Dol}_{r,d}(C)\rightarrow \Lambda_r\coloneqq \bigoplus_{m= 1}^r \HO^0(C,\omega_C^{\otimes m})
\]
which records the eigenvalues of $\phi$ for a semistable Higgs bundle $(\CF,\phi)$ of rank $r$ and degree $d$.  Since the morphism $\Hit$ is projective and $\CM^{\Dol}_{r,d}(C)$ is smooth, the BBDG decomposition theorem implies that there exist isomorphisms
\begin{align}
    \label{half_decomp}\Hit_*\!\BQ_{\CM^{\Dol}_{r,d}(C)}\cong&\bigoplus_{n\in\BZ}\pvrs\Ho^n\!\left(\Hit_*\!\BQ_{\CM^{\Dol}_{r,d}(C)}\right)[-n]\\
    \pvrs\Ho^n\!\left(\Hit_*\!\BQ_{\CM^{\Dol}_{r,d}(C)}\right)\cong &\bigoplus_{j\in J_{r,d,n}}\IC_{\ol{Z_j}}(\CL_j)\label{full_decomp}
\end{align}
for $Z_j$ a collection of locally closed irreducible subvarieties of $\CM^{\Dol}_{r,d}(C)$, and $\CL_j$ a collection of simple local systems on the $Z_j$ (see \cite{CHM17} for a more explicit description of the decompositions in the coprime case).  While the first decomposition \eqref{half_decomp} is not prima facie canonical, the objects
\[
\bigoplus_{n\leq j}\pvrs\Ho^n\!\left(\Hit_*\!\BQ_{\CM^{\Dol}_{r,d}(C)}\right)[-n]
\]
are, and so taking hypercohomology, we obtain the \textit{perverse filtration} of $\HO(\CM^{\Dol}_{r,d}(C),\BQ)$ by subspaces
\begin{equation}
\label{ent_per}
P_j\coloneqq\HO\!\left(\Lambda_r,\bigoplus_{n\leq j}\pvrs\Ho^n\!\left(\Hit_*\!\BQ_{\CM^{\Dol}_{r,d}(C)}\right)[-n]\right).
\end{equation}
This perverse filtration has been the object of intense study for some time, in particular as a result of the P=W conjecture of de Cataldo, Hausel and Migliorini \cite{dCHM12} (see e.g. \cite{deC19, deC20, CHS20} for recent work on this conjecture, and see \cite{MaSh22,HMMS22} for more recent independent proofs of the conjecture, as well as \cite{MSY23b}).  This theorem relates the perverse filtration on $\HO(\CM^{\Dol}_{r,d}(C),\BQ)$ with the weight filtration on $\HO(\CM^{\Betti}_{g,r},\BQ)$ under the isomorphism in cohomology induced by the diffeomorphism
\[
\Phi\colon \CM^{\Betti,\twist}_{g,r}\rightarrow \CM^{\Dol}_{r,d}(C).
\]

Here, $\CM^{\Betti,\twist}_{g,r}$ is the twisted character variety, the moduli space of representations of the fundamental group of the punctured curve $C'=C\setminus \{x\}$ with monodromy around the puncture $x$ given by multiplication by $e^{2\pi \sqrt{-1}d/r}$.  The diffeomorphism $\Phi$ is the morphism constructed by Hitchin \cite{Hit87b}, Donaldson \cite{Don87} and Corlette \cite{Cor88}, underlying the nonabelian Hodge correspondence.
\sssct
We will be most interested in the case for which $r$ and $d$ are \textit{not} coprime.  Then the moduli space $\CM^{\Dol}_{r,d}(C)$ is no longer smooth, and the decomposition theorem no longer applies to $\Hit_*\!\BQ_{\CM^{\Dol}_{r,d}(C)}$.  One standard remedy for this (pursued in e.g. \cite{Fel18,FeMa20,MM21}) is to replace $\BQ_{\CM^{\Dol}_{r,d}(C)}$ by the intersection complex $\IC_{\CM^{\Dol}_{r,d}(C)}(\BQ)$.  Then by the decomposition theorem for intersection complexes, we have decompositions analogous to \eqref{half_decomp} and \eqref{full_decomp}
\begin{align*}
    \Hit_*\!\IC_{\CM^{\Dol}_{r,d}(C)}(\BQ)\cong&\bigoplus_{n\in\BZ}\pvrs\Ho^n\!\left(\Hit_*\!\IC_{\CM^{\Dol}_{r,d}(C)}(\BQ)\right)[-n]\\
    \pvrs\Ho^n\!\left(\Hit_*\!\IC_{\CM^{\Dol}_{r,d}(C)}(\BQ)\right)\cong &\bigoplus_{j\in J'_{r,d,n}}\IC_{\ol{Z_j}}(\CL_j)
\end{align*}
and a perverse filtration analogous to \eqref{ent_per}
\[
P^{\IC}_j=\HO\!\left(\Lambda_r,\bigoplus_{n\leq j}\pvrs\Ho^n\!\left(\Hit_*\!\IC_{\CM^{\Dol}_{r,d}(C)}(\BQ)\right)[-n]\right)
\]
of the intersection cohomology $\ICA(\CM^{\Dol}_{r,d}(C),\BQ)$.
\smallbreak
A second new feature of the non-coprime case is that the stack $\FM^{\Dol}_{r,d}(C)$ is no longer a $\BC^*$-gerbe over the coarse moduli space, and has quite different topology.  In addition, it is also highly singular, and so there is a difference between the Borel--Moore homology (i.e. the derived global sections of the dualizing complex $\BD\BQ_{\FM^{\Dol}_{r,d}(C)}$, or equivalently the dual of the compactly supported cohomology) and the singular cohomology of the stack.  It turns out that the Borel--Moore homology is a more well-behaved invariant, since it is more sensitive to the singularities, and is motivic.  
\smallbreak
A consequence of Theorem~\ref{thm:purity_relative} is that there are decompositions analogous to \eqref{half_decomp} and \eqref{full_decomp}
\begin{align*}
    (\Hit \!p)_*\BD\BQ_{\FM^{\Dol}_{r,d}(C)}\cong&\bigoplus_n\pvrs\Ho^n\!\left((\Hit\! p)_*\BD\BQ_{\FM^{\Dol}_{r,d}(C)}\right)[-n]\\
    \pvrs\Ho^n\!\left((\Hit \!p)_*\BD\BQ_{\FM^{\Dol}_{r,d}(C)}\right)\cong &\bigoplus_{j\in J''_{r,d,n}}\IC_{\ol{Z_j}}(\CL_j)
\end{align*}
so that we can define a perverse filtration of $\HO^{\BM}(\FM^{\Dol}_{r,d}(C),\BQ)$ by subspaces
\[
P^{\stacky}_j\coloneqq \HO\!\left(\Lambda_r,\bigoplus_{n\leq j}\Ho^n((\Hit \!p)_*\BD\BQ_{\FM^{\Dol}_{r,d}(C)})[-n]\right).
\]
This filtration generalises the perverse filtration on the cohomology of Higgs moduli spaces from the case of coprime $r,d$.  
\sssct
In general, and in contrast with the non-stacky/coprime case, there is no construction via the classical nonabelian Hodge homeomorphism of an isomorphism between the Borel--Moore homology of the Betti moduli stack and the Dolbeault moduli stack.  However it is suggested by \cite{PS19} that there should be an isomorphism after passing to the associated graded object of filtrations on at least one side of the nonabelian Hodge correspondence.  It is certainly tempting to speculate that this filtration is provided by the perverse filtration (with respect to the morphism $p$); see \S \ref{Betti_sec} for further comments, along with a conjecture making this speculation concrete, and see \cite{DHSM22} for a proof of this conjecture, using the results of the present paper.

\sssct
\label{2per}
We can relate the perverse filtration for $\HO^{\BM}(\FM^{\Dol}_{r,d}(C),\BQ)$ to the above standard remedy (involving intersection complexes) via another application of Theorem \ref{thm:purity_relative}, and the study of \textit{supports}.  In a little more detail, Theorem \ref{thm:purity_relative} yields a decomposition
\begin{equation}
\label{PDdec}
p_*\BD\BQ_{\FM^{\Dol}_{r,d}(C)}\cong\bigoplus_{n\in\BZ}\bigoplus_{j\in K_{r,d,n}}\IC_{\ol{Z_j}}(\CL_j)[-n]
\end{equation}
where $K_{r,d,n}$ are indexing sets, $Z_j$ are locally closed subspaces of $\CM^{\Dol}_{r,d}(C)$ and $\CL_j$ are simple local systems on them.  One may show, using Theorem \ref{ICin}, that a shift of the summand $\IC_{\CM^{\Dol}_{r,d}(C)}(\BQ)$ appears in this decomposition, so that there is an inclusion (up to overall shifts) of pure Hodge structures
\[
\ICA(\CM^{\Dol}_{r,d}(C),\BQ)\subset \HO^{\BM}(\FM^{\Dol}_{r,d}(C),\BQ).
\]
Moreover it then follows that there is a precise relation between the perverse filtrations:
\[
P_n^{\stacky}\cap \ICA(\CM^{\Dol}_{r,d}(C),\BQ)=P^{\IC}_n.
\]
See \S \ref{Higgs_sec} for the details, incorporating all of the correct shifts and the upgrade to mixed Hodge structures.

\subsection{Cuspidal cohomology and the source of purity}

The appearance of the summand $\IC_{\CM^{\Dol}_{r,d}(C)}(\BQ)$ in the decomposition \eqref{PDdec} is part of a general result for 2-Calabi--Yau categories $\mathscr{C}$ (Theorem \ref{ICin}).  This result is best viewed from the point of view of geometric representation theory.  It implies not just that the intersection cohomology of components of the coarse moduli space containing stable objects appear naturally as summands of the BPS algebra for $\mathscr{C}$, but moreover they appear as a canonical subspace of generators for this algebra.  We call this subspace \textit{cuspidal cohomology}, reflecting the close connection (via the Langlands correspondence) with geometrically cuspidal functions on $\mathrm{Bun}_{r,d}(C)$ for $C$ a smooth projective curve over $\mathbb{F}_q$. See \cite{Dr81,Ko09,De15,BoSc19,Yu23} for background and recent developments on cuspidal functions, and see \S \ref{cu_sec} for detailed discussion and results on cuspidal cohomology and BPS algebras.
\smallbreak
This connection between nonabelian Hodge theory and geometric representation theory offers a partial answer to the question: Where does all this purity come from?  For $(r,d)>1$ the stack $\FM^{\Dol}_{r,d}(C)$ is not amenable to either of the two standard arguments for proving purity, i.e. there is no known stratification into pieces that are manifestly pure, nor does it submit to the standard ``sandwich'' argument for proving purity of the Borel--Moore homology in the $(r,d)=1$ case.  Purity of the stack $\FM_{\mathbf{d}}(\Pi_Q)$ is unexpected for precisely the same reasons, but follows from a new type of argument.  Namely, by the cohomological integrality/PBW theorem \cite{BenSven} in Donaldson--Thomas theory, the Borel--Moore homology of $\FM_{\mathbf{d}}(\Pi_Q)$ is built up out of summands of tensor powers of ``BPS cohomology'' for $\mathbf{d}'\lvert \mathbf{d}$ and tautological classes in $\HO(\pt/\BC^*,\BQ)$, and it is possible to show that this BPS cohomology is pure \cite{Da18}.  
\smallbreak
The ultimate reason for purity in the case of, say, $\FM^{\Dol}_{r,d}(C)$ appears to be similar.  Via the cohomological integrality theorem in Donaldson--Thomas theory, the Borel--Moore homology of this stack should be thought of as a Yangian algebra associated to a BPS Lie algebra, a subspace of the zeroth piece of the perverse filtration that we construct in this paper.  This Yangian is built out of tensor powers of the BPS Lie algebra and tautological classes.  The BPS Lie algebra in turn is the positive half of a generalised Kac--Moody Lie algebra generated by cuspidal cohomology, which we show is isomorphic to the intersection cohomology of moduli spaces $\CM^{\Dol}_{r/n,d/n}(C)$, with $n=\mathrm{gcd}(r,d)$; this generation result is a special case of a result proved recently for very general 2CY categories in \cite{DHSM22,DHSM23}, building on the present paper.  This intersection cohomology \textit{is} pure by the sandwich argument (see \S \ref{Higgs_sec} and \S \ref{cu_sec}), and this purity transmits all the way to $\HO^{\BM}(\FM^{\Dol}_{r,d}(C),\BQ)$.  The same narrative applies to the purity of $\FM^H_{\nu}(S)$ for $S$ a K3 or Abelian surface.
\subsection{Structure of the paper}
In \S\ref{sec:mhm} we collect some facts from the theory of mixed Hodge structures and mixed Hodge modules, as well as describing the extension of the theory to a certain class of stacks including all of the stacks mentioned in \S \ref{glob_intro}.  In \S\ref{sec:quivers} we fix some quiver notation, and discuss Theorem \ref{preprojy}.  \S\ref{sec:2CY_local} has the dual purpose of collecting all of the category theory that we will need in the paper, and establishing formality in general 2CY categories.
\smallbreak
In \S\ref{local_struc} we combine recent results on the local structure of Artin stacks with formality to give an \'etale local description of good moduli spaces $p\colon \FM\rightarrow \CM$ of objects in 2CY categories.  \S\ref{sec:purity_perverse} contains the main general results of the paper, where we combine all of the material from the previous sections to prove results on mixed Hodge structures and perverse filtrations, as well as a support theorem that will be the key to defining cuspidal cohomology.  Along the way we give a more general definition of $\Ho^n\!p_!\underline{\BQ}_{\FM}$, with fewer restrictions on the stack $\FM$ than in \S\ref{sec:mhm}.  
\smallbreak
In the longest section, \S \ref{sec:applications}, we develop applications of our main results to each of the 2CY categories introduced in \S \ref{glob_intro}.  Additionally, we extend the results to some 2CY categories without good moduli spaces \S \ref{gen_sheaves}, give the definition and main theorem regarding cuspidal cohomology \S \ref{cu_sec}, explain connections with nonabelian Hodge theory for stacks \S \ref{Betti_sec} and intersection cohomology \S \ref{cpf}, and finally discuss more exotic 2CY categories coming from the study of Kuznetsov components \S \ref{sec:Kuznetsov_components}.  The paper finishes with the appendix, concerning the determination of coherent completions (in the sense of Alper, Hall and Rydh) in terms of Maurer--Cartan stacks.
\smallbreak
The paper is arranged in order of logical flow.  As a result of this, it may well be that the most reasonable way to read it for the working mathematician who cares first and foremost about, say Higgs bundles, or cohomological Hall algebras, or K3 surfaces, is to first read the introduction, and then skip straight to \S \ref{sec:purity_perverse}, or even the relevant part of \S \ref{sec:applications}, before reading the sections in-between if desired.

\subsection{Notations and conventions}
\begin{itemize}
\item
If an algebraic group $G$ acts on a scheme $X$, we will denote by $X/G$ the stack-theoretic quotient, and by $X/\!\!/G$ the categorical quotient.  If $x\in X$ is a geometric point, we will abuse notation by denoting by $x$ also the induced points of $X/G$ and $X/\!\!/G$.  For $\chi$ a linearisation of the $G$ action, we denote by $X/\!\!/_{\chi} G$ the GIT quotient.
\item
If $\mathcal{F}$ is an object of a triangulated category with an agreed-upon t-structure, we set $\Ho^i(\mathcal{F})=(\bm{\tau}^{\leq i}\bm{\tau}^{\geq i}\mathcal{F})[i]$.
\item
All direct and inverse image functors are considered as functors between derived categories or their dg enhancements unless explicitly stated otherwise.  
\item
Where a triangulated category is denoted by the symbols $\Dub^{?}(\ldots)$, we denote its dg enhancement by the symbols $\Dub_{\dg}^?(\ldots)$.
\item
For $X$ a finite type scheme or stack, we denote by $\BQ_X$ the constant constructible sheaf on $X$ with stalks $\BQ$, and by $\ul{\BQ}_X\in\Ob(\Db(\MHM(X)))$ its upgrade to a complex of mixed Hodge modules.
\item
Given algebras or categories $A$ and $B$, we denote by $\Mod_{A}$ the category of right $A$-modules, $\Mod^{A}$ the category of left $A$-modules, and $\Mod_{B}^{A}$ the category of $(A,B)$-bimodules\footnote{Here is a mnemonic for remembering the roles of the subscript/superscript.  Algebras act \textit{on} modules, so one should perhaps have both $A$ and $B$ \textit{above} $\Mod$.  This can be achieved by rotating $\Mod^A_B$ through ninety degrees, at which point $A$ is on the left, and $B$ is on the right.}.  We denote by $\fdmod_A,\fdmod^A,\fdmod_A^B$ the full subcategories containing modules for which the underlying vector space is finite-dimensional.
\item

If $\mathcal{F}$ is an object of a $k$-linear triangulated category $\mathscr{C}$, we denote by
\[
\chi_{\mathscr{C}}(\mathcal{F},\mathcal{F})=\sum_{n\in\mathbb{Z}}(-1)^n\dim_k (\Hom(\mathcal{F},\mathcal{F}[n]))
\]
the Euler form of $\CF$ with itself.
\item
If $\mathcal{F}$ is a cohomologically graded object in a tensor category $\mathscr{A}$, whenever we take the free commutative algebra $A=\Sym(\CF)$ generated by $\mathcal{F}$ we impose the Koszul sign rule.  So for example if $\mathcal{F}$ is a cohomologically graded vector space, and $a,b\in \mathcal{F}$, then in $A$
\[
a\cdot b=(-1)^{\lvert a\lvert \lvert b\lvert}b\cdot a.
\]
\item
We switch between cohomological and homological gradings by putting $\mathcal{F}^{i}=\mathcal{F}_{-i}$.
\item
By ``variety'' we mean a finite type reduced separated scheme $X$.  In particular we do not assume that $X$ is irreducible.  
\item
For $X$ a complex variety or stack, we write
\[
\HO\!\left(X,\BQ\right)\coloneqq \HO\!\left(X_{\mathrm{an}},\BQ\right);\quad \quad \HO_c\!\left(X,\BQ\right)\coloneqq \HO_c\!\left(X_{\mathrm{an}},\BQ\right);\quad\quad \HO_i^{\BM}\!\left(X,\BQ\right)\coloneqq \HO^i_c\!\left(X,\BQ\right)^{\vee}.
\]
\end{itemize}
\subsection{Acknowledgements}
This research was supported by the grant “Categorified Donaldson–Thomas theory” No. 759967 of the European Research Council, and by a Royal Society university research fellowship.  I would like to thank Lino Amorim, Victor Ginzburg, Daniel Halpern-Leistner, Maxim Kontsevich, Francesco Sala, Olivier Schiffmann, Travis Schedler, Sebastian Schlegel--Mejia, Nick Sheridan, and Yan Soibelman for helpful conversations and suggestions.  Special thanks go to Jon Pridham for patiently explaining DAG, and Carlos Simpson for pointing out an illuminating subtlety regarding the nonabelian Hodge correspondence for stacks.  In addition I would like to thank Sjoerd Beentjes for useful discussions and suggestions regarding an earlier draft, and for suggesting the application to Kuznetsov components.  Finally, special thanks go to the anonymous referee for a very careful reading, along with insightful suggestions that led to a strengthening of the results in \S \ref{CSsec} and \S \ref{gen_sheaves}.

\section{Hodge theory}
\label{sec:mhm}
\subsection{Mixed Hodge structures}
\label{sec:mixed_hodge_structures}
A pure $\BQ$-Hodge structure of weight $n$ on a finite-dimensional $\BQ$-vector space $V$ is a descending filtration of $V_{\BC}=V\otimes\BC$ by $\BC$-subspaces $F^pV_{\BC}$ for $p\in\BZ$ such that for all $r\in\mathbb{Z}$ we have $V_{\BC}=\overline{F^rV_{\BC}}\oplus F^{n+1-r}V_{\BC}$.  A mixed Hodge structure on a finite-dimensional $\BQ$-vector space $V$ is given by an ascending filtration $W_{\bullet}V$ of $V$, again indexed by the integers, along with a descending filtration $F^{\bullet}V_{\BC}$ of the complexification, such that for each $n\in\BZ$, the filtration induced by $F^{\bullet}$ on $(\Gr^W_{\bullet}\!V)_{\mathbb{C}}$ is a pure mixed Hodge structure of weight $n$.  By classical Hodge theory, if $X$ is a smooth projective complex variety, the singular cohomology groups $\HO^n(X,\BQ)$ carry pure Hodge structures of weight $n$.  
\subsubsection{}
Given a finite type complex scheme $X$, Deligne \cite{DelII, DelIII} defined a mixed Hodge structure on the singular cohomology $\HO(X,\BQ)$ and the compactly supported cohomology $\HO_c(X,\BQ)$.  This is all achieved by reducing to the case of smooth projective schemes, via the study of mixed Hodge structures on simplicial schemes satisfying cohomological descent (see \cite{DelIII} or \cite[Sec.5]{PS08}).  Considering a finite type stack $\FX$ as a simplicial scheme via the nerve construction applied to a smooth atlas of $\FX_{\mathrm{red}}$, Deligne's construction thus endows the singular cohomology of a complex finite type stack with a mixed Hodge structure.  This is exploited in \cite{DelIII} to define and study the mixed Hodge structure on $\HO(\pt/G,\BQ)$ for certain affine groups $G$.   There does not seem to be a straightforward adaptation in the literature of Deligne's theory to deal with compactly supported cohomology of simplicial schemes or stacks.  In the next section we will address this via a kind of generalised Borel construction\footnote{Though it should be noted that we are still following Deligne in a sense, since we are extending the observation of \cite{DelIII} that the Hodge structure on the cohomology of $\HO(\pt/\BC^*,\BQ)$ can be approximated by the Hodge structure on $\HO(\BP^N,\BQ)$ for $N\gg 0$.}, which is well-adapted to the more general theory of mixed Hodge modules.
\smallbreak
Denoting by $p\colon X\rightarrow \pt$ the structure morphism of the scheme $X$, and bearing in mind that the category of constructible sheaves on a point is isomorphic to the category of vector spaces, we may express Deligne's theorem by saying that 
\begin{align}
\label{tende}
\Ho^n((X\xrightarrow{p}\pt)_*\BQ_X) \in\Ob(\Vect),&&
\Ho^n((X\xrightarrow{p}\pt)_!\BQ_X) \in \Ob(\Vect)
\end{align}
carry mixed Hodge structures, which are the usual pure weight $n$ mixed Hodge structure on these vector spaces if $X$ is smooth and projective.
\subsubsection{}
We fix 
\[
\LLL\coloneqq \HO_c(\BA^1,\BQ).
\]
This is a one-dimensional complex of pure Hodge structures of weight $2$, concentrated in cohomological degree $2$.  Similarly, we define $\LLL^{-1}\coloneqq \HO_c(\BA^1,\BQ)^{\vee}$ to be the dual.  Explicitly, this is a one-dimensional complex of pure Hodge structures of weight $-2$, concentrated in cohomological degree $-2$.  We define $\LLL^r$ for arbitrary $r\in\BZ$ by taking tensor powers of these complexes.

\subsection{Mixed Hodge modules}
\label{sec:algebraic_MHMs}
In this paper we will make extensive use of the theory of algebraic mixed Hodge modules, as developed in \cite{Sai88, Sai90}.  This is an extremely powerful generalisation of the theory of mixed Hodge structures, where in \eqref{tende} we may relax the condition that the target of $p$ is a point, and allow for more interesting coefficient sheaves than the constant sheaf $\BQ_X$.  For each variety $X$, Saito constructs a category of mixed Hodge modules on $X$, with functors between these categories lifting the usual six functors of the (derived) category of analytically constructible sheaves on these varieties.  For the reader that would like to gain more familiarity with this theory than the bare bones presented here, but without necessarily working through the references above, we recommend \cite{Sai89,Sai16,Schnell} as excellent introductions.

\subsubsection{}
Given a variety $X$, we denote by $\MHM(X)$ the category of mixed Hodge modules on $X$, and by $\Db(\MHM(X))$ the bounded derived category.  There are faithful functors
\[
    \rat\colon\MHM(X)\rightarrow \Perv(X);\quad\quad
    \rat_W\colon\MHM(X)\rightarrow \Perv_{\mathrm{Filt}}(X)
\]
taking a mixed Hodge module to its underlying perverse sheaf, or perverse sheaf equipped with the weight filtration, respectively.  In particular, the natural heart of the category $\Db(\MHM(X))$ corresponds to the heart of the perverse t-structure on the bounded derived category of analytically constructible sheaves on $X$.  One may identify $\MHM(\pt)$ with the full subcategory of the category of $\BQ$-Hodge structures that are graded-polarisable\footnote{See \cite[Sec.3.3]{PS08} for the definition of polarisation, that we will not use.}.

\smallbreak
If $f\colon X\rightarrow Y$ is a morphism of separated finite type schemes, we denote by 
\begin{align*} f_!,f_*\colon& \Db(\MHM(X))\rightarrow \Db(\MHM(Y))\\
f^!,f^*\colon& \Db(\MHM(Y))\rightarrow \Db(\MHM(X))
\end{align*}
the (exceptional, respectively usual) direct and inverse image functors at the level of derived categories.  We denote by
\begin{align*}
\Ho^i\!f_!,\Ho^i\!f_*\colon &\Db(\MHM(X))\rightarrow \MHM(Y)\\
\Ho^i\!f^!,\Ho^i\!f^*\colon &\Db(\MHM(Y))\rightarrow \MHM(X)
\end{align*}
the cohomology of the derived functors.

\subsubsection{}

A mixed Hodge module is called pure of weight $r$ if $\Gr_i^W\!(\rat_W(\mathcal{F}))=0$ for $i\neq r$.  In other words, a mixed Hodge module is pure of weight $r$ if the only jump in its weight filtration is in the $r$th place.  An object $\mathcal{F}\in\Ob(\Db(\MHM(X)))$ is called \textit{pure} if each $\Ho^r(\CF)$ is pure of weight $r$.  For example, if $X$ is a smooth projective variety, $\HO(X,\BQ)$ is pure.  Note that a complex $\mathcal{F}\in\Ob(\Db(\MHM(X)))$ is pure if and only if $\mathcal{F}\otimes\LLL^a$ is pure for every $a\in\BZ$.

\smallbreak

\subsubsection{}
If $X$ is a variety, the constant constructible sheaf $\BQ_X$ lifts to a complex $\ul{\BQ}_X\coloneqq a^*\ul{\BQ}_{\pt}\in\Db(\MHM(X))$.  Here we have denoted by $a\colon X\rightarrow \pt$ the structure morphism, and by $\ul{\BQ}_{\pt}$ the pure weight zero Hodge structure on $\BQ$ with $\Gr_F^{\bullet}(\ul{\BQ}_{\pt}\otimes_{\BQ}\BC)$ concentrated in degree zero.  Since $X$ is finite type, the constant sheaf $\BQ_X$ has bounded amplitude in the derived category of perverse sheaves.  Recall that $\ul{\BQ}_X$ generally has cohomology in multiple degrees if $X$ is not smooth (since the constant sheaf $\BQ_X$ is not itself a shift of a perverse sheaf if $X$ is sufficiently singular).  Similarly, if $X$ is smooth, $\ul{\BQ}_X$ is pure, but if $X$ is singular it may fail to be.  
\smallbreak
Connecting with \S \ref{sec:mixed_hodge_structures}, it was shown in \cite{Sai00} that the mixed Hodge structures $a_*a^*\underline{\BQ}_{\pt}$ and $a_!a^*\underline{\BQ}_{\pt}$ are naturally isomorphic to the mixed Hodge structures defined by Deligne on $\HO(X,\BQ)$ and $\HO_c(X,\BQ)$ respectively.
\smallbreak

If $\chi$ is a locally constant function on a scheme $X$, we denote by $\underline{\BQ}_X\otimes\LLL^{\chi}$ the complex of mixed Hodge modules on $X$ which on a connected component $Z\in\pi_0(X)$ is equal to $\underline{\BQ}_{Z}\otimes\LLL^{\chi(Z)}$.

\subsubsection{}
If $\CL$ is a simple pure polarisable variation of Hodge structure on a locally closed smooth connected subvariety $Z\subset X$, up to isomorphism there is a unique simple mixed Hodge module $\ICS_{\overline{Z}}(\CL)$ on $X$, satisfying $\ICS_{\overline{Z}}(\CL)\lvert_Z\cong \CL$.  The mixed Hodge module $\ICS_{\overline{Z}}(\CL)$ is supported on $\overline{Z}$, the closure of $Z$.  Assuming that $Z$ is connected and even-dimensional, we abbreviate
\[
\ICSn_{\overline{Z}}\coloneqq \ICS_{\overline{Z}}\left(\underline{\BQ}_{Z}\otimes\LLL^{-\dim(Z)/2}\right).
\]
This is a simple pure weight zero mixed Hodge module.  With the shifting conventions above, the intersection cohomology complex $\ICSn_{\overline{Z}}$ is Verdier self-dual.  We define the normalised intersection cohomology
\[
\ICA(\overline{Z})\coloneqq\HO(\overline{Z},\ICSn_{\overline{Z}}).
\]

\subsubsection{The decomposition theorem}

The following is a fundamental result of Saito, and is a huge generalisation of the fact that singular cohomology groups of smooth complex varieties carry pure Hodge structures:
\begin{theorem}\cite{Sai90}\label{thm:Saito_theorem}
Let $p\colon X\rightarrow Y$ be a projective morphism of varieties.  Then if $\CF\in\Ob(\Db(\MHM(X)))$ is pure, so is $p_!\CF$.  Moreover, if $\CF$ is pure, there are isomorphisms in $\Db(\MHM(X))$ and $\MHM(X)$ respectively
\[
\CF\cong \bigoplus_{i\in \BZ}\Ho^i(\CF)[-i];\quad\quad
\Ho^n(\CF)\cong \bigoplus_{j\in J_n}\ICS_{Z_j}(\CL_j)
\]
where $Z_j\subset X$ are locally closed subvarieties indexed by sets $J_n$ depending on $p$ and $\CF$, and $\CL_j$ are simple pure weight $n$ polarisable variations of Hodge structure on them.
\end{theorem}

\subsubsection{Purity is local}
We will use the \'{e}tale case of the following lemma, which follows more or less from definitions:
\begin{lemma}
\label{loc_pur_lem}
Purity is local in the Zariski, analytic, smooth and \'etale topologies.  I.e. an object $\CF\in\Ob(\Db(\MHM(X)))$ is pure if and only if $\Ho^i\!q^*(\mathcal{F})$ is pure of weight $i$ for every $i\in\BZ$, for $q\colon U\rightarrow X$ a cover in one of these topologies.
\end{lemma}
In order to make sense of the lemma, we have to say something about the theory of analytic mixed Hodge modules.  If $X$ is an analytic variety, Saito constructs the category $\MHM(X)$ of analytic mixed Hodge modules on $X$ in \cite{Sai90}.  The six functor formalism works the same way, but is not defined at the level of derived categories.  The part of the theory that we will use is that if $q^*\colon U\rightarrow X$ is a smooth morphism of analytic varieties of relative dimension $d$, there is a pullback morphism $\Ho^{d}\!q^*\colon \MHM(X)\rightarrow \MHM(U)$, see \cite[Sec.2.18 \& Prop.2.19]{Sai90} for details.
\begin{proof}[Proof of Lemma \ref{loc_pur_lem}]
Since perverse sheaves form a stack for the \'etale, analytic and Zariski sites of $X$, and purity is the statement that for $i\neq j$ the perverse sheaves $\Gr^W_i\!(\rat_W(\Ho^j(\CF)))$ vanish, for these three topologies the statement is obvious.  For the smooth topology, we use that if $q$ is a smooth morphism of relative dimension $d$, analytically locally we can write $\Ho^d\!q^*\Ho^i(\CF)=\Ho^i(\CF)\boxtimes \ul{\BQ}_{F'}[d]$ where $F'\subset F$ is an analytic open subscheme of the fibre of $q(x)$.  Now $\ul{\BQ}_{F'}[d]$ is an analytic mixed Hodge module, pure of weight $d$, from which it follows that $\Ho^{d+i}(q^*\CF)$ is analytically locally pure of weight $d+i$ at $x$ if and only if $\Ho^i(\CF)$ is analytically locally pure of weight $i$ at $q(x)$.  Then we conclude using the statement for the analytic topology.
\end{proof}

\subsection{Direct images from stacks}
\label{stack_di}
Since the compactly supported cohomology of a stack is generally unbounded below, it will be useful to consider unbounded (below) complexes of mixed Hodge modules, without significantly tinkering with the foundations of \cite{Sai90}.  This we do as follows.  Let $\Dub^{\mathrm{b},\geq n}(\MHM(X))$ denote the full subcategory of $\Db(\MHM(X))$ containing those mixed Hodge modules $\mathcal{F}$ on $X$ satisfying the condition that $\Ho^i(\mathcal{F})=0$ for $i<n$.  For $n\geq m$ there is a truncation functor 
\[
\tau_{m\rightarrow n}\coloneqq \bm{\tau}^{\geq n}\colon \Dub^{\mathrm{b},\geq m}(\MHM(X))\rightarrow \Dub^{\mathrm{b},\geq n}(\MHM(X))
\]
and by definition there is an identity $\tau_{m\rightarrow n}\circ \tau_{l\rightarrow m}=\tau_{l\rightarrow n}$.  We define $\Dblf(\MHM(X))$ to be the limit category.  Precisely, objects of $\Dblf(\MHM(X))$ are given by $\mathbb{Z}$-tuples of objects $\mathcal{F}_n\in\Ob(\Db(\MHM(X)))$ satisfying $\bm{\tau}^{\leq n-1}(\mathcal{F})=0$, along with isomorphisms for $n\geq m$
\[
\eta_{m,n}\colon \tau_{m\rightarrow n}\mathcal{F}_m\rightarrow \mathcal{F}_n
\]
such that the diagram
\[
\xymatrix{
\tau_{l\rightarrow n}\mathcal{F}_l\ar[d]^=\ar[rr]^{\eta_{l,n}}&&\mathcal{F}_n\\
\tau_{m\rightarrow n}\tau_{l\rightarrow m}\mathcal{F}_l\ar[rr]^{\tau_{m\rightarrow n}\eta_{l,m}}&&\tau_{m\rightarrow n}\mathcal{F}_m\ar[u]^{\eta_{m,n}}
}
\]
commutes.  This is the same thing as providing the objects $\mathcal{F}_n$ for $n\in\BZ$, along with a set of isomorphisms $\eta_{n,n+1}$.  Given an object $\CF\in\Ob(\Dblf(\MHM(X)))$ we define $\Ho^n(\CF)=\Ho^n(\CF_m)$ for any $m<n$.  It follows from the definitions that $\Ho^n(\CF)$ is independent of $m$.  
\smallbreak
We say that an object $\mathcal{F}=\{\mathcal{F}_n,\eta_{n,n+1}\}_{n\in\mathbb{Z}}\in\Ob(\Dblf(\MHM(X)))$ is \textit{pure} if each $\mathcal{F}_n$ is.  Equivalently, we ask that the cohomology mixed Hodge module $\Ho^n(\mathcal{F})$ be pure of weight $n$ for every $n$.  We define $\Dulf(\MHM(X))$ to be the colimit category of the dual construction using the truncation functors $\bm{\tau}^{\leq n}$, so that for $\CF\in\Ob(\Dblf(\MHM(X)))$, the Verdier dual $\BD\CF$ is an object of $\Ob(\Dulf(\MHM(X)))$ and vice versa.

\subsubsection{}
There does not seem to be a general treatment of the six functor formalism for derived categories of mixed Hodge modules on general stacks.  We thus restrict our attention to a certain class of stacks for which a theory of equivariant mixed Hodge modules (as developed by Achar \cite{EMHM}, following \cite{BL06}) is sufficient for our purposes:
\begin{definition}
\label{eqs_def}
Let $\FX$ be an algebraic stack.  We say that $\FX$ is \textit{exhausted by global quotient stacks}, if there is an infinite sequence of open substacks 
\[
\FX_1\subset \FX_2\subset\ldots \FX
\]
such that
\begin{enumerate}
    \item Each $\FX_n$ is a stack-theoretic quotient of a finite type separated scheme by an affine algebraic group.
    \item $\FX$ is the union of the $\FX_n$.
    \item 
    For every $n$, and every smooth morphism $Y\rightarrow \FX$ from a finite type variety, the codimension of $Y\setminus U_n$ inside $Y$ is greater than $n$, where we define $U_n=\FX_n\times_{\FX} Y$.
\end{enumerate}
\end{definition}
\smallbreak

In particular, any stack which is exhausted by global quotient stacks is locally finite type.  A special case that will suffice for most purposes occurs when $\FX$ is itself a global quotient of a quasi-projective variety.  Then we may set $\FX_n=\FX$ for all $n$ in the definition above.

\subsubsection{}
\label{heart_di}
Let $p\colon\FX\rightarrow \Msp$ be a morphism to a finite type separated complex scheme, and assume that $\FX$ is exhausted by global quotient stacks $\FX_N$.  We define the object $p_!\ul{\BQ}_{\FX}\in \Dblf(\MHM(\Msp))$, as in \cite[Sec. 2.2]{BenSven}; we briefly recall the construction.  For each $n\in\mathbb{Z}$ we define an element $\mathcal{F}_n\in\Ob(\Dub^{b,\geq n}(\MHM(\Msp)))$ as follows.  First pick $N\gg 0$, and consider the stack $\FX_N$.  By supposition, we can write $\FX_N=Y/G$ for $Y$ a separated scheme and $G$ an affine algebraic group.  We embed $G\subset \Gl_r(\BC)$ for some $r$, pick $M\gg 0$, and set $A=\Hom(\BC^M,\BC^r)$.  We consider $A$ as a $G$-equivariant variety via the natural $\Gl_r(\BC)$-action.  There is an open subvariety $A'\subset A$ containing those $f\in A$ that are surjective.  Set $B_N=(Y\times A')/G$.  This is a scheme by \cite[Prop.~23]{EG98}.  There is a natural map
\begin{equation}
\label{ApproxMap}
q\colon B_N\rightarrow \Msp
\end{equation}
defined by composing $p$ with the composition of morphisms of stacks
\[
B_N\hookrightarrow (Y\times A)/G\xrightarrow{\pi}\FX_N
\]
where $\pi$ is the natural projection.  Then we define
\[
\CF_n=\bm{\tau}^{\geq n}(q_!\ul{\BQ}_{B_N}\otimes\LLL^{-rM}).
\]
\smallbreak
It is straightforward to show that although we have made a number of choices in the definition, if $\mathcal{F}'_n$ is defined via a different set of choices, there is a natural isomorphism $\CF_n\cong \CF'_n$, see \cite[Sec.~2.2]{BenSven}.  In the same way, one shows that there are natural isomorphisms $\eta_{n,n+1}\colon \tau_{n\rightarrow n+1}\CF_{n}\rightarrow \CF_{n+1}$.
\smallbreak 
Applying the above construction to the structure morphism of the stack $\FX$, we define
\[
\HO^n_c(\FX,\mathbb{Q})=\HO_c^n(B_N,\ul{\BQ}_{B_N}\otimes \LLL^{-rM})
\]
for $r,M,N\gg 0$ depending on $n$ as above.  We define 
\[
\HO^{\BM}_{n}(\FX,\mathbb{Q})=(\HO_c^n(\FX,\mathbb{Q}))^{\vee}
\]
where the dual is taken in the category of mixed Hodge structures.
\subsubsection{}The next proposition underpins the strategy outlined in \S \ref{glob_intro} for proving (global) purity of moduli stacks.
\begin{proposition}
\label{int_pur}
Let $p\colon\FX\rightarrow \CM$ be a morphism from a stack to a projective scheme, and assume that $\FX$ is exhausted by global quotient stacks.  If $p_!\ul{\BQ}_{\FX}$ is pure, then $\HO_c(\FX,\mathbb{Q})$ and $\HO^{\BM}(\FX,\mathbb{Q})$ are pure.
\end{proposition}
\begin{proof}
We give the proof for the case of compactly supported cohomology.  Purity for Borel--Moore homology follows by duality.
\smallbreak
Set $d=\dim(\CM)$.  Fix a number $n\in\mathbb{Z}$, and $r,M,N\gg 0$ as above depending on $n$.  Then 
\begin{align*}
\HO_c^n(\FX,\BQ )=&\HO_c^{n}( B_{N},\ul{\BQ}_{B_{N}}\otimes\LLL^{-rM})\\
\cong  & \HO_c^n(\CM,q_!\ul{\BQ}_{B_{N}}\otimes\LLL^{-rM})\\
\cong & \HO_c^n(\CM,\bm{\tau}^{\geq n-d}(q_!\ul{\BQ}_{B_{N}}\otimes\LLL^{-rM})).
\end{align*}
The final isomorphism uses Remark \ref{CArem} below.  The right hand side only depends on the choice of sufficiently large $r,M,N$ up to canonical isomorphism.  By supposition, the complex $\bm{\tau}^{\geq n-d}(q_!\ul{\BQ}_{B_{N}}\otimes\LLL^{-rM})$ is pure, and $\CM$ is projective.  So by Saito's theorem, recalled as Theorem~\ref{thm:Saito_theorem}, $\HO_c^n(\FX,\BQ)$ is pure of weight $n$.
\end{proof}
\begin{remark}
\label{CArem}
Set $d=\dim(\CM)$, and denote by $a\colon\CM\rightarrow \pt$ the structure morphism.  The functor $a_!$ sends mixed Hodge modules to complexes with cohomological amplitude $[-d,d]$.  See e.g. \cite[Prop.10.3.3]{KS13} for the upper bound.  The lower bound follows more or less from definitions: perverse sheaves lie in cohomological amplitude $[-d,0]$ with respect to the constructible t-structure, and $a_!$ is left exact with respect to this t-structure.
\end{remark}
Via the same proof as Proposition \ref{int_pur}, we have the more general result:
\begin{proposition}
\label{intr_pur}
Let $p\colon\FX\rightarrow \CM$ be a morphism from a stack to a finite type separated scheme, and let $\pi\colon \CM\to\CN$ be a projective morphism of schemes.  Assume that $\FX$ is exhausted by global quotient stacks, so that the complexes of mixed Hodge modules $p_!\underline{\BQ}_{\FX}$ and $(\pi p)_!\underline{\BQ}_{\FX}$ are defined in the derived category.  If $p_!\ul{\BQ}_{\FX}$ is pure, then $(\pi p)_!\ul{\BQ}_{\FX}$ is also pure.
\end{proposition}

\subsection{Perverse filtrations}
We will use local purity and the following proposition to construct perverse filtrations on Borel--Moore homology of stacks of objects in 2CY categories.

\begin{proposition}
\label{int_filt}
Let $p\colon \FX\rightarrow \CM$ be a morphism from a stack to a finite type separated scheme, and assume that $\FX$ is exhausted by global quotient stacks.  If $p_!\ul{\BQ}_{\FX}$ is pure, then for every $n\in\BZ$ the morphism $p$ induces an ascending perverse filtration on $\HO^{\BM}_{n}(\FX,\BQ)$ by mixed Hodge structures.
\end{proposition}

\begin{proof}
Since $\HO^{\BM}(\FX,\BQ)$ is defined to be the graded dual of $\HO_c(\FX,\BQ)$ in the category of mixed Hodge structures, the statement is equivalent to the claim that for each $n\in\mathbb{Z}$ there is an ascending filtration $0\subset \ldots \subset L^n_{s-1}\subset L^n_{s}\subset\ldots $ of mixed Hodge structures stabilising to $\HO_c^n(\FX,\BQ)$.  Precisely, the required inclusion $P_{s}^n\hookrightarrow \HO_{n}^{\BM}(\FX,\BQ)$ is defined to be the dual of the cokernel of the inclusion $L^n_{-s}\hookrightarrow \HO_c^n(\FX,\BQ)$.
\smallbreak
We continue to denote $d=\dim(\CM)$, and we set
\begin{align}
L^n_s\coloneqq &\HO_c^n\!\left(\CM,\bm{\tau}^{\leq s}(q_!\ul{\BQ}_{B_N}\otimes \LLL^{-rM})\right)
\\\cong &\HO_c^n\!\left(\CM,\bm{\tau}^{\geq n-d}\bm{\tau}^{\leq s}(q_!\ul{\BQ}_{B_N}\otimes \LLL^{-rM})\right).\label{altL}
\end{align}
Isomorphism \eqref{altL} again follows by Remark \ref{CArem}.  As in the proof of Proposition \ref{int_pur}, since we have fixed $n$, the mixed Hodge structure $L^n_s$ only depends on the choice of sufficiently large $M,N,r$ up to canonical isomorphism, so that $L^n_s$ is at least well defined.  By \eqref{altL} we deduce that $L^n_{s}=0$ for $s<n-d$.  By Remark \ref{CArem} again, if $s\geq n+d$ then, writing $\CG=q_!\ul{\BQ}_{B_N}\otimes \LLL^{-rM}$, each of the natural morphisms
\[
\HO_c^n\!\left(\CM,\bm{\tau}^{\geq n-d}\bm{\tau}^{\leq s}\CG\right)\rightarrow \HO_c^n\!\left(\CM,\bm{\tau}^{\geq n-d}\CG\right)\leftarrow \HO_c^n\!\left(\CM,\CG\right)
\]
is an isomorphism.  So for $s\geq n+d$ we have natural isomorphisms
\begin{align*}
L^n_s\cong&\HO_c^n\!\left(\CM,q_!\ul{\BQ}_{B_N}\otimes \LLL^{-rM}\right)
\cong \HO^n_c\!\left(B_N,\BQ\otimes \LLL^{-rM}\right)
\cong \HO^n_c\!\left(\FX,\BQ\right) 
\end{align*}
as required.  For sufficiently large $M,N,r$ we have a decomposition
\[
\bm{\tau}^{\geq n-d}\bm{\tau}^{\leq n+d}(q_!\ul{\BQ}_{B_N}\otimes \LLL^{-rM})\cong \bigoplus_{n-d\leq i\leq n+d}\Ho^{i}\!\left(q_!\ul{\BQ}_{B_N}\otimes \LLL^{-rM}\right)[-i]
\]
by purity of $p_!\ul{\BQ}_{\FX}$.  It therefore follows that for $s\leq n+d$ the natural transformation 
\begin{equation}
\label{bee_def}
\beta\colon\bm{\tau}^{\geq n-d}\bm{\tau}^{\leq s}(q_!\ul{\BQ}_{B_N}\otimes \LLL^{-rM})\rightarrow \bm{\tau}^{\geq n-d}\bm{\tau}^{\leq n+d}(q_!\ul{\BQ}_{B_N}\otimes \LLL^{-rM})
\end{equation}
has a left inverse, which we will denote $\alpha$.  Denote by $a\colon \Msp\rightarrow \pt$ the structure morphism.  Then $\Ho^n(a_!\alpha)$ is a left inverse to the morphism 
\begin{equation}
\label{PIdef}
\Ho^n(a_!\beta)\colon L^n_s\rightarrow \HO^n_c(\FX,\BQ).
\end{equation}
In particular we deduce that \eqref{PIdef} is an \textit{inclusion} of mixed Hodge structures, factoring through the similarly defined inclusions $L^n_{s'}\rightarrow \HO^n_c(\FX,\BQ)$ for $s'>s$.
\end{proof}
\subsubsection{}
\label{FPR}
Let $p\colon \FX\rightarrow \CM$ be as in the Proposition \ref{int_filt}.  Let $r\colon\CM\rightarrow \CN$ be a finite morphism of schemes, so that $r$ is in particular projective.  Purity of $(rp)_!\underline{\BQ}_{\FX}$ then follows from purity of $p_!\underline{\BQ}_{\FX}$ and Proposition \ref{intr_pur}.  Moreover, $r$ is exact with respect to the perverse t-structure on complexes of constructible sheaves, and the natural t-structure on complexes of mixed Hodge modules.  It follows that the perverse filtrations on $\HO^{\BM}_{n}(\FX,\BQ)$ with respect to $p$ and $rp$ coincide.
\subsubsection{}
The filtration in Proposition \ref{int_filt} comes from the (non-canonical) decomposition of 
\[
\bm{\tau}^{\geq n-d}\bm{\tau}^{\leq s}(q_!\ul{\BQ}_{B_N}\otimes\LLL^{-rM})
\]
arising from Saito's version of the decomposition theorem of Beilinson--Bernstein--Deligne--Gabber.  I.e. the first part of Theorem \ref{thm:Saito_theorem} tells us that there is an isomorphism in $\Db(\MHM(\CM))$
\[
\bm{\tau}^{\geq n-d}\bm{\tau}^{\leq s}(q_!\ul{\BQ}_{B_N}\otimes\LLL^{-rM})\cong \bigoplus_{n-d\leq i\leq s} \Ho^i\!\left(q_!\ul{\BQ}_{B_N}\otimes\LLL^{-rM}\right)[-i].
\]
Although the decomposition is not canonically defined, the filtration by perverse cohomological degree is, so we can write
\[
L^n_s=\bigoplus_{n-d\leq i\leq s}\HO_c^n\!\left(\CM,\Ho^i(q_!\ul{\BQ}_{B_N}\otimes \LLL^{-rM})[-i]\right),
\]
and taking $M,N,r\gg 0$
\begin{align}
\nonumber \HO_c^n(\FX,\BQ)\cong &\bigoplus_{n-d\leq i\leq n+d}\HO_c^n\!\left(\CM,\Ho^i(q_!\ul{\BQ}_{B_N}\otimes \LLL^{-rM})[-i]\right)\\
\cong& \bigoplus_{n-d\leq i\leq n+d}\HO_c^n\!\left(\CM,\Ho^i(p_!\ul{\BQ}_{\FX})[-i]\right).\label{callback}
\end{align}

\subsubsection{Serre subcategories}
\label{Serre_sec}
In applications, we are often interested in Serre subcategories of 2CY categories.  The geometric situation will be: we have a morphism $p\colon \FX\rightarrow \CM$ from a stack to a variety as in the setup of Proposition \ref{int_filt}, and a locally closed subscheme $\iota\colon\CS\hookrightarrow \CM$.  Setting $\FS\coloneqq \CS\times_{\CM}\FX$, we would like to understand $\HO^{\BM}(\FS,\BQ)$, or equivalently $\HO_c(\FS,\BQ)$.
\smallbreak
Note that even if $p_!\ul{\BQ}_{\FX}$ is pure, there is no guarantee that $\iota^*p_!\ul{\BQ}_{\FX}$ will be.  So, for instance, we cannot infer purity of $\HO_c(\FS,\BQ)$ from purity of $p_!\ul{\BQ}_{\FX}$ and projectivity of $\CS$.  This becomes a point of consideration when proving e.g. purity of the (semistable) global nilpotent cone -- see Proposition \ref{GNC_prop}.  Nevertheless, we do have the following generalisation of Proposition \ref{int_filt}:
\begin{proposition}
\label{PIS_prop}
Let the Cartesian diagram
\[
\xymatrix{
\FS\ar[d]^{r}\ar[r]&\FX\ar[d]^p\\
\CS\ar[r]^{\iota}&\CM
}
\]
be as in the preceding discussion, i.e. $p$ is a morphism from a locally finite type stack which is exhausted by global quotient stacks, to a finite type separated scheme, and $\iota$ is an inclusion of a locally closed subscheme.  Assume that $p_!\ul{\BQ}_{\FX}$ is a pure complex of mixed Hodge modules.  Then for every $n\in\mathbb{Z}$ the morphism $p$ induces an ascending filtration on $\HO_n^{\BM}(\FS,\BQ)$ by mixed Hodge structures.
\end{proposition}
\begin{proof}
As in the proof of Proposition \ref{int_filt} it is sufficient to construct an ascending filtration
\[
0\subset \ldots \subset L^{\CS,n}_{s-1}\subset L^{\CS,n}_{s}\subset\ldots 
\]
of mixed Hodge structures stabilising to $\HO_c^n(\FS,\BQ)$.  The cohomological amplitude of $(\CS\rightarrow \pt)_!\iota^*\CF$, for $\CF$ a mixed Hodge module, is still contained in $[-d,d]$, where $d=\dim(\CM)$ as in Proposition \ref{int_filt}.  We thus define
\[
L^{\CS,n}_{s}=\HO_c^n\!\left(\CS,\iota^*\bm{\tau}^{\geq n-d}\bm{\tau}^{\leq s}(q_!\ul{\BQ}_{B_N}\otimes \LLL^{-rM})\right).
\]
We define $\beta$ as in \eqref{bee_def}, and continue to denote by $\alpha$ a left inverse to this natural morphism.  Then $\Ho^n((\CS\rightarrow \pt)_!\iota^*\alpha)$ is a left inverse to the morphism
\begin{equation}
\label{PIS_def}
\Ho^n(\CS\rightarrow\pt)_!\iota^*\beta\colon L^{\CS,n}_s\rightarrow \HO^n_c(\FS,\BQ)
\end{equation}
so that as in the proof of Proposition \ref{int_filt} the morphism \eqref{PIS_def} is an inclusion of mixed Hodge structures factoring through the similarly defined inclusion $L^{\CS,n}_{s'}\hookrightarrow \HO^n_c(\FS,\BQ)$ for $s'>s$.
\end{proof}
\begin{definition}
\label{ambient_per}
We call the ascending filtrations constructed in Propositions \ref{int_filt} and \ref{PIS_prop} the \textit{perverse filtration with respect to }$p$.
\end{definition}
We repeat here the standard warning, with notation as in Proposition \ref{PIS_prop}: if it so happens that the complexes $p_!\ul{\BQ}_{\FX}$ and $r_!\ul{\BQ}_{\FS}$ are both pure, the perverse filtrations on $\HO^{\BM}(\FS,\BQ)$ with respect to $r$ and with respect to $p$ can be quite different.  An easy example of how they can differ is provided by contemplating the morphisms $\iota\colon \{0\}\hookrightarrow \BA^1$ and $p=\id_{\BA^1}$.

\bigbreak
\section{Quivers and preprojective algebras}\label{sec:quivers}
\subsection{Quiver representations}\label{sec:quiver_reps}
Until Theorem \ref{pp3_thm}, $k$ is an arbitrary fixed field throughout \S \ref{sec:quivers}.  
\smallbreak
A quiver is a directed graph, i.e. a pair of sets $Q_1$ and $Q_0$ denoting the arrows and the vertices, respectively, along with a pair of maps  $s,t\colon Q_1\rightarrow Q_0$ taking an arrow to its source and target, respectively.  We will always assume that $Q_1$ and $Q_0$ are finite.  We denote by $k Q$ the free path algebra of $Q$ over $k$, with $k$-basis given by paths in $Q$, including ``lazy paths'' of length zero at each vertex, labelled $e_i$ for $i\in Q_0$.  Multiplication is defined by concatenation of paths.  A left $k Q$-module $\rho$ has a dimension vector $\udim(\rho)\in\BN^{Q_0}$ defined by 
\[
\udim(\rho)_i=\dim(e_i\cdot \rho).
\]

We define by 
\[
\chi_Q(\mathbf{d},\mathbf{e})=\sum_{i\in Q_0}\mathbf{d}_i\mathbf{e}_i-\sum_{a\in Q_1}\mathbf{d}_{s(a)}\mathbf{e}_{t(a)}
\]
the Euler pairing on dimension vectors for $Q$.
\subsubsection{}

We denote by $\FM_{\dd}(Q)$ the moduli stack of $\dd$-dimensional left $k Q$-modules.  There is a natural equivalence of stacks $\FM_{\dd}(Q)\cong \BA_{\dd}(Q)/\Gl_{\dd}$, where
\[
\BA_{\dd}(Q)\coloneqq\prod_{a\in Q_1}\Hom(k^{\dd_{s(a)}},k^{\dd_{t(a)}})
\]
and $\Gl_{\dd}\coloneqq\prod_{i\in Q_0}\Gl_{\dd_i}(k)$ acts by change of basis.

\subsection{Preprojective algebras}
\label{preproj_sec}
Given a quiver $Q$ we denote by $\ol{Q}$ the double of $Q$, i.e. the quiver having the same vertices as $Q$, and with $\ol{Q}_1\coloneqq Q_1\coprod Q_1^{\opp}$.  Here $Q_1^{\opp}$ contains an arrow $a^*$ for every arrow $a\in Q_1$, with the orientation reversed.  We define the \textit{preprojective algebra} 
\[
\Pi_Q\coloneqq k \ol{Q}/\langle \sum_{a\in Q_1}[a,a^*]\rangle.
\]
We define $\mathfrak{gl}_{\mathbf{d}}\coloneqq \prod_{i\in Q_0}\mathfrak{gl}_{\mathbf{d}_i}(k)$, the Lie algebra of the Lie group $\Gl_{\dd}$, and we define the (co)moment map
\begin{equation}
\label{comoment}
\mu_{\mathbf{d}}\colon\BA_{\mathbf{d}}(\overline{Q})\rightarrow \mathfrak{gl}_{\mathbf{d}};\quad\quad
(\rho(a),\rho(a^*))_{a\in Q_1}\mapsto \sum_{a\in Q_1}[\rho(a),\rho(a^*)].
\end{equation}
There is an equivalence of stacks $\FM_{\mathbf{d}}(\Pi_Q)\cong \mu_{\mathbf{d}}^{-1}(0)/\Gl_{\mathbf{d}}$, where on the left hand side we have the moduli stack of $\mathbf{d}$-dimensional $\Pi_Q$-modules.
\subsubsection{Derived preprojective algebras}
\label{dppa}
We define $R=\bigoplus_{i\in Q_0}k$.  Let $E$ be an $R$-bimodule, for which we identify arrows $a\in Q_1$ from $i$ to $j$ with a basis of $e_j\cdotb E\cdotb e_i$.  We denote by $E^{\vee}=\Hom_{\Mod^R_R}(E,R)$ the bimodule dual.  The diagonal $k Q$-bimodule $kQ$ has a 2-term resolution by projective bimodules
\[
kQ\otimes_R E\otimes_R kQ\xrightarrow{\phi} kQ\otimes_R kQ
\]
so that the complex
\[
kQ^!\coloneqq \left(kQ\otimes_R kQ\xrightarrow{\Hom(-,kQ\otimes kQ)(\phi)} kQ\otimes_R E^{\vee}\otimes_R kQ\right)
\]
calculates $\RHom_{\Mod^{kQ}_{kQ}}(kQ,kQ\otimes kQ)$.  
\smallbreak
The \textit{2-Calabi--Yau completion} \cite{Ke11} $\mathscr{G}_2(kQ)$ is defined to be the tensor algebra of the bimodule $kQ^![1]$ over $kQ$, which we may describe explicitly.  First take the doubled quiver $\overline{Q}$.  Next, for every vertex $i$ we add a loop $u_i$ based at $i$, to form a quiver $Q'$.  We set the degrees of the arrows $a$, $a^*$ and $u_i$ to be $0,0,-1$ respectively.  Then 
\begin{equation}
    \label{dppa_def}
\mathscr{G}_2(kQ)=(kQ',d)
\end{equation}
where $d$ is the unique $k$-linear differential satisfying the Leibniz rule and sending $u_i$ to $e_i\sum_{a\in Q_1}[a,a^*]e_i$.  It is clear that $\HO^0(\mathscr{G}_2)$ is isomorphic to the preprojective algebra of $Q$.  If $Q$ is not of finite (ADE) type, the natural morphism
\[
\mathscr{G}_2(kQ)\rightarrow \HO^0(\mathscr{G}_2(kQ))
\]
is a quasi-isomorphism (see \cite{Kel08}).  Moreover, since $\mathscr{G}_2(kQ)$ is concentrated in nonpositive degrees, there is an inclusion of categories $\fdmod^{\Pi_Q}\rightarrow \Db(\fdmod^{\mathscr{G}_2(kQ)})$, induced by the morphism
\begin{equation}
    \label{der_pp}
\mathscr{G}_2(kQ)\rightarrow \Pi_Q.
\end{equation}
This is the inclusion of the heart of the natural t-structure on the target.  We call $\mathscr{G}_2(kQ)$ the \textit{derived preprojective algebra} for $Q$.

\subsubsection{}We denote by $\CM_{\dd}(\Pi_Q)$ the coarse moduli space of $\dd$-dimensional $\Pi_Q$-modules.  Explicitly,
\[
\CM_{\dd}(\Pi_Q)=\Spec(\Gamma(\CO_{\mu_{\mathbf{d}}^{-1}(0)})^{\Gl_{\mathbf{d}}})
\]
and so $\CM_{\dd}(\Pi_Q)$ is the affinization of $\FM_{\dd}(\Pi_Q)$.  The $k$-points of this scheme are in bijection with isomorphism classes of semisimple $\Pi_Q$-modules.  We denote by $0_{\dd}\in\CM_{\dd}(\Pi_Q)$ the $k$-point representing the unique semisimple nilpotent $\dd$-dimensional $\Pi_Q$-module.  Explicitly, this is the $\mathbf{d}$-dimensional module for which all of the arrows of $\ol{Q}$ act via zero morphisms.  
\smallbreak
We denote by
\[
\JH_{\dd}\colon \FM_{\dd}(\Pi_Q)\rightarrow \CM_{\dd}(\Pi_Q)
\]
the affinization map.  At the level of geometric points over $K\supset k$, it is the morphism taking a $\Pi_Q\otimes_k K$-module to its semisimplification.
\subsubsection{}

Given a \textit{King stability condition} $\zeta\in\BQ^{Q_0}$, we define the slope of a nonzero $\Pi_Q$-module $\rho$ by 
\[
\psi^{\zeta}(\rho)\coloneqq\frac{\udim(\rho)\cdot \zeta}{\dim(\rho)}.
\]
A $\Pi_Q$-module $\rho$ is $\zeta$-semistable if for all proper nonzero submodules $\rho'\subset \rho$ there is an inequality $\psi^{\zeta}(\rho')\leq \psi^{\zeta}(\rho)$.  The $\Pi_Q$-module $\rho$ is $\zeta$-stable if the inequality is strict.  We denote by $\FM^{\zeta\sstab}_{\dd}(\Pi_Q)$ the moduli stack of $\zeta$-semistable $\Pi_Q$-modules, and by $\CM_{\dd}^{\zeta\sstab}(\Pi_Q)$ the coarse moduli space as defined via GIT by King~\cite{MR1315461}.  By the universal property of the coarse moduli space there is a morphism
\[
\JH^{\zeta}_{\dd\colon}\FM_{\dd}^{\zeta\sstab}(\Pi_Q)\rightarrow \CM^{\zeta\sstab}_{\dd}(\Pi_Q)
\]
which at the level of points, takes a $\zeta$-semistable $\Pi_Q$-module $\rho$ to the direct sum $\bigoplus_{i\in I}\rho_i^{\oplus m_i}$.  Here, $I$ is a set indexing the distinct isomorphism classes of $\zeta$-stable $\Pi_Q$-modules appearing as subquotients in the Jordan--H\"older filtration of $\rho$ (in the category of $\zeta$-semistable $\Pi_Q$-modules), and $m_i$ is the number of times that $\rho_i$ appears as a subquotient in such a filtration.

\sssct
In simple examples we can already begin to see the tricky nature of $\BQ_{\FM_{\mathbf{d}}(\Pi_Q)}$.

\begin{example}\label{qex}
We show that $\BQ_{\FM_{\dd}(\Pi_Q)}$ may fail to be a shifted perverse sheaf.  Let $Q$ be the quiver with two vertices, and one arrow between them.  We fix $\mathbf{d}=(1,2)$.  Since $r\colon \mu_{\mathbf{d}}^{-1}(0)\rightarrow \FM_{\mathbf{d}}(\Pi_Q)$ is smooth, it is sufficient to show that $\BQ_{\mu^{-1}_{\mathbf{d}}(0)}$ is not a shifted perverse sheaf.  We write
\[
\mu_{\mathbf{d}}^{-1}(0)=\{(a,b),(c,d)\in \BC^2\times \BC^2\;\lvert\; ac=ad=bc=bd=0\}.
\]
So $X=\mu_{\mathbf{d}}^{-1}(0)$ is a union of two planes identified at the origin $0$.  Write $p\colon \BA^2\coprod\BA^2\eqqcolon\tilde{X}\rightarrow X$ for the normalisation morphism.  Note that $p_*\BQ_{\tilde{X}}$ is a cohomologically shifted perverse sheaf, since $\BQ_{\tilde{X}}[2]$ is perverse and $p_*$ is perverse exact, since it is finite.  There is a short exact sequence of constructible sheaves
\[
0\rightarrow \BQ_X\rightarrow p_*\BQ_{\tilde{X}}\rightarrow \BQ_{0}\rightarrow 0
\]
From the long exact sequence in perverse cohomology we find 
\[
{}^{\mathfrak{p}}\!\Ho^2(\BQ_X)\cong p_*\BQ_{\tilde{X}}[2];\quad\quad
{}^{\mathfrak{p}}\!\Ho^1(\BQ_X)\cong \BQ_0.
\]
\end{example}
\begin{example}
\label{qe2}
Next we consider an example that shows that even when some shift of $\BQ_{\FM_{\dd}(\Pi_Q)}$ is a perverse sheaf, its natural lift\footnote{I.e. its lift to the category of mixed Hodge modules with descent data on a smooth atlas.} to a mixed Hodge module may not be pure.  Again, we can demonstrate this by considering $\mu^{-1}_{\mathbf{d}}(0)$, since $r$ is smooth.  
\smallbreak
Let $Q$ be the same quiver as in Example \ref{qex}, and set $\mathbf{d}=(1,1)$.  Then $Z=\mu^{-1}_{\mathbf{d}}(0)\subset \BA^2$ is the vanishing locus of the function $xy$.  Let $q\colon \tilde{Z}\rightarrow Z$ be the normalisation morphism from $\tilde{Z}=\BA^1\coprod \BA^1$.  Then there is a short exact sequence of mixed Hodge modules
\begin{equation}
\label{MHMses}
0\rightarrow\underline{\BQ}_0\rightarrow \underline{\BQ}_Z[1]\rightarrow q_*\underline{\BQ}_{\tilde{Z}}[1]\rightarrow 0
\end{equation}
so in particular $\BQ_Z[1]$ is indeed perverse, as it is the extension of two perverse sheaves.  We can read off the weight filtration from \eqref{MHMses}:
\[
W_{n}\underline{\BQ}_Z[1]=\begin{cases}0& \textrm{for }n<0\\
\underline{\BQ}_0 &\textrm{for }n=0\\
\underline{\BQ}_Z[1]&\textrm{for }n>0.\end{cases}
\]
\end{example}
\subsubsection{}

We state here a fundamental property of the morphism $\JH_{\dd}$ that we will use throughout the paper.
\begin{theorem}\cite{preproj3}
\label{pp3_thm}
Set $k=\BC$.  Let $Q$ be a finite quiver, and let $\dd\in\BN^{Q_0}$ be a dimension vector.  Then the complex of mixed Hodge modules
\[
\JH_!\ul{\BQ}_{\Mst_{\dd}(\Pi_Q)}\otimes \LLL^{\chi_Q(\mathbf{d},\mathbf{d})}\in\Ob(\Dblf(\MHM(\Msp_{\dd}(\Pi_Q))))
\]
is pure, and
\begin{equation}
\label{perv_bound}
\bm{\tau}^{\geq 1}\left(\JH_{\dd,!}\ul{\BQ}_{\FM_{\dd}(\Pi_Q)}\otimes \LLL^{\chi_Q(\mathbf{d},\mathbf{d})}\right)=0
\end{equation}
so that the induced perverse filtration on $\HO^{\BM}(\Mst_{\mathbf{d}}(\Pi_Q),\BQ\otimes \LLL^{\chi_Q(\mathbf{d},\mathbf{d})})$ begins in degree zero.
\end{theorem}
In Examples \ref{qex} and \ref{qe2} it is possible to prove this by picking a sensible stratification of $\Mst_{\mathbf{d}}(\Pi_Q)$ and using the resulting long exact sequences.  Outside of tame type this strategy breaks down, and the proof instead uses the cohomological integrality theorem in Donaldson--Thomas theory \cite{BenSven},  and a lemma on the supports of BPS sheaves for the category of $\Pi_Q[z]$-modules from \cite{2016arXiv160202110D}.   Using these results, along with dimensional reduction from the 3CY setting to the 2CY category of $\Pi_Q$-modules \cite{KR17, YZ20}, we can reduce the purity statement to the statement that the Borel--Moore homology of the stack of pairs $(\rho,f)$ of a $\dd$-dimensional $Q$ representation $\rho$ along with a nilpotent endomorphism of $\rho$ is pure.  This \textit{is} provable by stratifying the stack and checking purity of the Borel--Moore homology of each stratum by hand.  Full details can be found in~\cite{preproj3}.
%, we provide a sketch proof in order to make the paper somewhat self-contained, and to give a feel for what BPS cohomology is (as it will make a brief reappearance in \S \ref{}).
\begin{remark}
In fact the theorem in ibid. is proved more generally, for the morphism $\JH^{\zeta}$.  We can consider $\JH$ as the special case of $\JH^{\zeta}$ for the degenerate stability condition $\zeta=(1,\ldots,1)$.  We will show in \S \ref{proproj_redux} how to deduce the general case from the special case, as it is illustrative of the general methods of this paper.
\end{remark}
%\begin{proof}[Sketch proof of Theorem \ref{pp3_thm}]
%Consider the quiver $\tilde{Q}$ defined the same way as $Q'$, but with ungraded arrows.  On this quiver we consider the element
%\[
%W=(\sum_{i\in Q_0} u_i)(\sum_{a\in Q_1}[a,a^*]
%\]
%We consider the commutative diagram
%\[
%\xymatrix{
%\FM(\BC \tilde{Q})\ar[d]^{\tilde{\JH}}\ar[r]^{\pi}& \FM(\BC\overline{Q})\ar[d]^{\JH}\\
%\CM(\BC\tilde{Q})\ar[r]^q&\CM(\BC\overline{Q}).
%}
%\]
%By the dimensional reduction theorem, there is an isomorphism
%\[
%\JH_!\pi_!\phi_{\Tr(W)}\underline{\BQ}
%\]
%\end{proof}
\sssct

Note that $2\chi_Q(\mathbf{d},\mathbf{d})$ is equal to the Euler form for $\fdmod^{\mathscr{G}_2(kQ)}$ applied to any module of dimension vector $\mathbf{d}$.  Since the truncation functors commute with restriction along inverse image along \'etale morphisms, we will be able to use the bound \eqref{perv_bound} to bound (above) the cohomological amplitude of $p_!\ul{\BQ}_{\Mst}$ for $p\colon\Mst\rightarrow\Msp$ the morphism from the moduli stack to the moduli space of objects of each of the 2CY categories introduced in \S \ref{glob_intro}, via our local neighbourhood theorem (Theorem \ref{loc_str_thm}).

\section{2CY categories and formality}
\label{sec:2CY_local}
\subsection{dg categories}
 Throughout \S \ref{sec:2CY_local} we work over an arbitrary field $k$ of characteristic zero.  We briefly recall the definitions and salient features of ($k$-linear) dg categories, referring the reader to \cite{Kel06} for more details.  
\subsubsection{}
 A (small) dg category $\mathscr{C}$ is given by a set of objects $\Ob(\mathscr{C})$, as well as a cochain complex of morphisms $\mathscr{C}(i,j)\in\Ob(\Vect_{\dg})$, for each pair of objects $i,j\in\Ob(\mathscr{C})$.  I.e. for each pair of objects we are given a $\BZ$-graded $k$-vector space
 \[
 \mathscr{C}(i,j)\coloneqq \Hom_{\mathscr{C}}(i,j)=\bigoplus_{p\in\BZ}\Hom^p_{\mathscr{C}}(i,j)
 \]
along with a degree one morphism $d\colon\mathscr{C}(i,j)\rightarrow \mathscr{C}(i,j)$ satisfying $d^2=0$.  Composition is given by morphisms
\[
m_2\colon \Hom_{\mathscr{C}}(j,k)\otimes \Hom_{\mathscr{C}}(i,j)\rightarrow \Hom_{\mathscr{C}}(i,k)
\]
in the category of $k$-linear cochain complexes, and these compositions are required to be associative.  Succinctly, a dg category is a category enriched over $k$-linear cochain complexes.  A dg functor is a 1-morphism in the 2-category of categories enriched over $k$-linear cochain complexes.  As part of this definition, if $\mathscr{C}$ is a dg category, then each $i\in\Ob(\mathscr{C})$ carries an identity morphism, respected by all dg functors.  I.e. for each $i\in\Ob(\mathscr{C})$ there is a degree zero morphism $\eta_i\colon k \rightarrow \Hom_{\mathscr{C}}(i,i)$ satisfying $m_2\circ (\unit_{i,j}\otimes \eta_i)=\unit_{i,j}$ and $m_2\circ (\eta_i\otimes \unit_{j,i})=\unit_{j,i}$ where $\unit_{i,j}$ denotes the identity morphism on the cochain complex $\mathscr{C}(i,j)$.
\smallbreak
If $A$ is a dg algebra, we will frequently identify $A$ with the dg category $\mathscr{C}$ with one object $i$ satisfying $\Hom_{\mathscr{C}}(i,i)=A$ and with composition given by the multiplication map for $A$.
\subsubsection{}
Given a dg category $\mathscr{C}$ we form the homotopy category $\HO^0(\mathscr{C})$ as follows.  We set $\Ob(\HO^0(\mathscr{C}))=\Ob(\mathscr{C})$ and $\HO^0(\mathscr{C})(i,j)=\HO^0(\mathscr{C}(i,j),d)$.
\subsubsection{}
The category of cochain complexes itself can be enriched to a dg category $\Dub_{\dg}(\Mod^k)$ in a natural way:
\[
\Hom^p_{\Dub_{\dg}(\Mod^k)}((V,d_V),(W,d_W))=\prod_{r\in\mathbb{Z}}\Hom(V^r,W^{r+p})
\]
with differential
\[
d(f)^p=d_W\circ  f^{p-1}-(-1)^p f^{p-1}\circ d_V.
\]
\smallbreak
If $\mathscr{C}$ is a dg category, a (left) $\mathscr{C}$-module is a dg-functor $\mathscr{C}\rightarrow \Dub_{\dg}(\Mod^k)$.  Via the dg structure of $\Dub_{\dg}(\Mod^k)$, the category of left dg-modules over $\mathscr{C}$ itself is enriched to a dg category, which we denote $\dg\Mod^{\mathscr{C}}$.  We define right $\mathscr{C}$-modules to be dg-functors $\mathscr{C}^{\op}\rightarrow \Dub_{\dg}(\Mod^k)$, yielding the dg category $\dg\Mod_{\mathscr{C}}$.  A morphism $f$ of $\mathscr{C}$-modules is called a quasi-isomorphism if $f(i)$ is a quasi-isomorphism for every $i\in\Ob(\mathscr{C})$.  Formally inverting the quasi-isomorphisms inside the category $\HO^0(\dg\Mod^{\mathscr{C}})$, we arrive at the derived category $\Dub(\Mod^{\mathscr{C}})$.  This is a triangulated category.  Via the dg Yoneda embedding we obtain a functor of dg categories $G\colon\mathscr{C}\rightarrow \dg\Mod_{\mathscr{C}}$, and we call $\mathscr{C}$ pre-triangulated if the full subcategory of $\Dub(\Mod_{\mathscr{C}})$ containing those objects that are equivalent to $G(i)$ for $i\in \mathscr{C}$ is a triangulated subcategory of $\Dub(\Mod_{\mathscr{C}})$.
\subsubsection{}
Given a triangulated category $\mathscr{D}$, a \textit{dg enhancement} of $\mathscr{D}$ is a dg category $\mathscr{E}$ along with an equivalence of triangulated categories $\HO^0(\mathscr{E})\rightarrow\mathscr{D}$.  We will generally use the notation $\Dub_{\dg}^{?}(\mathscr{A})$ for a dg enhancement of a triangulated category $\Dub^{?}\!(\mathscr{A})$.  There is a natural dg enhancement $\Dub_{\dg}(\Mod^{\mathscr{C}})$ of $\Dub(\Mod^{\mathscr{C}})$, defined to be full subcategory of $\dg\Mod^{\mathscr{C}}$ containing the modules with a certain property (P) \cite[Sec.3.1]{Kel94} (see also \cite[Sec.3.2]{Kel06}).  In the case of the bounded derived category of an ordinary algebra, this amounts to considering the dg category of bounded above complexes of projective modules with bounded cohomology.
\smallbreak
For $\mathscr{A}$ a dg category, we denote by $\Perf(\mathscr{A})$ the full subcategory of $\Dub(\Mod_{\mathscr{A}})$ containing the compact objects, and by $\Perf_{\dg}(\mathscr{A})$ the full subcategory of $\Ddg(\Mod_{\mathscr{A}})$ containing the same objects.  A dg category $\mathscr{A}$ is called \textit{smooth} if the diagonal bimodule $\mathscr{A}$ is a compact object in $\Dub(\Mod_{\mathscr{A}^e})$, where $\mathscr{A}^e\coloneqq\mathscr{A}\otimes\mathscr{A}^{\opp}$.  Given a $\mathscr{A}\otimes\mathscr{A}^{\opp}$-module $M$ we denote by $M^!$ the derived dual, i.e. the chain complex
\[
\Hom_{\dg\Mod_{\mathscr{A}^e}}(QM,\mathscr{A}\otimes\mathscr{A}^{\opp})
\]
where the target is given the outer $\mathscr{A}^e$-action, and $QM$ is a replacement of $M$ by a quasi-isomorphic module satisfying property (P).  Via the inner $\mathscr{A}^e$-action on $\mathscr{A}\otimes\mathscr{A}^{\opp}$, the derived dual is again an $\mathscr{A}^e$-module (see \cite[Sec.3]{Ke11} for details).  

\subsubsection{}
We denote by $\CHC_{\bullet}(\mathscr{A})$ the Hochschild/cyclic complex of $\mathscr{A}$.  There is a natural quasi-isomorphism
\[
\CHC_{\bullet}(\mathscr{A})\cong\mathscr{A}\otimes^{\DL}_{\mathscr{A}^e}\mathscr{A}.
\]
Explicitly, following \cite{Lo13} we write 
\[
\CHC_{n}(\mathscr{A})=\begin{cases}\bigoplus_{i_0,\ldots i_{n}\in\Ob(\mathscr{A})}\big(\mathscr{A}(i_{n},i_0)[1]\otimes\mathscr{A}(i_{n-1},i_{n})[1]\otimes\ldots\otimes\mathscr{A}(i_{0},i_1)[1]\big)[-1]&n\geq 1\\
\bigoplus_{i\in\Ob(\mathscr{A})}\mathscr{A}(i,i) &n=0\end{cases}
\]
equipped with the differential $\Ab\coloneqq \Ab_1+d_{\Hoch}$, where \[
d^n_{\Hoch}=\sum_{r=0}^{n-2}\underbrace{\unit\otimes\ldots\otimes\unit}_{r\textrm{ times}}\otimes\Ab_2\otimes\underbrace{\unit\otimes\ldots\otimes\unit}_{n-2-r\textrm{ times}}+(\Ab_2\otimes \underbrace{\unit\otimes\ldots\otimes\unit}_{n-2 \textrm{ times}})\circ F_n
\]
and $F_n$ is the cyclic permutation of tensor factors sending the last to the first position.  Note that since $\Ab_2$ has degree 1, signs appear when actually evaluating this differential on elements.  Here and throughout the paper, if $\Am_2,\Am_1=d$ and later $\Am_3,\Am_4,\ldots$ are the operations $\mathscr{A}(-,-)\otimes\ldots\otimes\mathscr{A}(-,-)\rightarrow \mathscr{A}(-,-)$ defining the structure of a category, dg category, or $A_{\infty}$-category, following \cite{GJ90} we denote by $\Ab_1,\Ab_2,\ldots$ the morphisms
\[
\Ab_n\coloneqq S\circ \Am_n\circ (S^{-1})^{\otimes n}\colon \mathscr{A}(-,-)[1]\otimes\ldots\otimes\mathscr{A}(-,-)[1]\rightarrow \mathscr{A}(-,-)[1]
\]
induced by them, where we define
\begin{align*}
S\colon &\mathscr{A}(-,-)\rightarrow \mathscr{A}(-,-)\otimes k[1]\cong \mathscr{A}(-,-)[1]\\
&a\mapsto a\otimes 1.
\end{align*}
\smallbreak
For $V$ a graded vector space, following \cite{GJ90} it is standard to write $\Ab_n[a_1,\ldots,a_n]$ for $\Ab_n$ applied to $(a_1\otimes 1)\otimes \ldots (a_n\otimes 1)\in SV^{n}$.  So, for instance, it follows from the Koszul sign rule that $\Ab_1[a]=(-\Am_1(a)\otimes 1)$; see \cite[Lem.1.3]{GJ90} for this and the signs that appear in the relation between $\Ab_n$ and $\Am_n$ in general.  The cyclic complex carries the Connes operator $\ConB$, defined via
\begin{align}
\ConB_n\colon &\CHC_{n}(\mathscr{A})\rightarrow \CHC_{n+1}(\mathscr{A})
\\
& \ConB_n\coloneqq \sum_{i=1}^n(\eta \otimes F_n^i+ F_{n+1}(\eta \otimes F_n^i)).
\end{align}
Note that, as usual, signs appear when actually applying $\ConB$ to elements in $\CHC_{\bullet}(\mathscr{A})$, and taking care of them one finds that $\ConB^2 =\ConB\Ab+\Ab\ConB=\Ab^2=0$: see e.g. chapters one and two of \cite{Lo13}.
\smallbreak
Assuming that $\mathscr{A}$ is smooth, there is a quasi-equivalence
\begin{align*}
\bullet^{\flat}\colon&\CHC_{\bullet}(\mathscr{A})\rightarrow\Hom_{\Ddg(\Mod_{\mathscr{A}^e})}(\mathscr{A}^!,\mathscr{A})
\end{align*}
realised via adjunction and the natural quasi-isomorphism $\mathscr{A}\rightarrow \mathscr{A}^{!!}$.  There is a fully faithful morphism of dg categories $\mathscr{A}\rightarrow \Perf_{\dg}(\mathscr{A})$ via the Yoneda embedding, and the induced morphism
\begin{equation}
\label{AtoPerf}
\CHC_{\bullet}(\mathscr{A})\rightarrow \CHC_{\bullet}(\Perf_{\dg}(\mathscr{A}))
\end{equation}
is a quasi-isomorphism, since $\CHC_{\bullet}(-)$ sends Morita quasi-equivalences to quasi-isomorphisms.
\subsection{Calabi--Yau structures}
\label{LeftRight}
We next recall the two notions of CY structure on dg categories that will be relevant for this paper. Our notation is a mixture of the notation from \cite{BCS} and \cite{BD19}.
\subsubsection{Left nCY structures}
\label{lccy}
A \textit{left pre-nCY structure} on $\mathscr{A}$ is a cycle 
\[
\eta\colon k[n]\rightarrow \CHC_{\bullet}(\mathscr{A})^{S^1}, 
\]
where the target denotes the negative cyclic complex of $\CHC_{\bullet}(\mathscr{A})$, defined below.  As shown in \cite{hoyois}, the operator $\ConB$ on the Hochschild complex matches the action of the nondegenerate 1-simplex of the $S^1$-action on the Hochschild complex; so one may think of negative cyclic homology as the homotopy fixed points of an $S^1$-action on the Hochschild complex $\CHC_{\bullet}(\mathscr{A})$.  In more algebraic terms, we have
\[
\CHC_{\bullet}(\mathscr{A})^{S^1}\simeq \HC^-(\mathscr{A}) \coloneqq \RHom_{k[\delta]}(k,\CHC_{\bullet}(\mathscr{A})).
\]
Here $k[\delta]$ is the free (super)commutative $k$-algebra on the generator $\delta$ placed in homological degree $-1$, i.e. $k[\delta]=k\oplus k\cdot\delta$, by the Koszul sign rule, and $\delta$ acts on $\CHC_{\bullet}(\mathscr{A})$ via the operator $\ConB$.  An explicit model for $\CHC_{\bullet}(\mathscr{A})^{S^1}$ is defined by replacing $k$ by the quasi-isomorphic complex of $k[\delta]$ modules $P=(k[\delta][u],d)$ where $u$ has homological degree -2, and $du=\delta$, yielding what is usually called the \textit{negative cyclic complex}:
\[
\CHC_{\bullet}(\mathscr{A})^{S^1}\coloneqq \Hom_{\dg\Mod^{k[\delta]}}(P,\CHC_{\bullet}(\mathscr{A})).
\]
\smallbreak
The truncation map $k[\delta]\rightarrow k$ sending $\delta$ to zero is a $k[\delta]$-module morphism, defining a morphism
\[
\bullet^{\natural}\colon \HC^-(\mathscr{A})\rightarrow \RHom_{k[\delta]}(k[\delta],\CHC_{\bullet}(\mathscr{A}))\simeq \CHC_{\bullet}(\mathscr{A})
\]
which in topological terms is the inclusion map from the homotopy fixed points of the $S^1$-action.  Finally, a left pre-nCY structure $\eta$ is called a \textit{left nCY structure} if 
\begin{align}
    \label{nui_def}
(\eta(1)^{\natural})^{\flat}\colon \mathscr{A}^!\rightarrow\mathscr{A}[n]
\end{align}
is a quasi-equivalence.
\sssct
The morphism \eqref{AtoPerf} establishes a bijection between left nCY structures (considered as cohomology classes) on the categories $\mathscr{A}$ and $\Perf_{\dg}(\mathscr{A})$.  If $B$ is a dg algebra we may consider it as a dg category with one object.  An algebra equipped with a left nCY structure in the above sense (possibly with the extra condition that $g$ lifts to a class in negative cyclic homology) is what is often called a Calabi--Yau algebra in the literature (see \cite{Gi06} for an excellent account of these algebras).

\subsubsection{Right nCY structures}
A dg category $\mathscr{A}$ is called \textit{locally proper} if all morphism complexes have finite-dimensional total cohomology, i.e. are perfect objects in $\Db(\Mod^k)$.  Our principal geometric example will be the full dg subcategory of quasi-coherent sheaves $\mathcal{F}$ on a finite type scheme $S$, containing those $\CF$ for which the total cohomology of $\CF$ is a coherent sheaf on $S$ with compact support contained in the regular locus of $S$.

\smallbreak

Let $\sigma\colon \CHC_{\bullet}(\mathscr{A})\rightarrow k[-n]$ be a Hochschild cocycle.  By the adjunction
\[
\RHom_{\mathscr{A}^e}(\mathscr{A},\RHom(\mathscr{A},k))\cong \RHom(\mathscr{A}\otimes^{\DL}_{\mathscr{A}^e}\mathscr{A},k)
\]
we obtain a morphism of $\mathscr{A}$-bimodules $\mathscr{A}[n]\rightarrow \mathscr{A}^{\vee}$ where $\mathscr{A}^{\vee}$ is the linear dual of $\mathscr{A}$:
\[
\mathscr{A}^{\vee}(i,j)=\mathscr{A}(j,i)^{\vee}.
\]
\smallbreak
We define $\HC(\mathscr{A})=(\mathscr{A}\otimes^{\DL}_{\mathscr{A}^e}\mathscr{A})\otimes^{\DL}_{k[\delta]} k$, the cyclic homology of $\mathscr{A}$.  As in the case of negative cyclic homology considered in \S \ref{lccy}, there is an explicit chain complex for cyclic homology given by 
\[
\CHC_{\bullet}(\mathscr{A})_{S^1}\coloneqq\CHC_{\bullet}(\mathscr{A})\otimes_{k[\delta]}P.
\]
\smallbreak
We denote by $\CC_{n}(\mathscr{A})_{\cyc}$ the cokernel of the morphism 
\[
(1-F_n)\colon \CHC_{n}(\mathscr{A})\rightarrow \CHC_n(\mathscr{A}).
\]
The differential $\Ab$ induces a differential on $\CC_{\bullet}(\mathscr{A})_{\cyc}$ making it into a chain complex.  There is a natural morphism $\CHC_{\bullet}(\mathscr{A})_{S^1}\rightarrow \CC_{\bullet}(\mathscr{A})_{\cyc}$ which is moreover a quasi-isomorphism \cite[Thm.2.1.5]{Lo13}.  As in \cite{BD19} we define a cocycle
\[
\tilde{\sigma}\colon \CHC_{\bullet}(\mathscr{A})_{S^1}\rightarrow k[-n]
\]
to be a \textit{right nCY structure} if the induced morphism $\mathscr{A}[n]\rightarrow \mathscr{A}^{\vee}$ is a quasi-equivalence.  
\subsubsection{}The following theorem of Brav and Dyckerhoff relates left and right nCY structures:

\begin{proposition}[{\cite[Prop.~3.1]{BD19}}]\label{BDthm}
Let $\mathscr{A}$ be a dg category equipped with a left nCY structure $\eta$, and let $\mathscr{P}\subset \mathscr{A}$ be a full locally proper subcategory.  The left nCY structure on $\mathscr{A}$ induces a canonical right nCY structure $\tilde{\sigma}_{\mathscr{P}}(\eta)$ on $\mathscr{P}$.
\end{proposition}

We also have the converse
\begin{lemma}\label{lem:left_right_CY}
Let $\mathscr{A}$ be a smooth, locally proper dg category equipped with a right nCY structure 
\[
\sigma\colon\CHC_{\bullet}(\mathscr{A})_{S^1}\rightarrow k[-n].
\]
Then $\sigma$ induces a left nCY structure $\tilde{\eta}(\sigma)$ on $\mathscr{A}$, and moreover the assignment $\sigma\mapsto \tilde{\eta}(\sigma)$ is an inverse to the assignment of Proposition~\ref{BDthm}.
\end{lemma}

\begin{proof}
The construction of $\tilde{\eta}(\sigma)$ is the adjoint of the construction of \cite{BD19}, and the statement that these procedures are inverses to each other follows from standard properties of adjunction.
\smallbreak
We give a few more details.  Since the category $\mathscr{A}$ is locally proper, there is an adjunction
\[
-\otimes_{\mathscr{A}^e}^{\DL} \mathscr{A}_{k}\colon\Perf_{\dg}(\mathscr{A}^e)\leftrightarrow \Perf_{\dg}(k)\colon -\otimes_{k}  ({}_{k}\mathscr{A}^{\vee})
\]
where $\mathscr{A}$ is the diagonal bimodule, which we indicate is to be treated e.g. as a $(\mathscr{A}^e,k)$-bimodule via the notation $\mathscr{A}_{k}$, and similarly for its vector dual $\mathscr{A}^{\vee}$.  Applying $\CHC_{\bullet}$ to the functor \\$-\otimes_{k}\!({}_{k}\mathscr{A}^{\vee})$ we obtain a morphism
\[
k\simeq \CHC_{\bullet}(\Perf_{\dg}(k))\rightarrow \CHC_{\bullet}(\Perf_{\dg}(\mathscr{A}^e))\simeq \CHC_{\bullet}(\mathscr{A})\otimes \CHC_{\bullet}(\mathscr{A}^{\opp}).  
\]
Applying $\sigma$ to the $\CHC_{\bullet}(\mathscr{A})$ factor, via the isomorphism $\CHC_{\bullet}(\mathscr{A}^{\opp})\cong \CHC_{\bullet}(\mathscr{A})$ we obtain the desired class $\tilde{\eta}(\sigma)$.  This is the adjoint of the construction in \cite[Prop.~3.1]{BD19}, and the proof that the class $\tilde{\eta}(\sigma)$ is cyclically invariant and nondegenerate is the same as theirs.  
\smallbreak
By the standard adjunction properties of duals, the composition
\[
\Perf_{\dg}(k\otimes \mathscr{A})\xrightarrow{-\otimes^{\DL}_{k\otimes \mathscr{A}}({}_{k}\mathscr{A}^{\vee}\otimes \mathscr{A})}\Perf_{\dg}(\mathscr{A}\otimes\mathscr{A}^{\opp}\otimes \mathscr{A})\xrightarrow{-\otimes^{\DL}_{\mathscr{A}\otimes\mathscr{A}^{\opp}\otimes \mathscr{A}}(\mathscr{A}\otimes \mathscr{A}_{k})} \Perf_{\dg}(\mathscr{A}\otimes k)
\]
is a quasi-equivalence.  Applying $\CHC_{\bullet}$ to the above quasi-equivalence, and precomposing with $\eta$, yields $\tilde{\eta}(\tilde{\sigma}_{\mathscr{A}}(\eta))=\eta$.  The argument that $\tilde{\sigma}_{\mathscr{A}}(\tilde{\eta}(\sigma))=\sigma$ is the same.
\end{proof}

\subsection{$A_{\infty}$-categories}
\label{sec:A_infinity_algebras}
We recall some basic features of the theory of $A_{\infty}$-algebras and $A_{\infty}$-categories, which is the language in which the deformation theory of objects in $2$-Calabi--Yau categories will be described.  We refer to \cite{Kel01,Kel06b,KLH,Se08} for more details regarding $A_{\infty}$-categories and their (bi)modules, and to \cite{KoSo02,KSnotes} for more details on the point of view we adopt below.
\subsubsection{}
A (small) $A_{\infty}$-category $\mathscr{A}$ is given by a set of objects $\Ob(\mathscr{A})$, and for each pair of objects $i,j\in\Ob(\mathscr{A})$ a $\BZ$-graded $k$-vector space
\begin{equation*}
    \mathscr{A}(i,j)\coloneqq\Hom_{\mathscr{A}}(i,j) = \bigoplus_{p \in \BZ} \Hom_{\mathscr{A}}^p(i,j).
\end{equation*}
For each $n\geq 1$ and each tuple $i_1,\ldots,i_{n+1}\in\Ob(\mathscr{A})$ we are given also the data of a graded linear map of degree 1,
\begin{equation*}
    \Ab_{n}\colon \mathscr{A}(i_n,i_{n+1})[1]\otimes\mathscr{A}(i_{n-1},i_{n})[1]\otimes\ldots\otimes \mathscr{A}(i_1,i_2)[1] \to \mathscr{A}(i_1,i_{n+1})[1],
\end{equation*}
satisfying the infinite list of relations, one for each $n \geq 1$, 
\begin{equation}\label{eq:A_infinity_relation_m}
    \sum_{\substack{r,t\geq 0, s\geq 1\\r+s+t=n}} \Ab_{r+1+t}(\unit_{*,*}^{\otimes r} \otimes \Ab_{s} \otimes \unit_{*,*}^{\otimes t}) = 0.
\end{equation}
Here $\unit_{*,*}$ denotes the identity map on whichever Hom space of $\mathscr{A}$ it is applied to.
For each $n$ one may express these compatibility relations in terms of the operators $\Am_n$.  For instance, the relation for $n = 1$ reads $\Ab_1\Ab_1 = 0$, so $(A,\Am_1)$ is a cochain complex.  The relation for $n = 2$ is equivalent to
\begin{equation}\label{eq:A_infinity_relation_mm}
    \Am_1\Am_2 = \Am_2(\Am_1 \otimes \unit_{*,*} + \unit_{*,*} \otimes \Am_1).
\end{equation}
Hence $\Am_1$ is a graded derivation with respect to the multiplication $\Am_2$.  In passing from \eqref{eq:A_infinity_relation_m} to \eqref{eq:A_infinity_relation_mm} one of the terms has passed to the left hand side: in general, thanks to the Koszul sign rule, the $\Am$ versions of the identities defining $A_{\infty}$-structures involve increasingly complicated sign rules compared to the $\Ab$ versions.

\subsubsection{}
Given an $A_{\infty}$-category $\mathscr{A}$, we may attempt to form a graded category $\HO^{\bullet}(\mathscr{A})$, by setting the objects of $\HO^{\bullet}(\mathscr{A})$ to be the objects of $\mathscr{A}$, and defining the morphism spaces to be $\HO^{\bullet}(\mathscr{A})(i,j)\coloneqq\HO^{\bullet}(\Hom_{\mathscr{A}}(i,j),\Am_1)$.  The composition of morphisms is induced from the $\Am_2$ operations in $\mathscr{A}$.  Associativity holds automatically, although from the definitions given above it is not guaranteed that $\HO^{\bullet}(\mathscr{A})$ contains identity morphisms.  We always assume that it does, i.e. that $\mathscr{A}$ is \textit{weakly unital}, and more precisely that for each $i\in\Ob(\mathscr{A})$ there is an endomorphism $1_i\in\mathscr{A}(i,i)$ that becomes the identity in the category $\HO^{\bullet}(\mathscr{A})$.  We will be particularly interested in a stronger notion of unitality:
\begin{definition}
We call $\mathscr{A}$ \textit{unital}\footnote{Some authors use ``strongly unital'' for this notion.} if for each $i\in\Ob(\mathscr{A})$ there is a graded linear map $\eta_i \colon k \to \Hom_{\mathscr{A}}(i,i)$ of degree zero, so a distinguished element $1_i \coloneqq \eta_i(1)$, such that
\begin{enumerate}
    \item $\Am_1\eta_i = 0$, i.e. $\Am_1(1_i) = 0$,
    \item $\Am_2(\unit_{i,j} \otimes \eta_i) = \unit_{i,j} = \Am_2(\eta_j \otimes \unit_{j,i})$, where $\unit_{i,j}$ is the identity morphism on $\mathscr{A}(i,j)$, i.e. $\eta$ is a unit with respect to $\Am_2$, and
    \item $\Am_{n}(\unit_{\ast,\ast}^{\otimes r-1} \otimes \eta_{\ast} \otimes \unit_{\ast,\ast}^{\otimes n-r}) = 0$ for all $n \geq 3$ and $1\leq r\leq n$.
\end{enumerate}
\end{definition}

\subsubsection{}
We fix $R$ to be the category with objects the same as $\mathscr{A}$, and homomorphisms given by scalar multiples of the identity morphisms in $\mathscr{A}$.  We use the notational shorthand
\begin{equation}
\label{Rdef}
\underbrace{\mathscr{A}\otimes_{R}\ldots\otimes_{R}\mathscr{A}}_{n\textrm{ times}}\coloneqq\bigoplus_{i_0,\ldots,i_n\in\Ob(\mathscr{A})}\mathscr{A}(i_{n-1},i_n)\otimes_k\ldots\otimes_k\mathscr{A}(i_0,i_1).
\end{equation}
A functor between $A_{\infty}$-categories $f \colon \mathscr{A} \to \mathscr{B}$ is given by a morphism of sets $f\colon\Ob(\mathscr{A})\rightarrow \Ob(\mathscr{B})$, and for all $n\geq 1$ a degree zero morphism
\begin{equation}
    f_{n} \colon \mathscr{A}[1]\otimes_R\ldots\otimes_R\mathscr{A}[1]\rightarrow \mathscr{B}[1]
\end{equation}
satisfying the infinite list of relations
\begin{equation}
        \sum_{r+s+t=n} f_{r+1+t}(\unit_{\ast,\ast}^{\otimes r} \otimes \Ab_{s} \otimes \unit_{\ast,\ast}^{\otimes t}) = \sum_{i_1+\ldots+i_r=n}  \Ab_l(f_{i_1} \otimes f_{i_2} \otimes \ldots \otimes f_{i_r})
    \end{equation}
for $n\geq 1$.  For instance for $n=1$, the relation says that each $f_1$ is a morphism of cochain complexes.  A functor $f$ is called \textit{strict} if $f_n=0$ for all $n\geq 2$.  The functor $f$ is a \textit{quasi-equivalence} if $f_1$ is a quasi-isomorphism when evaluated on any pair of objects in $\mathscr{A}$, and $\HO^0(f)$ is an equivalence of categories.  The functor $f$ is a \textit{quasi-isomorphism} if $\HO^0(f)$ is moreover an isomorphism of categories.  We always presume that $f$ becomes an ordinary functor after passing to the homotopy category, i.e. it respects homotopy units.  The composition of two $A_{\infty}$-functors $f \colon \mathscr{B} \to \mathscr{C}$ and $g \colon \mathscr{A} \to \mathscr{B}$ is defined by
\begin{equation*}
    (f \circ g)_{n} = \sum_{i_1+\ldots+i_r=n} f_r(g_{i_1} \otimes g_{i_2} \otimes \ldots \otimes g_{i_r}).
\end{equation*}

\subsubsection{}
An $A_{\infty}$-category is called \textit{minimal} if $\Ab_1\colon \Hom_{\mathscr{A}}(i,j)\rightarrow \Hom_{\mathscr{A}}(i,j)$ is the zero map for every $i,j\in\Ob(\mathscr{A})$.  An $A_{\infty}$-category is called \textit{formal} if it is quasi-equivalent to a minimal $A_{\infty}$-category for which the higher multiplications $\Ab_i$ vanish for $i\geq 3$.
\smallbreak
Every $A_{\infty}$-category is quasi-isomorphic to a minimal one.  This is the homological perturbation lemma due to Kadei\v{s}hvili~\cite{MR580645}; see also~\cite[Thm.3.4]{MR1672242}.  In particular, if $\mathscr{A}$ is a dg category, we may find an $A_{\infty}$-quasi-isomorphism $f\colon \mathscr{B}\rightarrow \mathscr{A}$ from a minimal $A_{\infty}$-category $\mathscr{B}$.  In addition, every weakly unital $A_{\infty}$-category is quasi-isomorphic to a unital one, by \cite[Sec.3.2.1]{KLH}.  This result goes by the name of ``strictification of units''.  One of the main technical results we will use is the strictification of units in the cyclic setting, which we describe next.

\subsection{Cyclic $A_{\infty}$-categories}
\label{cycAsec}
We recall some facts about cyclic $A_{\infty}$-categories, referring the reader to \cite{Co07,Kaj07,KSnotes,Se08} for more details. For simplicity, in this section and \S \ref{ncdc_sec}-\ref{new_form} we work with $A_{\infty}$-categories containing finitely many objects, and assume that the homomorphism spaces $\mathscr{A}^{\bullet}(i,j)$ are finite-dimensional, although these conditions may be partially relaxed with some care, see e.g. \cite{Cho08,ChoLee11} for details.  
\subsubsection{}
A \textit{cyclic pairing} of dimension $d$ on an $A_{\infty}$-category $\mathscr{A}$ is a degree zero nondegenerate graded skew-symmetric\footnote{The question of whether one should consider symmetric or skew-symmetric forms here involves a slightly delicate analysis of signs, for which we refer to \cite[Sec. 2]{Cho08} for details.} pairing $\langle -,- \rangle \colon \mathscr{A}(i,j)[1] \otimes \mathscr{A}(j,i)[1] \to k[2-d]$ for all $i,j\in\Ob(\mathscr{A})$ such that for $n\geq 1$
\[
\langle \bullet,\bullet\rangle\circ(\Ab_n\otimes\unit_{\ast,\ast})=\langle \bullet,\bullet\rangle\circ(\Ab_n\otimes \unit_{\ast,\ast})\circ F_{n+1},
\]
i.e.
\begin{equation}\label{eq:cyclic_pairing}
    \langle \Ab_n[a_0,\ldots,a_{n-1}],a_n\otimes 1 \rangle = (-1)^{(|a_0|+1)\sum_{l=1}^n (|a_l|+1)} \langle \Ab_n[a_1,\ldots,a_n],a_0\otimes 1 \rangle
\end{equation}
for all $a_0,a_1,\ldots,a_n $ graded homogeneous homomorphisms in $\mathscr{A}$.
\smallbreak
A functor $f\colon (\mathscr{A},\langle -,-\rangle_{\mathscr{A}})\rightarrow (\mathscr{B},\langle -,-\rangle_{\mathscr{B}})$ between cyclic $A_{\infty}$-categories is a usual $A_{\infty}$-functor that satisfies the additional two conditions:
\begin{itemize}
\item There is an equality of pairings $\langle -,-\rangle_{\mathscr{B}}\circ (f_1\otimes f_1)=\langle -,-\rangle_{\mathscr{A}}$.
\item For all $n\geq 3$
\[
\sum_{1\leq r\leq n}\langle -,-\rangle_{\mathscr{B}}\circ (f_r\otimes f_{n-r})=0.
\]
\end{itemize}
The second condition, in particular, becomes easier to motivate within the context of noncommutative differential geometry; see~\S \ref{ncdc_sec}.
\subsubsection{}
Let $\mathscr{A}$ and $\mathscr{B}$ be two $A_{\infty}$-categories.  We define $R$ as in \eqref{Rdef} and define $R'$ as the analogous category for $\mathscr{B}$.  An $(\mathscr{A},\mathscr{B})$-bimodule is given by a $(R,R')$-bimodule $M$, along with a set of morphisms for all $i,j\geq 0$, with $i+j\geq 1$
\[
\Am_{i,j}\colon \mathscr{A}^{\otimes_{R} i}\otimes_R M\otimes_{R'} \mathscr{B}^{\otimes_{R'} j}\rightarrow M
\]
of degree $1-i-j$ satisfying higher compatibility relations analogous to those for an $A_{\infty}$-category.  A morphism $f$ between $A_{\infty}$-bimodules is given by morphisms 
\[
f_{i,j}\colon \mathscr{A}^{\otimes_{R} i}\otimes_R M\otimes_{R'} \mathscr{B}^{\otimes_{R'} j}\rightarrow N
\]
for $i,j\geq 0$ satisfying certain higher compatibility relations analogous to those for an $A_{\infty}$-functor.  Such a morphism is called strict if $f_{i,j}=0$ for $i+j\geq 1$.  A cyclic pairing determines and is determined by a strict isomorphism of $(\mathscr{A},\mathscr{A})$-bimodules
\[
    f\colon \mathscr{A}[1]\rightarrow \mathscr{A}^{\vee}[1-d];\quad\quad
    f_1 \colon a\mapsto \langle a,\bullet\rangle
\]
satisfying $f_1=-f_1^{\vee}[d]$.
\begin{lemma}[\cite{Cho08}]\label{ChoLemma}
A homogeneous degree $-d$ element $\psi\in \CHC_{\bullet}(\mathscr{A})^{\vee}$ determines a morphism of $\mathscr{A}$-bimodules $f\colon\mathscr{A}[1]\rightarrow \mathscr{A}^{\vee}[1-d]$.  If, furthermore, $\psi\in \CHC_{\bullet}(\mathscr{A})^{\vee}_{\cyc}$ and is nondegenerate, meaning that $f$ is a quasi-isomorphism, then there is an $A_{\infty}$-isomorphism of categories $\mathscr{A}\rightarrow \mathscr{B}$ such that the induced morphism of $\mathscr{B}$-bimodules $\mathscr{B}[1]\rightarrow \mathscr{B}^{\vee}[1-d]$ is strict, i.e. it comes from a cyclic pairing.
\end{lemma}
This lemma is an analogue of the noncommutative Darboux theorem in noncommutative symplectic geometry, again see \S \ref{ncdc_sec}.  
\subsubsection{}
Once we fix the pairing $\langle -,-\rangle$, the cyclic $A_{\infty}$-category $\mathscr{A}$ is determined by the \textit{potential} $\mathcal{W}\in (\prod_{i\geq 2}\CHC_{i}(\mathscr{A}))^{\vee}_{\cyc}$, that is, by choosing degree $3-d$ elements 
\begin{equation*}
\mathcal{W}_n\in \mathscr{A}(i_{0},i_1)[1]^{\vee}\otimes\ldots\otimes\mathscr{A}(i_{n-1},i_n)[1]^{\vee}\otimes\mathscr{A}(i_n,i_0)[1]^{\vee}
\end{equation*}
for each $n\geq 2$, where we identify $\mathscr{A}(i_n,i_0)[1]^{\vee}=\mathscr{A}(i_0,i_n)[d-1]$ via the pairing $\langle-,-\rangle$.  Thus $\mathcal{W}_n$ determines an operation $\Ab_{n-1}$ satisfying \eqref{eq:cyclic_pairing}, since we assume that $\mathcal{W}_n$ is cyclically invariant.  

\subsubsection{}
In the $2$-Calabi--Yau setting, the following lemma is an elementary but crucial observation; see also~\cite[Sec.2.3]{Tu14}.
\begin{lemma}\label{rigLem}
    Let $(\mathscr{A},\Am_n)$ be a minimal $A_{\infty}$-category, and assume that $\CF_1,\ldots,\CF_r$ is a set of objects in $\mathscr{A}$ satisfying the following conditions:
    \begin{enumerate}
        \item $\mathscr{A}^{n}(\CF_i,\CF_j) =\begin{cases} k &\textrm{if }i=j\textrm{ and }n=0\\0&\textrm{if } n<0\\
        0&\textrm{if }n=0\textrm{ and }i\neq j.\end{cases}$
        \item The full subcategory $\mathscr{A}'$ of $\mathscr{A}$ containing the $\CF_i$ is unital.
        \item $\mathscr{A}'$ carries a cyclic pairing of dimension two: $\langle -,- \rangle \colon \mathscr{A}'(\CF_i,\CF_j)[1] \otimes \mathscr{A}'(\CF_j,\CF_i)[1] \to k$ for all $i,j$.
    \end{enumerate}
    Then $\{\CF_1,\ldots,\CF_r\}$ is a $\Sigma$-collection and $\Am_{n} = 0$ for all $n \geq 3$.  In particular $\mathscr{A}'$ is formal.
\end{lemma}

\begin{proof}
    Since $\mathscr{A}'^{p} = 0$ for $p < 0$, the pairing is nondegenerate, and $\mathscr{A}$ is minimal, the pairing induces isomorphisms
    \begin{equation*}
        \mathscr{A}^p(\CF_i,\CF_j) \xrightarrow{\cong} \mathscr{A}^{2-p}(\CF_j,\CF_i)^{\vee},
    \end{equation*}
and so $\mathscr{A}^p(\CF_i,\CF_j) = 0$ for $p > 2$, or for $p=2$ and $i\neq j$.  The statement that the $\CF_i$ form a $\Sigma$-collection follows from the given data regarding dimensions of graded morphism spaces along with skew-symmetry of the pairing.
\smallbreak
We wish to show that
    \begin{equation}
    \label{unit_test}
        \Am_{n}(a_n,\ldots,a_1) = 0 \quad \text{for } n \geq 3
    \end{equation}
    Unitality implies that \eqref{unit_test} holds if $|a_i| \leq 0$ for some $1 \leq i \leq n$, since then by assumption (1) of the lemma, $a_i$ must be a multiple of a unit.
On the other hand, if $\lvert a_i\lvert \geq 1$ for all $i$ and $\lvert a_i\lvert > \!1$ for some $i$ we find that the degree of $\Am_n(a_n,\ldots,a_1)$ is greater than two, so $\Am_n(a_n,\ldots,a_1)=0$.  So we may assume that all the morphisms in \eqref{unit_test} have degree one, and we may assume that the domain $i$ of $a_1$ is equal to the target $j$ of $a_n$ (for else $\mathscr{A}^2(i,j)=0$).  

Then $\Am_n(a_n,\ldots,a_1)$ is a scalar multiple of the linear dual of $1_i$ under $\langle-,-\rangle$.  We find
\[
\langle \Am_n(a_n,\ldots,a_1),1_i\rangle=\pm\langle \Am_n(a_{n-1},\ldots,a_1,1_i),a_n\rangle=0
\]
by cyclicity and unitality, finishing the proof.
\end{proof}

\subsection{Noncommutative differential calculus}
\label{ncdc_sec}
We retain our finiteness assumptions from the previous section, and recall some background from \cite{KSnotes}, as preparation for tackling strictification of (cyclic) units, and formality.  

\subsubsection{}For $\mathscr{A}$ an $A_{\infty}$-category, we form the completed tensor algebra
\begin{equation}
\label{Odef}
\mathcal{O}(X_{\mathscr{A}})\coloneqq \FreeT_{R}\!\left(\mathscr{A}[1]^{\vee}\right).
\end{equation}
On the right hand side we have formal linear combinations of elements $a_1^*\otimes\ldots\otimes a_n^*$, where $a_i^*$ are linear operators on (shifted) homomorphism spaces of $\mathscr{A}$.  For example, if $\mathscr{A}$ contains only one object $i$, so that we may identify $\mathscr{A}$ with the $A_{\infty}$-endomorphism algebra $A$ of $i$, $\CO(X_{\mathscr{A}})$ is the algebra of noncommutative formal functions on the (shifted) underlying vector space of $A$, i.e. $k\langle\!\langle x_1,\ldots,x_p\rangle\!\rangle$ where $x_1,\ldots,x_p$ is a dual basis for $A[1]$.  In general, the right hand side of \eqref{Odef} is a topological category\footnote{In the sense of a topological algebra with many objects.}, as it can be constructed as the completion of the usual freely generated category $\mathrm{T}_{R}\mathscr{A}[1]^{\vee}$ with respect to the topology which has as a basis of open neighbourhoods of zero the two sided ideals
\[
\FreeT^{\geq n}_{R}\!\left(\mathscr{A}[1]^{\vee}\right)=\prod_{m\geq n}\underbrace{\mathscr{A}[1]^{\vee}\otimes_R\ldots\otimes_R \mathscr{A}[1]^{\vee}}_{m\textrm{ times}}.
\]

\subsubsection{}The object $\mathcal{O}(X_{\mathscr{A}})$ carries the usual composition given by composition of tensors, so that we may think of it as a category with the same set of objects as $\mathscr{A}$.  From the dual of operations $\Ab_n$ we define degree one operations 
\[
\Ab_n^{\vee}\colon \mathscr{A}[1]^{\vee}\rightarrow (\mathscr{A}[1]^{\vee})^{\otimes_R n}.
\]
The sum $\sum_{n\geq 0}\Ab_n^{\vee}$ extends uniquely to a continuous derivation (vector field) $Q_{\mathscr{A}}$ of $\mathcal{O}(X_{\mathscr{A}})$.  It is easy to check that the conditions on $\Ab_n$ to define an $A_{\infty}$-category are equivalent to the condition $[Q_{\mathscr{A}},Q_{\mathscr{A}}]=0$, and the condition on a morphism $f\colon\mathscr{A}\rightarrow \mathscr{B}$ to define a morphism of $A_{\infty}$-categories is equivalent to the condition that $\hat{f}\circ Q_{\mathscr{B}}=Q_{\mathscr{A}}\circ\hat{f}$, where $\hat{f}\colon \mathcal{O}(X_{\mathscr{B}})\rightarrow \mathcal{O}(X_{\mathscr{A}})$ is the algebra morphism induced by the Taylor coefficients $f_n^{\vee}$.  

\smallbreak
The above constructions establish an equivalence of categories between finite-dimensional $A_{\infty}$-categories, and finitely generated free topological dg categories $\mathcal{O}(X)$ equipped with degree one vector fields $Q$ satisfying $[Q,Q]=0$.

\subsubsection{}
We define the $R$-bimodule $D\mathscr{A}=\mathscr{A}[1]^{\vee}\oplus\mathscr{A}[2]^{\vee}$, and set\footnote{Where there is no danger of confusion, from now on we will drop the $\mathscr{A}$ subscript from $X_{\mathscr{A}}$.}
\[
\Omega(X_{\mathscr{A}})=\FreeT_R(D\mathscr{A})
\]
where the completion is with respect to the $\mathscr{A}[1]^{\vee}$-factor of $T\mathscr{A}$, i.e. in the case above of a category with a single object, we can write
\[
\Omega(X)=k\langle\!\langle x_1,\ldots,x_p\rangle\!\rangle\langle dx_1,\ldots,dx_p\rangle
\]
where $\lvert dx_i\lvert=\lvert x_i\lvert+1$.  In general, for $\theta\in\mathscr{A}(i,j)[1]^{\vee}$ a linear form on $\mathscr{A}(i,j)[1]$ we write $d\theta$ for its shift in $\mathscr{A}[2]^{\vee}$.  We have the decomposition
\begin{equation}
    \label{OmegaDec}
\Omega(X)=\bigoplus_{n\geq 0}\Omega^n(X)
\end{equation}
where the decomposition arises from counting the number of terms of the form $d\theta$ (as opposed to $\theta$) appearing in a tensor.
\smallbreak
A homomorphism of completed algebras $f\colon \mathcal{O}(X_{\mathscr{A}})\rightarrow \mathcal{O}(X_{\mathscr{B}})$ determines a unique morphism $\Omega(X_{\mathscr{A}})\rightarrow \Omega(X_{\mathscr{B}})$ via the usual product rule in differential calculus, which we will often denote by the same symbol $f$.  We define 
\[
\Omega_{\cyc}(X)\coloneqq \Omega(X)/[\Omega(X),\Omega(X)]_{\topo}
\]
where the ``$\topo$'' subscript signifies that we take the topological completion of the two-sided ideal generated by the commutators in $\Omega(X)$.  Again we may decompose
\begin{equation}
    \label{COmegaDec}
\Omega_{\cyc}(X)=\bigoplus_{n\geq 0}\Omega_{\cyc}^n(X)
\end{equation}
where $\Omega_{\cyc}^n(X)$ is the space of linear combinations of forms containing $n$ terms of the form $d\theta$.  Note that $\Omega^0(X)=\mathcal{O}(X)$.  We define the reduced space
\[
\ol{\mathcal{O}}(X)\subset \mathcal{O}(X)
\]
to be the subspace annihilating all of the unit morphisms in $\mathscr{A}$.  
\subsubsection{}
Both $\Omega(X)$ and $\Omega_{\cyc}(X)$ carry the usual de Rham differential, defined by setting $d_{\dR}f=df$ and $d_{\dR}df=0$, for $f\in\mathcal{O}(X)$, and then extending by the Leibniz rule and continuity.  This differential increases the degree by one, where the degree of a form is determined by the gradings of \eqref{OmegaDec} and \eqref{COmegaDec}.  The complex $(\Omega_{\cyc}(X),d_{\dR})$ is acyclic \cite[Lem.4.8]{Kaj07}.
\smallbreak

Given a vector field $Q$, we define a contraction derivation $\iota_Q$ on $\Omega(X)$ via the morphism that sends $f\in \mathscr{A}(i,j)[1]^{\vee}$ to zero, and sends $df\mapsto Q(f)$.  Then we define the Lie derivative
\[
L_{Q}=[d_{\dR},\iota_Q],
\]
which preserves the gradings \eqref{OmegaDec} and \eqref{COmegaDec}.
\subsubsection{}

Just as there is an equivalence of categories translating the study of $A_{\infty}$-categories into the world of noncommutative formal manifolds and vector fields, there is a natural way of viewing cyclic $A_{\infty}$-categories as noncommutative formal manifolds with symplectic structure, which we now recall, following \cite{KSnotes}.
\smallbreak

Let $\omega\in \Omega_{\cyc}^2(X)$ be a $d_{\dR}$-closed cyclic 2-form.  Via the projection of $R$-bimodules $\mathscr{A}[1]^{\vee}\oplus\mathscr{A}[2]^{\vee}\rightarrow \mathscr{A}[2]^{\vee}$, the element $\omega$ induces an element $\omega_0\in \mathscr{A}[2]^{\vee}\otimes_R\mathscr{A}[2]^{\vee}$, and (as in \cite[Def 11.1.2]{KoSo02}) we say it is nondegenerate if $\omega_0$ induces an isomorphism $\mathscr{A}[2]\cong \mathscr{A}[2]^{\vee}$.
%Given $f\in\mathscr{A}(i,j)$, we consider $f$ as a derivation on $\Omega(X)$ sending $g\mapsto 0$ and $dg\mapsto g(f)$.  Denote by $l\colon \Omega(X)\rightarrow R$ the natural augmentation.  We say that $\omega$ is nondegenerate if the pairing 
%\[
%(f,g)\mapsto l(\iota_f\iota_g(\omega))
%\]
%is nondegenerate.  
Concretely, if we pick a basis $\alpha_1,\ldots,\alpha_n$ for the union of all of the dual morphism spaces $\mathscr{A}(i,j)[1]^{\vee}$, and write 
\begin{equation}
    \label{omegaform}
\omega=\sum_{r\geq 0}f_r(\alpha_1,\ldots,\alpha_n,d\alpha_1,\ldots,d\alpha_n)
\end{equation}
where $f_r(\ldots)$ is a linear combination of monomials containing $r$ instances of the terms $\alpha_1,\ldots,\alpha_n$, we ask that $f_0$ determine a nondegenerate bilinear pairing on the vector space spanned by $d\alpha_1^*\ldots,d\alpha_n^*$.  Note that this definition is coordinate-independent, in the sense that if $\Phi\colon \mathcal{O}(X_{\mathscr{A}})\rightarrow \mathcal{O}(X_{\mathscr{A}})$ is a formal automorphism, $\Phi^*\omega$ is nondegenerate if and only if $\omega$ is.  By the noncommutative Darboux theorem \cite[Thm.4.15]{Kaj07}, for any nondegenerate $\omega$ there is a 
formal change of variables $\Phi$ such that $\Phi^*\omega$ is constant, i.e. can be written in the form \eqref{omegaform} with $f_r=0$ for $r\geq 1$. 

\subsubsection{}
If $\langle-,-\rangle=g\in \mathscr{A}[1]^{\vee}\otimes_{R}\mathscr{A}[1]^{\vee}$ is a nondegenerate skew-symmetric form on $\mathscr{A}$, we define the form $\omega=(d_{\dR}\otimes d_{\dR})g$.  Nondegeneracy of $\omega$ follows from nondegeneracy of $\langle-,-\rangle$, and clearly this symplectic form is constant.  Cyclic invariance of $\langle-,-\rangle$ translates to the condition $L_Q\omega =0$.  Since $\Omega_{\cyc}(X_{\mathscr{A}})$ is acyclic for the de Rham differential, it follows that there is a unique $\mathcal{W}\in \Omega_{\cyc}^0(X_{\mathscr{A}})$ such that $d\mathcal{W}=\iota_Q\omega$, which one may check is the potential $\mathcal{W}$ appearing in \S \ref{cycAsec}.  We recover $Q$ from $\mathcal{W}$ via $Q=\{\mathcal{W},-\}$, where $\{-,-\}$ is the Poisson bracket defined by the symplectic form $\omega$, and the condition $[Q,Q]=0$ is equivalent to $\{\mathcal{W},\mathcal{W}\}=0$.  We summarise the above discussion in a proposition.

\begin{proposition}
There is an equivalence of categories between cyclic $A_{\infty}$-categories and triples $(\mathcal{O}(X),\omega,\mathcal{W})$, where $\mathcal{O}(X)$ is a finitely generated free graded topological category, $\omega\in \Omega^2_{\cyc}(X)$ is a constant symplectic form, and $\mathcal{W}\in \mathcal{O}^{\geq 2}_{\cyc}(X)$ satisfies $\{\mathcal{W},\mathcal{W}\}=0$.
\end{proposition}
\subsection{Strictification of units and formality}
\label{new_form}
We continue to assume $\mathrm{char}(k)=0$.  Let $f\colon\mathscr{A}\rightarrow \mathscr{B}$ be a morphism of $A_{\infty}$-categories, and assume that $\mathscr{A}$ and $\mathscr{B}$ carry cyclic structures.  The morphism $f$ defines the morphisms
\begin{align*}
f^{\vee}\colon&\mathcal{O}(X_{\mathscr{B}})\rightarrow \mathcal{O}(X_{\mathscr{A}})\\
&\Omega(X_{\mathscr{B}})\rightarrow \Omega(X_{\mathscr{A}})
\end{align*}
and in terms of the symplectic forms $\omega_{\mathscr{A}}$ and $\omega_{\mathscr{B}}$ defined by the cyclic structures, the condition for $f$ to be a morphism of cyclic $A_{\infty}$-categories equates to the condition $f^{\vee}\omega_{\mathscr{B}}=\omega_{\mathscr{A}}$ \cite[Sec.4.5]{Kaj07}.
\subsubsection{}
Let $\psi$ be a degree zero vector field on $\mathcal{O}(X_{\mathscr{A}})$, and let us assume that the Taylor series of $\psi$ starts in degree 2 or higher, i.e. $\psi(x_r)$ is a noncommutative polynomial starting in degree 2 or higher for each $x_r\in \{x_1,\ldots,x_p\}=\mathscr{A}[1]^{\vee}$.  We let $\psi$ act on $\Omega(X_{\mathscr{A}})$ via the operator $L_{\psi}$.  Then we define the formal automorphisms of $\mathcal{O}(X_{\mathscr{A}})$ and $\Omega(X_{\mathscr{A}})$
\[
\hat{\psi}\coloneqq e^{\psi}=1+\psi+\psi\circ\psi/2+\ldots+\underbrace{\psi\circ\cdots\circ\psi}_{n\textrm{ times}}/n!+\ldots
\]
\begin{proposition}\cite[Sec.4.3]{Kaj07}
\label{KajSym}
If $L_\psi\omega=0$, then $\hat{\psi}^*\omega=\omega$ .  In particular, for $S\in \mathcal{O}_{\cyc}(X)$ we may consider the automorphism
\[
\hat{S}\coloneqq e^{\{S,-\}},
\]
and $\hat{S}^*\omega=\omega$.
\end{proposition}
\subsubsection{}
Strictification of units in the non-cyclic setting is well known by work of Seidel and Lef\`evre-Hasegawa, but in the cyclic setting there appears to be no previous treatment.
\begin{proposition}[Strictification of cyclic units]
\label{StrictProp}
Let $\mathscr{A}$ be a cyclic minimal $A_{\infty}$-category over a field of characteristic zero, such that for each pair $i,j\in \Ob(\mathscr{A})$ the graded vector space $\mathscr{A}(i,j)$ is concentrated in non-negative cohomological degrees, or the potential $\mathcal{W}$ has cohomological degree less than four \footnote{An earlier version of this paper left out this final assumption, without which the proposition is false.  Thanks are due to Lino Amorim for pointing this out.}.  Then there is a cyclic $A_{\infty}$-isomorphism $\mathscr{A}\rightarrow \mathscr{A}'$, where $\mathscr{A}'$ is unital.
\end{proposition}
Recall that we are calling an $A_{\infty}$-category unital if it contains \textit{strict} units, hence the ``strictification'' in the above proposition.  The proof combines the ideas of \cite[Sec.3.2.1]{KLH}, \cite{Kaj07} with \cite{Ren16}.  We give the details here.
\begin{proof}[Proof of Proposition \ref{StrictProp}]
Let $\mathscr{A}'$ be the usual category underlying $\mathscr{A}$ (since we assume that $\mathscr{A}$ is minimal, associativity for the compositions $\Am_2$ follow from the relation \eqref{eq:A_infinity_relation_m} with $n=3$).  The usual cyclic Hochschild complex of the category $\mathscr{A}'$ is equal to
\begin{equation}
    \label{dgla_in}
(\mathcal{O}_{\cyc}(X),\{ \mathcal{W}_3,-\}).
\end{equation}
An element $\mathcal{W}=\mathcal{W}_3+\mathcal{W}_{\geq 4}$ determines a cyclic $A_{\infty}$-structure if and only if $\{\mathcal{W},\mathcal{W}\}=0$, i.e. if $\{\mathcal{W}_3, \mathcal{W}_{\geq 4}\}+1/2\{\mathcal{W}_{\geq 4},\mathcal{W}_{\geq 4}\}=0$.  In other words, cyclic $A_{\infty}$-structures correspond to Maurer--Cartan elements of the differential graded Lie superalgebra \eqref{dgla_in}.
\smallbreak
Note that unitality of the $A_{\infty}$-structure determined by $\mathcal{W}$ is equivalent to the condition that $\mathcal{W}_{\geq 4}\in\ol{\mathcal{O}}_{\cyc}(X)$.  The inclusion
\[
 G\colon \ol{\mathcal{O}}_{\cyc}(X)\rightarrow\ker\left(\mathcal{O}_{\cyc}(X)\rightarrow  \mathcal{O}_{\cyc}(k^{\Ob(\mathscr{A})})\right)
\]
is a quasi-isomorphism (see e.g. \cite[Sec.6.2]{Bro98}).  First we show that we can transform $\mathcal{W}$ so that $\mathcal{W}_4\in \ol{\mathcal{O}}_{\cyc}(X)$.  Under our degree assumptions, $\mathcal{W}_{n+1}$ lies in the kernel on the right hand side; note that $\mathcal{O}^{\geq d}_{\cyc}(k^{\Ob(\mathscr{A})})$ is concentrated in cohomological degrees $\geq d$.  Since $G$ is a quasi-isomorphism, there exists a $S_3$ such that 
\[
\mathcal{W}_{4}+\{S_3,\mathcal{W}_3\}\in \ol{\mathcal{O}}_{\cyc}(X)
\]
and we consider the automorphism $\hat{S}_3$.  Collecting the lowest order terms, we find
\[
\hat{S}_3(\mathcal{W})=\mathcal{W}_3+\mathcal{W}_{4}+\{S_3,\mathcal{W}_3\}+\textrm{ higher order terms}.
\]
so that, replacing $\mathcal{W}$ with $\hat{S}_3(\mathcal{W})$ we may assume that $\mathcal{W}_4\in\ol{\mathcal{O}}_{\cyc}(X)$.  Next assume that $\mathcal{W}_4,\ldots,\mathcal{W}_n\in \ol{\mathcal{O}}_{\cyc}(X)$.  Then as above, we find $S_n$ such that $\mathcal{W}_{n+1}+\{S_n,\mathcal{W}_3\}\in \ol{\mathcal{O}}_{\cyc}(X)$, and apply $\hat{S}_n$ to $\mathcal{W}$ to furthermore assume that $\mathcal{W}_{n+1}\in \ol{\mathcal{O}}_{\cyc}(X)$.  Applying infinitely many of these automorphisms (which makes sense because we work formally, and each one is the identity operator up to increasingly high order), we transform $\mathcal{W}$ so that $\mathcal{W}_{\geq 4}\in \ol{\mathcal{O}}_{\cyc}(X)$ while leaving $\omega$ invariant by Proposition \ref{KajSym}, as required.
\end{proof}

\subsubsection{Formality}
After this detour through $A_{\infty}$-categories and noncommutative differential calculus, we finally arrive at the following
\begin{corollary}
\label{form_cor}
Let $\mathscr{A}$ be a dg category equipped with a left 2CY structure over a field of characteristic zero, and let $\mathcal{F}_1,\ldots,\mathcal{F}_r$ be a $\Sigma$-collection of $\mathscr{A}$.  Then the full sub dg category of $\mathscr{A}$ containing the objects $\mathcal{F}_1,\ldots,\mathcal{F}_r$ is quasi-isomorphic to an ordinary graded category, i.e. it is formal.
\end{corollary}
\begin{proof}
By Proposition~\ref{BDthm} the full subcategory containing the objects $\CF_1,\ldots,\CF_r$ has a right 2CY structure, i.e. if we choose an $A_{\infty}$ minimal model $\mathscr{A}$ for this category, there is a nondegenerate element $\eta\in \CHC_{\bullet}(\mathscr{A})^{\vee}$ determining a bimodule isomorphism $\mathscr{A}^{\vee}\rightarrow \mathscr{A}[2]$ which we may furthermore assume to be a strict morphism by Lemma \ref{ChoLemma}.  In other words, $\eta$ defines a cyclic pairing on $\mathscr{A}$.  Note that for right 2CY categories the cohomological degree of the potential $\mathcal{W}$ is one.  By Proposition \ref{StrictProp} there is a cyclic $A_{\infty}$-isomorphism $\mathscr{A}\rightarrow \mathscr{B}$, where $\mathscr{B}$ is unital.  Then by Lemma \ref{rigLem} we deduce that $\mathscr{B}$ is formal.
\end{proof}
Using this corollary, in \S \ref{sec:applications} we deduce formality for Yoneda algebras of $\Sigma$-collections of complexes of coherent sheaves on smooth surfaces satisfying $\omega_S\cong\mathcal{O}_S$, objects in e.g. Kuznetsov components, and all of the other (left) 2CY categories we introduced at the start of this paper.

\subsubsection{A brief history of the formality problem}
\label{prev_work_sec}
This formality result has quite an involved history, which we try to summarise here.  Firstly, there are many incarnations of the result for the Yoneda algebra of a semisimple representation of a 2CY algebra.  As pointed out by Bocklandt, Galluzzi and Vaccarino \cite{Bo16}, this result follows from Koszul duality arguments and \cite[Theorem 11.2.1]{VdB15}; despite the apparent exclusion of the $n=2$ case from Van den Bergh's theorem on nCY algebras in [ibid].  This observation, along with deformation-theoretic arguments, was used to understand formal neighbourhoods in the coarse moduli space of representations of 2-Calabi--Yau algebras in \cite{Bo16}, with fuller details (including but not limited to the \'etale neighbourhood theorem for the coarse moduli space) provided in \cite{SchKa19}.  These results on coarse moduli spaces in turn generalise the known results on coarse moduli spaces of representations of preprojective algebras due to Crawley--Boevey \cite{CB03}.  For Higgs bundles and representations of fundamental groups, formality dates back to classical results \cite{Del75,Go88,Si92} and is a crucial tool in the subject (e.g. in Simpson's isosingularity theorem).
\smallbreak
Moving to the algebraic geometry context, there is a parallel literature on essentially the same problem of formality and formal neighbourhoods.  For instance a very special case of our formality result states that for $\mathcal{F}$ a polystable coherent sheaf on a K3 surface, the Yoneda algebra of $\mathcal{F}$ is formal; this was originally a conjecture of Kaledin and Lehn \cite{KaLe07}, proved recently~\cite{MR3942159}, extending~\cite{AS18}.

\smallbreak

Once expressed in the correct language of 2CY categories, and armed with strictification of units in cyclic categories, formality becomes straightforward, following straight from Lemma~\ref{rigLem}.  This lemma, in turn, seems to be well-known to the experts; it is equivalent to the statement that a unital 2CY category (with all nonunital homomorphisms in strictly positive degrees) is determined by its Ext quiver, a result proved by Jie Ren \cite{Ren16} using the same deformation theoretic arguments we used to prove strictification of units in the cyclic context, see also \cite{Tu14}.  

\smallbreak

The approach to formality here is inspired by the well-established approach to the Hodge-theoretic \textit{explicit} construction of minimal models of full subcategories $\mathscr{A}'$ of simple-minded collections in categories of coherent sheaves on a $d$-dimensional Calabi--Yau variety $X$.  The ``simple-mindedness'' condition is given by conditions (1) and (2) of Lemma \ref{rigLem}, while the Calabi--Yau property of $X$ guarantees that property (3) also holds, except with the shift by two replaced by shift by $d$.  Then we obtain an explicit dg model for $\mathscr{A}'$ by resolving each of the $\mathcal{F}_i$ by a complex of vector bundles $\mathcal{V}_i^{\bullet}$, and after picking a Hermitian metric on each of the vector bundles $\mathcal{V}_i^n$, along with a K\"ahler metric on $X$, we obtain a degree $-1$ operator $d^*$ on each of the homomorphism spaces $\Hom_{\Dub_{\dg}(\Mod^k)}((V^{\bullet}_i,d),(V^{\bullet}_j,d))$, adjoint to $d$, along with the Laplacian $\Delta=dd^*+d^*d$.  Then there is an explicit minimal $A_{\infty}$-algebra structure on the space of harmonic sections, defined by summing over trees as in \cite{MR1672242, KoSo02,chuang2009abstract}, and the resulting $A_{\infty}$ category is a minimal model for $\mathscr{A}'$.  By construction, the resulting category is unital, and is known to be cyclic with respect to the Serre duality pairing \cite{Kaj07}.  In particular, by Lemma \ref{rigLem} we can conclude that $\mathscr{A}'$ is formal by using this \textit{explicit} minimal model.

\section{Local structure of stacks of objects}
\label{local_struc}

\subsection{\'Etale local structure of good moduli spaces}
\label{nbhd_thm}
Throughout \S \ref{sec:2CY_local} and \S \ref{local_struc} we work over a field $k$ of characteristic 0.  In \S \ref{local_struc} we assume in addition that $k=\overline{k}$.
\subsubsection{}
We first recall the notion of a good moduli space.
\begin{definition}\label{def:good_moduli_space}\cite[Sec.1.2]{MR3237451}
    A quasi-compact morphism $\phi \colon \FX \to X$ from an Artin stack $\FX$ to an algebraic space $X$ is a \emph{good moduli space} if
    \begin{enumerate}
        \item The direct image functor $\phi_*$ on quasi-coherent sheaves is exact.
        \item The induced morphism of sheaves $\CO_X \to \phi_{*}\CO_{\FX}$ is an isomorphism.
    \end{enumerate}
\end{definition}
If $\FX$ is locally noetherian then by~\cite[Thm.~6.6]{MR3237451} good moduli spaces are universal for morphisms to arbitrary algebraic spaces, and hence unique when they exist.
Each of the moduli stacks in \S \ref{glob_intro} possess good moduli spaces, by \cite[Sec.7]{AHLH18}.  In fact the moduli schemes recalled in \S \ref{glob_intro} \textit{are} these good moduli spaces; this follows since they are constructed as GIT quotients, so in each case of the morphisms $p\colon \FM\rightarrow \CM$ there is a cover of $\FM$ by global quotient stacks $U_i/G_i$, so that the restriction of $p$ to $U_i/G_i$ is the affinization map.
\smallbreak

We work at the higher level of generality allowed by Definition \ref{def:good_moduli_space} since once we leave the world of GIT, not all moduli stacks of semistable objects in 2CY categories are known to be global quotients; for instance the moduli stack of Bridgeland semistable objects in a Kuznetsov component is not known to be such a quotient, but is known to possess a good moduli space, again by appeal to the general results of \cite{AHLH18} (see \cite[Part VI]{Betal19}).

\subsubsection{}
We will use a selection of the \'{e}tale-local structure theorems of Alper, Hall and Rydh~\cite{MR4088350}, which we collect here.  In each case the theorem that we quote is not the full theorem, but a weaker form that is sufficient for our current purposes.
\smallbreak
\begin{theorem}\cite[Thm.1.1]{MR4088350}
\label{AHR_thm1}
Let $\FX$ be a quasi-separated algebraic stack, locally of finite type, with affine stabilisers.
Let $x \in \FX(k)$ be a $k$-point and assume that the stabiliser $G_x$ is reductive\footnote{Recall that this is equivalent to being linearly reductive, since we work over a field of characteristic zero.  The original version of this theorem in \cite{MR4088350} is stated for arbitrary algebraically closed ground fields, and is for linearly reductive groups.}.
Then there exists an affine scheme $\Spec A$ with an action of $G_x$, a $k$-point $w \in \Spec A$ fixed by $G_x$ and an \'{e}tale morphism
\[
f \colon \bigl(\Spec(A)/G_x,w\bigr) \to (\FX,x)
\]
such that $\{w\}/G_x \cong f^{-1}(\{x\}/G_x)$.

\end{theorem}
\begin{remark}
This is the version of \cite[Thm.~1.1]{MR4088350} that is needed for the version of \cite[Thm. 4.19]{MR4088350} that we use.
The authors' original theorem is more general, and in particular also yields \'{e}tale-local presentations of $(\FX,x)$ as quotient stacks given by quotients of an affine scheme by certain subgroup schemes of the stabiliser $G_x$.
\end{remark}

\begin{theorem}\cite[Thm.4.12]{MR4088350}
\label{thm:etale_local_AHR}
Let $\FX$ be a Noetherian algebraic stack, let $p\colon\FX\rightarrow X$ be a good moduli space with affine diagonal.  If $x\in\FX(k)$ is a closed point with stabiliser $G_x$, then there is an affine $G_x$-equivariant scheme $Y=\Spec(A)$ with fixed point $t$, and a Cartesian diagram
\[
\xymatrix{
\ar[d]^r(Y/G_x,t)\ar[r]&(\FX,x)\ar[d]^{p}\\
(Y/\!\!/G_x,r(t))\ar[r]^u&(X,p(x))
}
\]
in which the horizontal morphisms are \'etale.
\end{theorem}

\subsubsection{}
If $p\colon \FX\rightarrow \CX$ is a good moduli space with affine diagonal, then the \textit{coherent completion} $\hat{\FX}_{x}$ of a closed point $\iota\colon x\hookrightarrow  \FX$ is isomorphic to the complete local stack (in the sense of \cite[Def.A.9]{MR4088350}) $(Y/G_x\times_{Y/\!\!/G_x}\mathrm{Spec}(\hat{A}_{r(t)}),t)$, with $Y$ as in Theorem \ref{thm:etale_local_AHR} and $\hat{A}_{r(t)}$ the completion of $A$ at the maximal ideal corresponding to $r(t)$.  The coherent completion is independent of choices made, up to a 1-isomorphism that is canonically defined up to canonical 2-isomorphism.  The coherent completion can be defined more generally, and exists under weaker conditions than those listed above (see \cite[Sec.4.5]{MR4088350}).
\smallbreak
The following is the generalisation of the famous Artin approximation theorem \cite[Cor.~2.6]{MR268188} to stacks.
\begin{theorem}\cite[Thm.4.19]{MR4088350}\label{StAA}
Let $\FX$ and $\FY$ be locally finite type quasi-separated stacks over $k$ with affine stabilisers. Let $x$ and $y$ be points of $\FX$ and $\FY$, respectively, with reductive stabiliser groups.  Suppose that there is an isomorphism $\hat{\FX}_x\cong \hat{\FY}_y$ of coherent completions, and let $G$ denote the stabiliser group of $x$.  Then there is a $G$-equivariant affine scheme $\Spec(A)$, a $G$-fixed point $w$ of $\Spec(A)$, and a diagram of pointed stacks
\[
\xymatrix{
&(\Spec(A)/G,w)\ar[dl]\ar[dr]\\
(\FX,x)&&(\FY,y)
}
\]
in which the morphisms are \'etale.
\end{theorem}

We will also use the following form of Luna's fundamental lemma \cite[p.73]{Lu73}.

\begin{proposition}\cite[Prop.4.13]{MR4088350}
\label{Luna_lemma}
Let $f\colon\FX\rightarrow \FY$ be an \'etale, separated and representable morphism of Noetherian algebraic stacks such that there is a commutative diagram
\[
\xymatrix{
\ar[d]^{\pi_{\FX}}\FX\ar[r]^{f}&\FY\ar[d]^{\pi_{\CY}}\\
X\ar[r]&Y
}
\]
where $\pi_{\FX}$ and $\pi_{\FY}$ are good moduli spaces.
Let $x \in \FX(k)$ be a closed point.
If $f(x) \in \FY(k)$ is closed and f induces an isomorphism of stabiliser groups at $x$, then there exists an open neighbourhood $U \subset X$ of $\pi_{\FX}(x)$ such that $U \to X \to Y$ is \'{e}tale and $\pi_{\FX}^{-1}(U) \cong U \times_{Y} \FY$.  In other words, in the following diagram the rectangle containing the curved horizontal arrows is Cartesian and the curved horizontal arrows are \'{e}tale:
\vspace{7pt}
\[
\xymatrix{
\pi_{\FX}^{-1}(U)\ar@/^1pc/[rr]\ar[d]\ar@{^{(}->}[r]&\FX\ar[d]^{\pi_{\FX}}\ar[r]_{f}&\FY\ar[d]^{\pi_{\CY}}\\
U\ar@/_1pc/[rr]\ar@{^{(}->}[r]&X\ar[r]&Y.
}
\]
\end{proposition}
\vspace{2pt}
\subsection{Direct image mixed Hodge modules}
As explained in \S \ref{sec:mixed_hodge_structures}, the current state of the art does not provide a definition of $p_!\underline{\BQ}_{\FM}$ as an object in the bounded above derived category of mixed Hodge modules on $\CM$, if $p\colon \FM\rightarrow \CM$ is an arbitrary morphism from a finite type stack to a scheme.  
\smallbreak
We partially remedy this gap in the theory by explaining how to define canonical mixed Hodge modules $\Ho^i\!p_!\underline{\BQ}_{\FM}$ and $\Ho^i\!p_*\underline{\BQ}_{\FM}$ along with canonical isomorphisms
\begin{align*}
\rat\left(\Ho^i\!p_!\underline{\BQ}_{\FM}\right)\cong &{}^{\Fp}\!\Ho^i\!p_!\BQ_{\FM}\\
\rat\left(\Ho^i\!p_*\underline{\BQ}_{\FM}\right)\cong &{}^{\Fp}\!\Ho^i\!p_*\BQ_{\FM}
\end{align*}
for certain good moduli spaces $p\colon \FM\rightarrow \CM$.  Note that we cannot always simply apply $\Ho^i$ to the complexes $p_!\underline{\BQ}_{\FM}$ and $p_*\underline{\BQ}_{\FM}$ from \S \ref{stack_di}, since we defined these complexes under extra assumptions on $\FM$ (i.e. that it is exhausted by global quotient stacks).  

\sssct
We will need a small preparatory lemma.
\begin{lemma}
\label{point_lemma}
Let $\FX$ be a Noetherian algebraic stack, and let $p\colon\FX\rightarrow X$ be a good moduli space with affine diagonal.  Let $x$ be a closed point of $X$.  Then there is a closed point $\tilde{x}$ of $\FX$ such that $p(\tilde{x})=x$.
\end{lemma}
\begin{proof}
Good moduli spaces are stable under pullback, and so $\Gamma(\mathcal{O}_{\FX_x})\cong k$.  So $\FX_x$ is nonempty and admits an $A$-point, for some commutative algebra $A$.  Since $\FX_x$ is Noetherian, this $A$ point factors through a $B$-point, for $B$ finite type.  Since $B$ is finite type and $k=\overline{k}$, there is a surjection $B\rightarrow k$, providing a $k$-point of $\FX_x$.
\end{proof}
\subsubsection{}
Returning to the morphism $p\colon \FM\rightarrow \CM$, by Theorem \ref{thm:etale_local_AHR} and Lemma \ref{point_lemma} we can find a set $\{Y_j\}_{j\in J}$ of affine varieties and a set $\{G_j\}_{j\in J}$ of reductive algebraic groups, with each $Y_j$ a $G_j$-scheme, and an \'etale cover $\{u_j\colon Y_j/\!\!/G_j\rightarrow \CM\}_{j\in J}$ of $\CM$ along with Cartesian diagrams
\begin{equation}
\label{et_cov}
\xymatrix{
\ar[d]_{p_j}Y_j/G_j\ar[r]&\FM\ar[d]^{\pi}\\
Y_j/\!\!/G_j\ar[r]^{u_j}&\CM.
}
\end{equation}
In particular, the stacks $Y_j/G_j$ provide an \'etale cover of $\FM$.  We use this open cover to construct $\Ho^i\!p_!\underline{\BQ}_{\FM}$.  The construction of $\Ho^i\!p_*\underline{\BQ}_{\FM}$ follows identical lines, and will not be used in this paper.  
\smallbreak
Since mixed Hodge modules form a stack for the \'etale topology, it is sufficient to construct mixed Hodge modules $\CG_j\in\Ob(\MHM(Y_j/\!\!/G_j))$ and gluing maps between them satisfying the cocycle condition.  Additionally, we may define a mixed Hodge module on an algebraic space, or more generally a Deligne-Mumford stack, as a mixed Hodge module on an atlas satisfying descent, so we will not impose the condition that $X$ is a variety in the following.  
\smallbreak
We set
\[
\CG_j=\Ho^i\!\left(p_{j,!}\ul{\BQ}_{Y_j/G_j}\right),
\]
defined as in \S \ref{sec:algebraic_MHMs}.  Pick $j'\neq j$ with $j,j'\in J$.  Replacing $Y_j$ by $Y_j\times G_{j'}$, and $Y_{j'}$ by $Y_{j'}\times G_j$, we may assume that $G_j=G_{j'}$.  Let $y\in Y_j/\!\!/G_j\cap Y_{j'}/\!\!/G_j$.  Pick affine $U'\ni y$ with $U'\subset Y_j/\!\!/G_j\cap Y_{j'}/\!\!/G_j$.  Since $Y_j\rightarrow Y_j/\!\!/ G_j$ is affine, there is a $G_j$-invariant open affine subscheme $U\subset Y_j$ with $U'= U/\!\!/G_j$.  Write $p\colon U/G\rightarrow U/\!\!/ G_j$ for the affinization map.  Then the required gluing follows from the claim that there is a natural isomorphism
\[
\Ho^i\!\left(p_{j,!}\ul{\BQ}_{Y_j/G_j}\right)\lvert_{U'}\cong \Ho^i\!\left(p_{!}\ul{\BQ}_{U/G_j}\right).
\]
This follows by construction (see \S \ref{sec:algebraic_MHMs}), and by base change applied to the morphism $Y_j\times A'/G_j \rightarrow Y_j/\!\!/G_j$ defined as in \eqref{ApproxMap}.  We summarise the features of this construction in the following proposition.
\begin{defproposition}
\label{dfp}
Let $\FX$ be a Noetherian stack, and let $p\colon \FX\rightarrow X$ be a good moduli space with affine diagonal.  Then for all $i\in\BZ$ there exist mixed Hodge modules $\Ho^i\!p_!\underline{\BQ}_{\FX}$ and $\Ho^i\!p_*\underline{\BQ}_{\FX}$ in $\MHM(X)$ satisfying the following conditions:
\begin{enumerate}
\item
There are canonical isomorphisms of perverse sheaves on the algebraic space $X$
\[
\rat\left(\Ho^i\!p_!\underline{\BQ}_{\FX}\right)\cong {}^{\Fp}\!\Ho^i\!p_!\BQ_{\FX}\quad\quad \rat\left(\Ho^i\!p_*\underline{\BQ}_{\FX}\right)\cong {}^{\Fp}\!\Ho^i\!p_*\BQ_{\FX}
\]
\item
For all diagrams \eqref{et_cov} there are canonical isomorphisms 
\[
u_j^*\Ho^i\!p_!\underline{\BQ}_{\FX}\cong \Ho^i\!p_{j,!}\underline{\BQ}_{Y_j/G_j}\quad\quad u_j^*\Ho^i\!p_*\underline{\BQ}_{\FX}\cong \Ho^i\!p_{j,*}\underline{\BQ}_{Y_j/G_j}
\]
where $p_{j,!}\underline{\BQ}_{Y_j/G_j}$ and $p_{j,*}\underline{\BQ}_{Y_j/G_j}$ are the complexes of mixed Hodge modules constructed by approximation to the Borel construction in \S \ref{heart_di}.
\end{enumerate} 
Moreover, these two conditions define $\Ho^i\!p_!\underline{\BQ}_{\FX}$ and $\Ho^i\!p_*\underline{\BQ}_{\FX}$ up to canonical isomorphism.
\end{defproposition}

\subsection{\'Etale local structure of 2CY categories}

In this section $\bm{\FM}_{\mathscr{C}}$ will denote the derived stack of objects of a dg category $\mathscr{C}$.  See the appendix for a little background on this stack.  We always restrict attention to some open substack $\bm{\FM}$ that is at least 1-Artin, meaning that its cotangent complex has cohomological amplitude bounded above at $1$.  Moreover, our results will concern the classical truncation $\FM=t_0(\bm{\FM})$, which will be an Artin stack in the classical sense.

\begin{lemma}
\label{techy}
The stack $\FM$ has separated diagonal and affine stabilisers.
\end{lemma}
\begin{proof}
Let $A$ be a commutative $k$-algebra.  For $\xi_1,\xi_2$ two families of objects of $\mathscr{C}$ over $\Spec(A)$ (equivalently, pseudoperfect $A\otimes \mathscr{C}$-modules) we denote by the same symbols $\xi_1,\xi_2$ the associated morphisms $\Spec(A)\rightarrow \FM$.  Then there is a Cartesian diagram
\[
\xymatrix{
\FM\ar[r]^{\Delta}&\FM\times\FM\\
\ar@{_{(}->}[d]^j\Isom_{A\otimes\mathscr{C}}(\xi_1,\xi_2)\ar[u]\ar[r]^b&\Spec(A)\ar[u]^{(\xi_1,\xi_2)}\\
\Tot_A(\mathcal{H}\mathrm{om}_{A\otimes\mathscr{C}}(\xi_1,\xi_2))\ar[ur]^l,
}
\]
The morphism $l$ is obviously affine, and so $b$ is separated.  For $\iota\colon\Spec(K)\rightarrow \Spec(A)$ a geometric point, consider the finite collection of homomorphisms
\[
\epsilon_i\colon \iota^*\Tot_A(\mathcal{H}\mathrm{om}_{A\otimes\mathscr{C}}(\xi_1,\xi_2))\rightarrow \Hom(\Ho^i((\xi_1)_K),\Ho^i((\xi_2)_K))
\]
where we let $i$ range over the collection of $i$ for which the target is not zero (this collection is finite by perfectness over $K$).  Then $\iota^*j$ is the inclusion of the intersection of affines determined by considering the preimages of $\mathrm{Isom}(\Ho^i((\xi_1)_K),\Ho^i((\xi_2)_K))$ under $\epsilon_i$.
\end{proof}

\begin{proposition}
\label{app_prop}
Let $\mathscr{C}$ and $\mathscr{D}$ be two dg categories with stacks of objects $\bm{\FM}_{\mathscr{C}}$ and $\bm{\FM}_{\mathscr{D}}$.  Let $\BS=\{\CF_1,\ldots,\CF_r\}$ be a simple-minded collection in $\Ob(\mathscr{C})$, i.e. the $\CF_i$ are pairwise nonisomorphic, there are no $\Ext^n$s between objects of $\BS$ for $n<0$, and the only $\Ext^0$s are scalar multiples of identity morphisms.  Let $\BS'=\{\CF'_1,\ldots,\CF'_r\}$ be a simple-minded collection in $\mathscr{D}$, and assume that there is a quasi-equivalence between the full subcategory of $\mathscr{C}$ containing the objects $\BS$, and the full subcategory of $\mathscr{D}$ containing the objects $\BS'$, sending $\CF_i$ to $\CF'_i$ for each $i\leq r$.
\smallbreak
Pick $\mathbf{d}\in\BN^r$ and set
\[
\CF=\bigoplus_{i\leq r}\CF_i^{\oplus\mathbf{d}_i};\quad\quad
\CF'=\bigoplus_{i\leq r}\CF'^{\oplus \mathbf{d}_i}_i.
\]
Let $x,y$ be closed $k$-points of 1-Artin open substacks $\bm{\FM}\subset \bm{\FM}_{\mathscr{C}}$ and $\bm{\FM}'\subset \bm{\FM}_{\mathscr{D}}$ respectively, representing $\CF$ and $\CF'$.  Then there is an isomorphism of coherent completions
\[
(\widehat{t_0(\bm{\FM})}_x,\hat{x})\simeq (\widehat{t_0(\bm{\FM}')}_y,\hat{y}).
\]
\end{proposition}
We defer the proof of this technical lemma to Appendix \ref{Appendix}, where we also introduce a little of the terminology from derived algebraic geometry that we are using here.  

\subsubsection{}
Putting this result on coherent completions together with the results in \S \ref{nbhd_thm} we can finally prove the following.
\begin{theorem}
\label{loc_str_thm}
Let $p\colon \FM\rightarrow \CM$ be a good moduli space of locally finite type, with reductive stabiliser groups.  Assume we are given an isomorphism $\FM\cong t_0(\bm{\FM})$, where $\bm{\FM}\subset \bm{\FM}_{\mathscr{C}}$ is an open substack of the stack of objects of a dg category $\mathscr{C}$.  Let $x$ be a closed $k$-valued point of $\FM$ corresponding to an object
\[
\CF=\bigoplus_{i\leq n}\CF_i^{\oplus \mathbf{d}_i}.
\]
Assume that the full subcategory of $\mathscr{C}$ containing $\CF_1,\ldots,\CF_n$ carries a right 2CY structure, and that $\{\CF_i\}_{i\leq n}$ is a $\Sigma$-collection.  Let $Q$ be a quiver such that the Ext-quiver of $\{\CF_i\}_{i\leq n}$ is the double of $Q$.  Then there is an affine $\Gl_{\mathbf{d}}$-variety $H$, containing a fixed point $y$, and a commutative diagram
\begin{equation}\label{eq:neighbourhood_diagram}
\xymatrix{
(\FM_{\mathbf{d}}(\Pi_Q),0_{\mathbf{d}})\ar[d]^{\JH_{\mathbf{d}}}&(H/\Gl_{\mathbf{d}},y)\ar[d]^q\ar[r]^-{\ol{l}}\ar[l]_-{\ol{j}}&(\FM,x)\ar[d]^{p}\\
(\CM_{\mathbf{d}}(\Pi_Q),0_{\mathbf{d}})& (H/\!\!/\Gl_{\mathbf{d}},y)\ar[l]_-j\ar[r]^-l&(\CM,p(x))
}
\end{equation}
in which the horizontal morphisms are \'etale, and the squares are Cartesian.
\end{theorem}

The assumption that $p \colon \FM \to \CM$ is a good moduli space allows us to appeal to the results of Alper--Hall--Rydh summarised above.  This is a mild assumption in view of the work of Alper--Halpern-Leistner--Heinloth proving the existence of good moduli spaces in great generality~\cite{AHLH18}.

\begin{proof}[Proof of Theorem \ref{loc_str_thm}]
First note that $\Gl_{\mathbf{d}}$ is the stabiliser of the point $x$.
Since $\pi$ is a good moduli space with separated diagonal and affine stabilisers, and since every closed point of $\FM$ has a reductive stabiliser, $\pi \colon \FM \to \CM$ has affine diagonal by \cite[Cor.~13.11]{alper2021etale}.
By Theorem~\ref{thm:etale_local_AHR}, we may find a $\Gl_{\mathbf{d}}$-equivariant affine scheme $\Spec A$ and a Cartesian diagram of pointed stacks
\[
\xymatrix{
(\Spec A/\Gl_{\mathbf{d}},y')\ar[d]^{\pi_A}\ar[r]^-{\phi}& (\FM\ar[d]^{p},x)\\
(\Spec A/\!\!/\Gl_{\mathbf{d}},y')\ar[r]&(\CM,x)
}
\]
in which the horizontal morphisms are \'etale.  Since $\phi$ is \'etale it induces an isomorphism between coherent completions $(\widehat{\FN}_{y'},\hat{y'})\cong (\widehat{\FM}_x,\hat{x})$, where we write $\FN=\Spec A/\Gl_{\mathbf{d}}$.

\smallbreak 
Consider the collection $S_1,\ldots,S_r$ of 1-dimensional simple nilpotent modules for the preprojective algebra $\Pi_Q$.  By inspecting the Koszul dual of the derived preprojective algebra $\mathscr{G}_2(\mathbb{C}Q)$ defined in \eqref{dppa_def} (or via a gratuitous application of the formality theorem), the subcategory $\mathscr{D}$ of $\Ddg(\Mod_{{\mathscr{G}_2(\mathbb{C}Q)}})$ spanned by these modules is formal.  Moreover by construction, we have equalities $\dim(\Ext^n(\CF_i,\CF_j))=\dim(\Ext^n(S_i,S_j))$ for all $n,i,j$.  Since the multiplication operation $\Am_2$ on a right 2CY category is uniquely determined by the pairing $\langle\bullet,\bullet\rangle$ and cyclic invariance, we deduce that there is a quasi-isomorphism of categories between $\mathscr{D}$ and the full subcategory of $\mathscr{C}$ containing $\CF_1,\ldots,\CF_r$.  
\smallbreak
By Proposition \ref{app_prop} there is an isomorphism of coherent completions 
\[
(\widehat{\FM_{\mathbf{d}}(\Pi_Q)}_{0_{\mathbf{d}}},\hat{0}_{\mathbf{d}})\cong (\widehat{\FN}_{y'},\hat{y}').
\]
By Artin approximation for stacks (Theorem \ref{StAA}) we obtain a diagram with \'etale horizontal morphisms in the top row
\[
\xymatrix{
(\FM_{\mathbf{d}}(\Pi_Q),0_{\mathbf{d}})\ar[d]^{\JH_{\mathbf{d}}}&\ar[l] (\Spec(B)/\Gl_{\mathbf{d}},y)\ar[r]\ar[d]^{\pi_B}& (\FN,y')\ar[d]^{\pi_A} \\
(\CM_{\mathbf{d}}(\Pi_Q),0_{\mathbf{d}})&\ar[l] (\Spec(B)/\!\!/\Gl_{\mathbf{d}},y)\ar[r] &(\Spec(A)/\!\!/\Gl_{\mathbf{d}},y').
}
\]
The arrows in the bottom row are uniquely determined by the universal property of the affinization.  By Proposition \ref{Luna_lemma}, we may shrink $\Spec(B)/\!\!/\Gl_{\mathbf{d}}$ to some affine subscheme $Y$ (still containing $y$) such that in the diagram
\[
\xymatrix{
(\FM_{\mathbf{d}}(\Pi_Q),0_{\mathbf{d}})\ar[d]^{\JH_{\mathbf{d}}}&\ar[l] (\pi^{-1}(Y),y)\ar[r]\ar[d]^{q}& (\FN,y')\ar[r]\ar[d]^{\pi_A} &(\FM,x)\ar[d]^{p}\\
(\CM_{\mathbf{d}}(\Pi_Q),0_{\mathbf{d}})&\ar[l] (Y,y)\ar[r] &(\Spec(A)/\!\!/\Gl_{\mathbf{d}},y')\ar[r]&(\CM,x)
}
\]
all horizontal arrows are \'etale, and all squares are Cartesian.  There is an isomorphism
\[
\pi^{-1}(Y)\cong \overline{Y}/\Gl_{\mathbf{d}}
\]
where 
\[
\overline{Y}=(Y\times_{\Spec(A)/\!\!/\Gl_{\mathbf{d}}} \Spec(A))
\]
is affine, since $\Spec(A)\rightarrow \Spec(A)/\!\!/\Gl_{\mathbf{d}}$ is.
\end{proof}

\subsubsection{Relation with Halpern-Leistner's neighbourhood theorem}
It would be highly desirable, for applications to DT theory, to have a derived enhancement of Theorem \ref{loc_str_thm}.  Precisely, we would like that for each closed $x\in \FM$ satisfying the conditions of the theorem, there is an equivalence of derived stacks $\bm{\FU}\simeq \bm{\FU'}$ with $\bm{\FU}\rightarrow \bm{\FM}$ an \'etale neighbourhood of $x$ and $\bm{\FU'}\rightarrow \bm{\FM}_{\mathscr{D}}$ an \'etale neighbourhood of $0_{\dd}$, with $\mathscr{D}=\Perf_{\dg}(\mathscr{G}_2(\mathbb{C}Q))$.  In \cite[Thm.4.2.3]{halpern2020derived} Halpern-Leistner proves a result in this direction for more general 0-shifted symplectic (derived) stacks.  In the setup of that theorem, the open substack of $\bm{\FM}_{\mathscr{D}}$ corresponding to $\dd$-dimensional $\Pi_Q$-modules is obtained by setting $R=\Gamma(\CO_{\BA_{\mathbf{d}}(\overline{Q})})$ and taking the co-moment map defined by \eqref{comoment} and the Killing form.  The desired derived enhancement of Theorem \ref{loc_str_thm} is the statement that (after replacing $R$ by an \'etale open neighbourhood $R'$) we have a morphism as in the statement of \cite[Thm.4.2.3]{halpern2020derived}, for the \textit{given} co-moment map.  Note that the given co-moment map is a co-moment map in a stronger sense than \cite[Def.4.2.1]{halpern2020derived}.

\section{The Borel--Moore homology of the stack of objects in a 2CY category}\label{sec:purity_perverse}

\subsection{The local purity theorem}

Throughout \S \ref{sec:purity_perverse} we fix the base field $k=\BC$.  Our goal is to use the description of \'{e}tale neighbourhoods of moduli stacks of objects in 2CY categories from \S \ref{local_struc} to prove Theorem \ref{thm:purity_relative}.

\begin{theorem}
\label{loc_pur_thm}
Let $\mathscr{C}$ be a dg category, let $\bm{\FM}\subset \bm{\FM}_{\mathscr{C}}$ be a 1-Artin open substack of the stack of objects in $\mathscr{C}$, let $\FM=t_0(\bm{\FM})$ denote the classical truncation, and assume that there exists a good moduli space $p\colon \FM\rightarrow \CM$.  We assume that $\FM$ satisfies the following property:
\begin{enumerate}
    \item[*]
    Let $x$ be a closed $\BC$-valued point of $\FM$.  Then $x$ corresponds to an object $\bigoplus_{i\leq n}\CF_i^{\bigoplus \mathbf{d}_i}$ with $\{\CF_i\}_{i\in I}$ a $\Sigma$-collection in $\mathscr{C}$, and the full subcategory of $\mathscr{A}$ containing $\CF_i$ for $i\in I$ carries a right 2CY structure.
\end{enumerate}
Then for every $i\in \BZ$ the mixed Hodge module $\Ho^i\!p_!\ul{\BQ}_{\FM}\in\MHM(\Msp)$ is pure of weight $i$, and
\[
(\Ho^{i}\!p_!\ul{\BQ}_{\FM})\lvert_{p(x)}=0
\]
for $x$ representing an object $\CF$ and $i> \chi_{\mathscr{C}}(\CF,\CF)$.
\end{theorem}
\begin{remark}
If $\mathscr{C}$ is a left 2CY category, and closed points of $\FM$ correspond to semisimple objects of $\mathscr{A}$, then (*) holds automatically by Proposition \ref{BDthm}.  Therefore, Theorem \ref{thm:purity_relative} follows as a corollary of Theorem \ref{loc_pur_thm}.
\end{remark}
\begin{remark}
There are, however, examples in which $\bm{\FM}$ is not an open substack of the moduli stack of objects in a left 2CY category, but $\bm{\FM}$ still satisfies the conditions of Theorem \ref{loc_pur_thm}.  For instance, one may set $\bm{\FM}$ to be the open substack, points of which correspond to coherent sheaves with zero-dimensional support, of the derived stack of coherent sheaves on a smooth surface $S$ with $\mathcal{O}_S\ncong \omega_S$.
\end{remark}
\begin{remark}
One of the main results of the paper \cite{AHLH18} is that good moduli spaces exist very generally.  If a stack $\FM$ has a coarse moduli space $\Msp$ for which there is an \'etale cover $\coprod U_i\rightarrow \Msp$ and commutative diagrams
\[
\xymatrix{
[\Spec(A_i)/G]\ar[d]\ar[r]^-{\cong}&U_i\times_{\Msp} \FM\ar[d]\\
\Spec(A_i)/\!\!/G\ar[r]^-{\cong}&U_i
}
\]
then a good moduli space exists, and this takes care of all of the examples from \S \ref{exam_2CY}, that are described in terms of classical GIT.  
\end{remark}

\begin{proof}[Proof of Theorem \ref{loc_pur_thm}]
By Lemma \ref{techy} and \cite[Cor.~13.11]{alper2021etale}, the morphism $p$ has affine diagonal, and so we may define $\Ho^{i}\!p_!\ul{\BQ}_{\FM}$ via Definition/Proposition \ref{dfp}.  The conditions in the statement of the theorem also imply that all of the preconditions on $\FM$ of Theorem~\ref{loc_str_thm} hold.  By Theorem~\ref{loc_str_thm}, we can find an \'etale cover of $\CM$ by coarse moduli schemes $H/\!\!/\Gl_{\mathbf{d}}\xrightarrow{l}\CM$ (for various $\mathbf{d}$) such that $l$ fits into the diagram~\eqref{eq:neighbourhood_diagram} (for various $Q$).
\smallbreak
By Lemma~\ref{loc_pur_lem}, the purity statement reduces to showing that $l^*p_!\ul{\BQ}_{\FM}$ is pure (continuing with the notation of diagram \eqref{eq:neighbourhood_diagram}).  Applying base change to the Cartesian squares in \eqref{eq:neighbourhood_diagram}, we obtain isomorphisms
\begin{align*}
    l^*\Ho^i\!p_!\ul{\BQ}_{\FM}\cong&\Ho^i\!q_!\ul{\BQ}_{H/\Gl_{\mathbf{d}}}\\
    \cong&j^*\Ho^i\!\JH_{\mathbf{d},!}\ul{\BQ}_{\FM_{\mathbf{d}}(\Pi_Q)}.
\end{align*}
By Theorem \ref{pp3_thm}, the complex $\JH_{\mathbf{d},!}\ul{\BQ}_{\FM_{\mathbf{d}}(\Pi_Q)}\in\Dblf(\MHM(\CM_{\mathbf{d}}(\Pi_Q)))$ is pure, so that applying Lemma \ref{loc_pur_lem} once more, we deduce that $j^*\Ho^i\!\JH_{\mathbf{d},!}\ul{\BQ}_{\FM_{\mathbf{d}}(\Pi_Q)}$ is indeed pure of weight $i$, as required.

\smallbreak
For the statement regarding truncation functors, it is sufficient to prove the result for the underlying complex of perverse sheaves $p_!\BQ_{\FM}$.  After shrinking to an analytic neighbourhood of $0_{\mathbf{d}}$ and $p(x)$ in Theorem \ref{loc_str_thm}, we can assume that the morphisms $j$ and $l$ of diagram \eqref{eq:neighbourhood_diagram} are isomorphisms, so that $\overline{l}$ and $\overline{j}$ are too.  Let $\mathcal{F}=\bigoplus_{i\leq r}\CF_i^{\oplus \mathbf{d}_i}$ and let the quiver $Q$ be as in the proof of Theorem \ref{loc_str_thm}.  We claim that there is an identity
\begin{equation}
\label{euler_compare}
2\chi_Q(\mathbf{d},\mathbf{d})=\chi_{\mathscr{C}}(\mathcal{F},\mathcal{F}).
\end{equation}
By Corollary \ref{form_cor} there is an equivalence of categories between the full pretriangulated subcategory of $\mathscr{C}$ generated by $\mathcal{F}_1,\ldots,\mathcal{F}_r$ and the full pretriangulated subcategory of $\mathscr{E}=\Dub_{\dg}(\fdmod^{\Pi_Q})$ generated by the simple modules $S_1,\ldots,S_r$, sending $\mathcal{F}_i$ to $S_i$.  Set $S=\bigoplus_{i\leq r}S_i^{\oplus\mathbf{d}_i}$.  Thus we have the identities
\begin{align*}
\chi_{\mathscr{C}}(\mathcal{F},\mathcal{F})=\chi_{\mathscr{E}}(S,S)=2\chi_Q(\mathbf{d},\mathbf{d})
\end{align*}
as required.  The vanishing result follows from \eqref{perv_bound} and \eqref{euler_compare}.
\end{proof}

\subsubsection{}
Let $x\in \FM$ be a regular closed point, corresponding to an object $\CF\in\Ob(\mathscr{C})$, and let $U\ni p(x)$ be a connected open subscheme of the smooth locus of $\CM$, where we assume that $p^{-1}(U)\rightarrow U$ is a $\B\BC^*$-gerbe.  The stack $p^{-1}(U)$ is $1-\chi_{\mathscr{C}}(\CF,\CF)$-dimensional, and so $\BQ_{\FM}[-\chi_{\mathscr{C}}(\cdot,\cdot)]$, restricted to a neighbourhood of $x$, lives in (perverse) cohomological degree $1$.  \'Etale locally around $p(x)$ we may write $p$ as a projection $\pi\colon U \times\pt/\BC^*\rightarrow U$, and 
\[
\pi_!\BQ_{U \times\pt/\BC^*}=\BQ_{U}[2]\oplus\BQ_{U}[4]\oplus\ldots.
\]
We thus find that, restricting to an analytic neighbourhood $U$ of $p(x)$, we have an isomorphism 
\begin{equation}
\label{nty}
p_!\BQ_{\FM}[-\chi_{\mathscr{A}}(\cdot,\cdot)]\lvert _U\cong \BQ_U[2-\chi_{\mathscr{A}}(\cdot,\cdot)]\oplus \BQ_U[4-\chi_{\mathscr{A}}(\cdot,\cdot)]\oplus\ldots.
\end{equation}
Observing that $\dim(U)=2-\chi_{\mathscr{A}}(\CF,\CF)$, we see that, restricted to $U$, the complex $p_!\BQ_{\FM}[-\chi_{\mathscr{A}}(\cdot,\cdot)]$ indeed decomposes as a direct sum of its perverse cohomology sheaves, and has vanishing strictly positive perverse cohomology.
\subsubsection{}
The zeroth perverse cohomology sheaf of \eqref{nty} is the constant perverse sheaf.  Around singular points of $\FM$, Theorem \ref{loc_pur_thm} is somewhat less trivial, and the zeroth perverse cohomology is more mysterious.  In general, the zeroth perverse cohomology at a point $\CF$ should be obtained by taking symmetric convolution powers of BPS sheaves corresponding to classes in $\mathrm{K}_0(\mathscr{A})$ adding up to that of $\CF$.
\subsubsection{}
One of the key components of the proof of local purity (Theorem \ref{pp3_thm}) is the dimensional reduction isomorphism in cohomology, from \cite[Appendix A]{MR3667216}.  In brief, the application we are making of this theorem is that there is an isomorphism between $\JH_{\mathbf{d},!}\ul{\BQ}_{\FM_{\mathbf{d}}(\Pi_Q)}$ and $\mathcal{F}$, where $\mathcal{F}$ is defined as follows.  First, consider the \textit{3}-Calabi-Yau category of $A=\Pi_Q[x]$-modules.  The classical moduli stack of $A$-modules carries a vanishing cycle mixed Hodge module $\phi_{\Tr(\widetilde{W})}$, as well as admitting a morphism $\widetilde{p}$ to $\CM_{\mathbf{d}}(\Pi_Q)$ defined by forgetting the action of $x$, and semisimplifying the resulting $\Pi_Q$-module.  Set $\CF=\widetilde{p}_!\phi_{\Tr(\widetilde{W})}$.  Purity comes from comparing the different constraints on the possible impurity of $\JH_{\mathbf{d},!}\ul{\BQ}_{\FM_{\mathbf{d}}(\Pi_Q)}$ and $\mathcal{F}$; see the proof of \cite[Thm.4.7]{preproj3}.  Recently Tasuki Kinjo has extended the dimensional reduction isomorphism to general 2CY categories \cite{Kin21}.  With some substantial work, one should be able to use his version of the dimensional reduction isomorphism along with the arguments of \cite{preproj3} to produce an alternative proof of Theorem \ref{loc_pur_thm}.  The first obstacle would be to prove the relative cohomological integrality theorem for 3CY completions of 2CY categories (this is used crucially in \cite{preproj3} in determining the supports of the BPS sheaves for $\fdmod^{\Pi_Q[x]}$).

\subsubsection{}
Under fairly mild hypotheses, the condition (*) of Theorem \ref{loc_pur_thm} is automatic:
\begin{proposition}
\label{Sclosed}
Let $U\subset \FM_{\mathscr{C},\theta}^{\zeta\sstab}$ be the open substack of Bridgeland stable objects of fixed slope $\theta$  for some locally finite Bridgeland stability condition\footnote{I.e. a locally finite Bridgeland stability condition on the homotopy category.} $\zeta$ on a pre-triangulated dg category $\mathscr{C}$.  Then a $\BC$-point $x$ is closed if and only if it corresponds to a polystable object, i.e. a direct sum
\[
\bigoplus_{i\in I}\CF_i^{\oplus a_i}
\]
for $\CF_i$ distinct stable objects of $\mathscr{C}$.  If $\mathscr{C}$ is a left 2CY category, the objects $\mathcal{F}_i$ form a $\Sigma$-collection.
\end{proposition}
We refer to the original paper of Bridgeland \cite{MR2373143} for the definition of locally finite stability conditions, and Bridgeland stability conditions in general.
\begin{proof}[Proof of Proposition \ref{Sclosed}]
Both parts are standard, but we provide the proof for completeness.  Let $\mathcal{F}$ be a Bridgeland semistable object corresponding to a closed point $x$, which admits a finite Jordan-H\"older filtration
\begin{equation}\label{HNfilt}
0=\CF_0\subset\ldots\subset \CF_r=\CF
\end{equation}
with each $\CF_m/\CF_{m-1}$ stable of slope $\theta$.  We claim that the inclusion $\CF_{r-1}\hookrightarrow \CF_r$ splits.  The inclusion corresponds to a degree 1 morphism $\alpha\colon \CF_r/\CF_{r-1}[-1]\rightarrow \CF_r$ and we consider the family $\cone(t\cdot \alpha)$ over $\BA^1$, where $t$ is the coordinate for $\BA^1$.  Since we assume that $\mathscr{C}$ is a dg-category, this cone is functorial, and this family is well-defined.  The canonical morphism $\psi\colon\BA^1\rightarrow \FM_{\mathscr{C},\theta}^{\zeta\sstab}$ defined by this family sends $\BA^1\setminus \{0\}$ to the point $x$, and so since $x$ is assumed closed $\psi(0)=x$, and we deduce that $\alpha=0$.  Applying the same argument inductively, we see that all of the inclusions \eqref{HNfilt} split.  The fact that a point $x$ corresponding to a polystable object is closed is clear: e.g. $x$ corresponds to the unique closed point of the coherent completion $(\widehat{(\FM_{\mathscr{C},\theta}^{\zeta\sstab})}_x,\hat{x})$ from Proposition \ref{app_prop}.
\smallbreak
For the final statement, there are no negative Exts between the $\mathcal{F}_i$, since they belong to a heart of $\mathscr{C}$ defined by the Bridgeland stability condition $\zeta$.  In addition, since they are pairwise non-isomorphic stable objects with the same slope, there are no degree zero homomorphisms between distinct $\mathcal{F}_i$, and each object has a $1$-dimensional space of degree zero endomorphisms.  This all follows from the following elementary fact: if $f\colon\CF\rightarrow \CG$ is a morphism between semistable objects of the same slope $\theta$, then both $\ker(f)$ and $\coker(f)$ have slope $\theta$, if they are nonzero.  The fact that the $\CF_i$ form a $\Sigma$-collection is then guaranteed by the existence of the nondegenerate symmetric pairing
\[
\Ext^n(\CF_i,\CF_j)\cong \Ext^{2-n}(\CF_j,\CF_i)^{\vee}
\]
coming from the (right) 2CY structure, which in turn is a consequence of Proposition~\ref{BDthm}.
\end{proof}

\subsection{Supports and intersection cohomology}
\label{IC_sec}
Let $p\colon \FM\rightarrow \CM$ be a stack with good moduli space satisfying the conditions of Theorem \ref{loc_pur_thm}, so in particular $\FM$ is an open substack of the stack of objects in a dg category $\mathscr{C}$.  Then by Theorem \ref{loc_pur_thm}, for all $n\in\BZ$ we obtain decompositions in $\MHM(\CM)$
\begin{align}
\label{sede}
\Ho^n\!p_!\underline{\BQ}_{\FM}\cong &\bigoplus_{j\in J_n}\ICS_{\overline{Z_j}}(\CL_j)
\end{align}
for $Z_j\subset \CM$ locally closed subvarieties, and $\CL_j$ simple pure variations of Hodge structure on them.  In this section we initiate the study of these supports and variations of Hodge structure.

\subsubsection{}
\label{Tret}
For concreteness, for now we assume that we have a presentation $\FM=X/G$ of $\FM$ as a global quotient stack, with $G\subset \Gl_N$ a reductive algebraic group.  Replacing $X$ by $X\times_G \Gl_N$ if necessary, we may assume $G=\Gl_N$.  We denote by $q\colon X\rightarrow \CM$ the natural map to the good moduli space of $\FM$.  We assume that the centre $\BC^*\subset \Gl_N$ acts trivially on $X$, so that the $\Gl_N$-action factors through $\PGl_N$.  We fix an open subscheme $\CM^s\subset \CM$, such that the restriction of the morphism $q$ to $X^s\coloneqq q^{-1}(\CM^s)$ is a principal $\PGl_N$-torsor, equivalently $p\colon\FM^s\rightarrow \CM^s$ is a $\BC^*$-gerbe.  If $\FM$ is the classical truncation of an open substack of the stack of objects in an Abelian subcategory $\mathscr{A}$ of the homotopy category of a dg category $\mathscr{C}$, then we may take $\FM^s$ to be the open substack of simple objects in $\mathscr{A}$.  If $\mathscr{C}$ carries a left 2CY structure, then via Theorem \ref{loc_str_thm} $\FM^s$ is automatically smooth, and each connected component of $\FM^s$ is dense inside the connected component of $\FM$ that contains it, see \cite[Prop.5.3+Prop.5.4]{DHSM23}.
\smallbreak
We denote by
\begin{equation}
\label{Det_choice}
\Det\colon \FM\rightarrow \pt/\BC^*.
\end{equation}
the morphism induced by $X\rightarrow \pt$ and $G\hookrightarrow \Gl_N\xrightarrow{\mathrm{det}} \BC^*$.
\smallbreak
Since $\FM$ is a global quotient stack, we may define the complex $p_!\underline{\BQ}_{\FM}$ as in \S \ref{stack_di}, and by Theorem \ref{loc_pur_thm} there is a decomposition
\begin{align}
\label{fde}
p_!\underline{\BQ}_{\FM}\otimes\LLL^{\frac{1}{2}\chi_{\mathscr{A}}(\cdot,\cdot)}\cong \bigoplus_{n\leq 0}\Ho^n\!\left(p_!\underline{\BQ}_{\FM}\otimes\LLL^{\frac{1}{2}\chi_{\mathscr{A}}(\cdot,\cdot)}\right)[-n].
\end{align}

\smallbreak

We assume that $\CM^s$ is smooth, but not necessarily connected.  We moreover assume that each component is even-dimensional.  This will hold in all of our examples since $\FM$ is the moduli stack of objects in a left 2CY category, so $\CM^s$ is moreover symplectic, since it is simultaneously a variety and a (0)-shifted symplectic derived stack by \cite{PTVV} or \cite{BD18}.  
\smallbreak
For each connected component $\CN\subset \CM^s$ we define
\[
\ICSn_{\overline{\CN}}\coloneqq \ICS_{\CN}\left(\underline{\BQ}_{\CN}\otimes\LLL^{-\dim(\CN)/2}
\right)
\]
and the normalised intersection cohomology
\[
\ICA(\overline{\CN})\coloneqq\HO(\overline{\CN},\ICSn_{\overline{\CN}})
\]
as in \S \ref{sec:algebraic_MHMs}. Due to the twisting conventions of \S \ref{sec:algebraic_MHMs}, if $\overline{\CN}$ is projective, the intersection cohomology $\ICA(\overline{\CN})$ is Verdier self-dual.  
\subsubsection{}
With the conventions from \S \ref{Tret} in place we can state the following
\begin{theorem}
\label{ICin}
Fix $p\colon \FM\rightarrow\CM$ and $\CM^s\subset \CM$ as in \S \ref{Tret}.  Then for every component $\CN$ of $\CM^s$, and every $n\in 2\cdot \BZ_{\leq 0}$ there is precisely one $j\in J_n$ with $\ol{Z_j}=\ol{\CN}$ in the decomposition \eqref{sede}.  Moreover, for this $j$ we have an isomorphism of mixed Hodge modules
\[
\CL_j\cong \underline{\BQ}_{\CN}\otimes \LLL^{n/2-\dim(\CN)/2}[n].
\]
If $n$ is odd, the variety $\CN$ does not appear in the decomposition \eqref{sede}.  
\smallbreak 
Taking Verdier duals and global sections, after choosing a determinant line bundle as in \eqref{Det_choice}, there is a canonical inclusion of cohomologically graded mixed Hodge structures
\[
\bigoplus_{\substack{\CN\in\pi_0(\CM^s)\\n\geq 0}}\ICA(\ol{\CN})\otimes \LLL^n\hookrightarrow \HO^{\BM}(\FM,\BQ\otimes\LLL^{\frac{1}{2}\chi_{\mathscr{A}}(\cdot,\cdot)}).
\]
\end{theorem}
In other words, the intersection cohomology of each component $\CN$ of $\CM$ that intersects the locus $\CM^s$ naturally includes (infinitely many times, with Tate twists) in the Borel--Moore homology of $\FM$.  The inclusion of the word canonical in the theorem will require some care, since the decomposition \eqref{fde} of a pure complex of mixed Hodge modules into a direct sum of its cohomology sheaves is not a priori canonical.  In particular, the construction will make it clear that the inclusion is only canonical up to the choice of determinant line bundle.
\begin{proof}[Proof of Theorem \ref{ICin}]
Let $\CN\in\pi_0(\CM^s)$.  Since each $\CN$ is open in $\CM$, we can verify the statement regarding supports by restricting $p_!\underline{\BQ}_{\FM}$ to $\CN$.  Let $e$ be the value of the locally constant function $\chi_{\mathscr{C}}(\cdot,\cdot)$ on $\CN$.  At a point $x$ of $\CN$ representing an object $\CF$ of the dg category $\mathscr{C}$, we have 
\[
e=\sum_{0\leq i\leq 2}(-1)^i\dim(\Ext^i(\CF,\CF)).
\]
Our assumptions on $\CM^s$ imply that $\dim(\Ext^0(\CF,\CF))=1$, while $\dim(\Ext^1(\CF,\CF))=\dim T_{\CN,x}$.  It follows that $\dim(\Ext^2(\CF,\CF))=1$, since $\Ext^{\bullet}(\CF,\CF)$ carries a right 2CY structure.  So $\dim(\CN)=2-e$.
\smallbreak
Let $j\colon U\rightarrow \CN$ be an \'etale morphism such that there is a Cartesian diagram
\[
\xymatrix{
U\times \pt/\BC^*\ar[r]\ar[d]^{\pi_U}&p^{-1}(\CN)\ar[d]^p\\
U\ar[r]^j&\CN.
}
\]
There is an isomorphism
\[
\pi_{U,!}\BQ_{U\times \pt/\BC^*}[-e]\cong \bigoplus_{m\geq 0}\BQ_U[2-e-2m]
\]
and so it follows that $\CN$ is a support in $J_n$ for $n\in 2\cdot\BZ_{\leq 0}$, as claimed.  By taking the Verdier dual of the decomposition \eqref{fde} there is a decomposition
\begin{equation}
\label{VDdec}
\BD (p_!\underline{\BQ}_{\FM}\otimes\LLL^{e/2})\cong \bigoplus_{n\geq 0}\Ho^n(\BD (p_!\underline{\BQ}_{\FM}\otimes\LLL^{e/2}))[-n]
\end{equation}
We denote by
\[
\iota^0\colon \Ho^0(\BD (p_!\underline{\BQ}_{\FM}\otimes\LLL^{e/2}))\rightarrow \BD (p_!\underline{\BQ}_{\FM}\otimes\LLL^{e/2}).
\]
the canonical morphism obtained by applying the natural transformation $\bm{\tau}^{\leq 0}\rightarrow \id$ to \eqref{VDdec}.  Consider the natural morphism $\underline{\BQ}_{\CM}\rightarrow p_*p^*\underline{\BQ}_{\CM}$.  Restricting to $\CN$, taking the tensor product with $\LLL^{1-e/2}$, and taking the Verdier dual, we obtain the morphism
\[
p_!\underline{\BQ}_{p^{-1}(\CN)}\otimes\LLL^{e/2}\rightarrow \underline{\BQ}_{\CN}\otimes\LLL^{e/2-1}.
\]
Using the first part of the proof, this becomes an isomorphism after taking $\Ho^0$, from which it follows that the simple mixed Hodge module having support $\overline{\CN}$ in \eqref{sede} is $\ICSn_{\overline{\CN}}$.  We thus obtain a canonical embedding 
\[
i\colon \ICSn_{\overline{\CN}}\rightarrow \Ho^0\left(\BD(p_!\underline{\BQ}_{\FM}\otimes\LLL^{e/2})\right).
\]
We define
\[
\iota=\iota^0\circ i\colon\ICSn_{\overline{\CN}}\rightarrow \BD(p_!\underline{\BQ}_{\FM}\otimes\LLL^{e/2}).
\]
Fix a nonzero $\eta\in \HO^2(\B\BC^*,\BQ)$.  Since $\BC^*$ acts trivially on $X$, $L\coloneqq \eta\cap\bullet$ acts injectively on $\HO^{\BM}(\FM,\BQ)$ by the Atiyah--Bott lemma \cite[Prop.13.4]{AB83}.  Moreover, by \cite[Prop.1.4.4]{dCHM12}, the operator $L$ lifts to a morphism 
\begin{equation}
\tilde{L}\colon \BD(p_!\underline{\BQ}_{\FM})\rightarrow \BD(p_!\underline{\BQ}_{\FM})\otimes\LLL^{-1}.
\end{equation}
For completeness, we remind the reader of the construction.  Let $\overline{\FM}=X\times_{\GL_N} \BA^1$, where $\Gl_N$ acts on $\BA^1$ via $\det$, and the scaling action of $\BC^*$ on $\BA^1$.  The projection $\pi\colon \overline{\FM}\rightarrow \FM$ is the projection of the total space of a line bundle with Euler class $\Det^*(\eta)$.  Let $z\colon \FM\rightarrow \overline{\FM}$ be the inclusion of the zero section, and set $\overline{p}=p\circ\pi$.  Then we obtain morphisms of perverse complexes
\[
a\colon \BQ_{\overline{\FM}}\rightarrow z_*\BQ_{\FM};\quad\quad
b\colon z_*\BQ_{\FM}\rightarrow \BQ_{\overline{\FM}}[2]
\]
where $b$ is the (shifted) Verdier dual of $a$.  Then we consider the morphism of perverse complexes 
\begin{equation}
\label{ATh}
 \overline{p}_!z_*(a\circ b)\colon p_!\BQ_{\FM}\rightarrow p_!\BQ_{\FM}[2].
\end{equation}
Working in a bounded range of cohomology, and replacing $\FM$ by a scheme as in \S \ref{stack_di} we upgrade \eqref{ATh} to a morphism of complexes of mixed Hodge modules, to obtain
\begin{equation}
\label{naga}
\overline{p}_!z_*(a\circ b)\colon p_!\underline{\BQ}_{\FM}\rightarrow p_!\underline{\BQ}_{\FM}\otimes \LLL^{-1}.
\end{equation}
Taking the appropriate $\LLL$-twist of the Verdier dual of \eqref{naga}, we obtain the promised morphism $\tilde{L}$.  Composing $\iota$ with $\tilde{L}^n$ we obtain the morphism
\[
\ICSn_{\overline{\CN}}\rightarrow \BD(p_!\underline{\BQ}_{\FM}\otimes\LLL^{e/2})\otimes \LLL^{-n}.
\]
Pulling back to $U$, and taking the zeroth cohomology of the resulting morphism of perverse sheaves, we obtain the morphism
\[
\BQ_{U}[2-e]\xrightarrow{s\mapsto s\otimes (N\eta)^n} \BQ_U[2-e]\otimes \HO^{2n}(\pt /\BC^*,\BQ)
\]
which is an isomorphism.  We deduce that the VHS $\CL_j$ corresponding to the support $\CN$ appearing in the decomposition \eqref{sede} is indeed $\underline{\BQ}_{\CN}\otimes \LLL^{e/2+n/2-1}[e+n-2]$, and moreover the morphism
\[
L^n\circ \HO(\iota)\colon \ICA(\overline{\CN})\otimes\LLL^n\rightarrow \HO^{\BM}(\FM,\BQ\otimes\LLL^{\frac{1}{2}\chi_{\mathscr{A}}(\cdot,\cdot)})
\]
is the required (canonical, after choosing $\Det$) embedding.
\end{proof}
\subsubsection{}
In order to talk about mixed Hodge structures and mixed Hodge modules, in the setup to Theorem \ref{ICin} we have assumed that $\FM$ is a global quotient stack.  Without this assumption we can still relate the intersection cohomology of top-dimensional components of $\CM$ to the Borel--Moore homology of $\FM$:
\begin{proposition}
Let $\mathscr{C}$ be a left 2CY category let $\FM=t_0(\bm{\FM})$ with $\bm{\FM}\subset \bm{\FM}_{\mathscr{C}}$ an open substack of the stack of objects, and assume that property (*) of Theorem \ref{loc_pur_thm} holds.  Moreover assume that we have a good moduli space $p\colon\FM\rightarrow \CM$.  Let $\{\CN_t\}_{t\in T}$ be the set of components of $\CM$ that are generically smooth, and for which $p^{-1}(\CN_t)\rightarrow \CN_t$ is generically a $\B \BC^*$-gerbe.  Then there is a natural morphism of cohomologically graded vector spaces
\[
\bigoplus_{t\in T}\rat(\ICA(\CN_t))\rightarrow \HO^{\BM}(\FM,\BQ[-\chi_{\mathscr{C}}(\cdot,\cdot)])
\]
which is an injection if $p_!\BQ_{\mathfrak{M}}$ splits into shifted perverse sheaves.
\end{proposition}
\begin{proof}
By Theorem \ref{loc_pur_thm} the complex of perverse sheaves $\CP=\BD p_!\BQ[-\chi_{\mathscr{C}}(\cdot,\cdot)]$ satisfies ${}^{\mathfrak{p}}\!\Ho^n(\CP)=0$ for $n>0$, so that ${}^{\mathfrak{p}}\bm{\tau}^{\leq 0}\CP\cong {}^{\mathfrak{p}}\!\Ho^0(\CP)$, and there is a natural morphism ${}^{\mathfrak{p}}\!\Ho^0(\CP)\rightarrow \CP$.  By Theorem \ref{loc_pur_thm} again, there is a pure mixed Hodge module $\Ho^0\!p_!\underline{\BQ}_{\FM}$ satisfying 
\[
\rat\left(\Ho^0\!p_!\underline{\BQ}_{\FM}\right)\cong {}^{\mathfrak{p}}\!\Ho^0(\CP),
\]
and so by Theorem \ref{thm:Saito_theorem}, ${}^{\mathfrak{p}}\!\Ho^0(\CP)$ decomposes into a direct sum of intersection cohomology complexes.  Now the argument that $\IC_{\ol{\CN^t}}$ appears as a summand, for each $t\in T$, is as in the proof of Theorem \ref{ICin}.
\end{proof}

\subsection{Results on global cohomology}
The local purity result (Theorem \ref{loc_pur_thm}) implies all of our global results regarding the Borel--Moore homology of stacks of objects in 2CY categories.  
\subsubsection{}
The first result follows from Theorem~\ref{loc_pur_thm}, and Proposition \ref{int_filt}:
\begin{theorem}
\label{prop_cor}
Under the assumptions of Theorem \ref{loc_pur_thm}, and assuming moreover that $\FM$ is exhausted by global quotient stacks, for each $n$ the natural mixed Hodge structure on $\HO^{\BM}_{n}(\FM,\BQ)$ carries an ascending perverse filtration by mixed Hodge structures, with respect to the morphism $p\colon \FM\rightarrow \CM$.
\end{theorem}

In considering Serre subcategories of 2CY categories we will use the following more general statement, which in turn follows from Theorem~\ref{loc_pur_thm} and Proposition~\ref{PIS_prop}.
\begin{theorem}
\label{prop_serre_cor}
Retaining the assumptions of Theorem \ref{prop_cor}, let $\FN=p^{-1}(S)\subset \FM$ be a substack, where $S\subset \CM$ is a locally closed subscheme.  Then the mixed Hodge structure on $\HO^{\BM}_{n}(\FN,\BQ)$ carries a natural ascending perverse filtration by mixed Hodge structures, with respect to the morphism $p\colon \FM\rightarrow \CM$.
\end{theorem}
We refer the reader to Definition \ref{ambient_per} for the definition of the perverse filtration on $\FN$ with respect to the morphism from the ambient stack $\FM$.
\subsubsection{}

From Theorem \ref{loc_pur_thm} and Proposition \ref{int_pur} we deduce the theorem that will imply the Halpern-Leistner conjecture regarding coherent sheaves on K3 surfaces:
\begin{theorem}
\label{glob_pur_thm}
Assume that $\FM$ is a moduli stack of objects in a category $\mathscr{C}$ satisfying the conditions of Theorem \ref{loc_pur_thm}, and assume moreover that the good moduli space $\CM$ is a projective variety, and $\FM$ is exhausted by global quotient stacks.  Then the mixed Hodge structure on $\HO^{\BM}(\FM,\BQ)$ is pure.
\end{theorem}
\begin{remark}
Projectivity of $\CM$ is a sufficient but not necessary condition for purity of $\HO^{\BM}(\FM,\BQ)$, as we will see in the case of general moduli stacks of Higgs bundles, and as is already known in the case of moduli of representations of a preprojective algebra \cite{2016arXiv160202110D}.
\end{remark}
\begin{remark}
Let $\mathscr{C}$ be a left 2CY category, and let $\FM$ be a classical substack of the stack of objects in $\mathscr{C}$.  Without restricting to objects in $\mathscr{C}$ that are semistable with respect to some type of stability condition, it is generally not to be expected that there is a good moduli space $p\colon\FM\rightarrow \CM$.  However, after stratifying $\FM$ according to Harder--Narasimhan types, it is often possible to show that each stratum has pure Borel--Moore homology, using the 2CY property, and then use long exact sequences to show that the Borel--Moore homology of $\FM$ is pure.  See \S \ref{gen_sheaves} for the case of coherent sheaves on surfaces, and \S \ref{UnstHS} for Higgs sheaves.
\end{remark}

\section{Applications}
\label{sec:applications}

\subsection{Preprojective algebras revisited}
\label{proproj_redux}
For a warmup, we briefly revisit the ``local model'': the category of representations of a (possibly derived) preprojective algebra.  Fix $k$ a field of characteristic zero.  We recall the following result of Keller, a very special case of his general result on deformed Calabi--Yau completions:
\begin{proposition}\cite{Ke11}
Let $Q$ be a finite quiver.  Then the category $\Ddg(\Mod^{\mathscr{G}_2(kQ)})$ carries a left 2CY structure.
\end{proposition}

This brings us to formality for $\mathscr{G}_2(kQ)$-modules, a special case of Corollary \ref{form_cor}:
\begin{proposition}
\label{pp_formality}
Let $\CF_1,\ldots,\CF_r$ be a collection of pairwise non-isomorphic complexes of $\mathscr{G}_2(kQ)$-modules.  Assume that there are no negative Exts between the $\CF_i$, and that the only degree zero Exts are scalar multiples of identity maps.  Then the full subcategory of $\Ddg(\Mod^{\mathscr{G}_2(\mathbb{C}Q)})$ containing $\CF_1,\ldots,\CF_r$ is formal.
\end{proposition}
Via the morphism~\eqref{der_pp} we could consider a collection $\CF_1,\ldots,\CF_r$ of pairwise non-isomorphic simple $\Pi_Q$-modules in Proposition~\ref{pp_formality}.  Then the above result is in the literature, and (following \cite{SchKa19}) we would direct the reader to \cite{VdB15} for the statement, or at least its Koszul dual.
\subsubsection{}
For the rest of \S \ref{proproj_redux} we set $k=\BC$.  We fix a quiver $Q$ and a stability condition $\zeta\in\BQ^{Q_0}$.  Recall that we denote by 
\[
\JH^{\zeta}_{\dd}\colon\FM_{\dd}^{\zeta\sstab}(\Pi_Q) \to \CM^{\zeta\sstab}_{\dd}(\Pi_Q)
\]
the morphism from the stack of $\zeta$-semistable $\Pi_Q$-modules to the coarse moduli scheme.  Using Theorem \ref{loc_pur_thm} it follows that $
\JH^{\zeta}_{\dd,!}\ul{\BQ}_{\FM_{\dd}^{\zeta\sstab}(\Pi_Q)}$ is pure, which is one of the main results of \cite{preproj3}.  
\subsubsection{}
Especially in the special case of degenerate stability condition $\zeta=(0,\ldots,0)$ this may seem a rather circular application, since it is precisely the purity of this complex that we have imported from \cite{preproj3} and used throughout the paper!  So instead, we show how to recover the following global purity result from \cite{2016arXiv160202110D} using local purity.
\begin{theorem}\cite{2016arXiv160202110D}\label{1a}
The natural mixed Hodge structure on the Borel-Moore homology $\HO^{\BM}(\FM_{\dd}^{\zeta\sstab}(\Pi_Q),\BQ)$ is pure.
\end{theorem}
\begin{proof}
The algebra $\Pi_Q$ carries a $\BC^*$-action, by scaling all of the arrows of $\overline{Q}_1$; the defining relations of $\Pi_Q$ are quadratic, and so this action preserves the two-sided ideal generated by them.  Note that this action contracts $\CM_{\dd}^{\zeta\sstab}(\Pi_Q)$ onto the subscheme $N$ corresponding to nilpotent $\zeta$-semistable $\Pi_Q$-representations.  This scheme is projective since it is the fibre over the origin of the (projective) GIT quotient morphism $\CM^{\zeta\sstab}_{\mathbf{d}}(\Pi_Q)\rightarrow \CM_{\mathbf{d}}(\Pi_Q)$.
\smallbreak
By \eqref{callback}, it is sufficient to show that
\[
\HO_c(\CM_{\dd}^{\zeta\sstab}(\Pi_Q),\Ho^i(\JH^{\zeta}_{\mathbf{d},!}\ul{\BQ}_{\FM_{\dd}^{\zeta\sstab}(\Pi_Q)})[-i])
\]
is pure for every $i\in\BZ$.  Taking the Verdier dual, it is enough to show instead that $\HO(\CM_{\dd}^{\zeta\sstab}(\Pi_Q),\CG[i])$ is pure, where $\CG$ is the Verdier dual of $\Ho^i(\JH^{\zeta}_{\mathbf{d},!}\ul{\BQ}_{\FM_{\dd}^{\zeta\sstab}(\Pi_Q)})$.  Since purity is preserved by Verdier duality, and $\JH^{\zeta}_{\mathbf{d},!}\ul{\BQ}_{\FM_{\dd}^{\zeta\sstab}(\Pi_Q)}$ is pure by Theorem \ref{loc_pur_thm}, $\CG[i]$ is pure.  The perverse sheaf underlying $\CG$ is $\BC^*$-equivariant, so that the restriction map
\begin{equation}
    \label{sandwich}
\HO(\CM_{\dd}^{\zeta\sstab}(\Pi_Q),\CG[i])\rightarrow \HO(N,\CG[i]\lvert_N)
\end{equation}
is an isomorphism.  By \cite[\S 4.5]{Sai90} the morphism $(\CM_{\dd}^{\zeta\sstab}(\Pi_Q)\rightarrow \pt)_*$ increases weights, the morphism $(N\rightarrow \CM_{\dd}^{\zeta\sstab}(\Pi_Q))^*$ decreases weights, while the morphism $(N\rightarrow \pt)_*$ preserves them (since it is projective).  Since $\CG[i]$ is pure, it thus follows from the fact that \eqref{sandwich} is an isomorphism that $\HO(\CM_{\dd}^{\zeta\sstab}(\Pi_Q),\CG[i])$ is pure, as required.
\end{proof}

\subsection{Semistable coherent sheaves on surfaces}
\label{CSsec}
Let $X$ be a connected $d$-dimensional\footnote{Meaning that the maximal dimension the components of $X$ is $d$.} projective scheme over $k$, an algebraically closed field of characteristic zero.  We start by collecting some background on moduli of coherent sheaves, following \cite{MR2665168}.
\smallbreak
A coherent sheaf $\CF$ is said to have pure $i$-dimensional support if $\CF$ has $i$-dimensional support, and so does every nontrivial subsheaf $\CF'\subset \CF$.  Every coherent sheaf $\mathcal{F}$ on $X$ admits a unique torsion filtration
\begin{equation}
\label{tor_filt}
\CF_{(-1)}\subset \CF_{(0)}\subset \ldots\CF_{(d)}=\CF
\end{equation}
where for $i\geq 0$ the sheaf $\CF_{(i)}/\CF_{(i-1)}$ has pure $i$-dimensional support, and we set $\CF_{(-1)}=0$.  So $\CF$ has pure $i$-dimensional support if and only if $\CF_{(i-1)}=0$ and $\CF_{(i)}=\CF$ for some $i$.  The filtration is stable under taking tensor products with invertible coherent sheaves.  For $\CF$ a coherent sheaf we write $\ho^i(\CF)=\dim_k\HO^i(X,\CF)$.  Later we will use the following easy result.
\begin{lemma}
\label{tors_grab}
Let $\CF$ be a coherent sheaf on $X$ with pure $d'$-dimensional support, for $0\leq d'\leq d$, and let $\CL$ be an ample line bundle on $X$.  Then for $n\ll 0$ we have
\[
\ho^0(\CF\otimes\CL^n)=\begin{cases} \mathrm{length}(\CF)& \textrm{if }d'=0\\
0&\textrm{otherwise.}\end{cases}
\]
Consequently, for $\CF$ a coherent sheaf, and $n\ll 0$, there is an identity $\ho^0(\CF\otimes\CL^n)=\mathrm{length}(\CF_{(0)})$ with $\CF_{(0)}$ as in \eqref{tor_filt}.
\end{lemma}
\begin{proof}
The $d'=0$ case is clear, since then $\CF\otimes\CL^n\cong\CF$.  For the remaining cases, we may assume that $\CL$ is very ample and set $\CL=\CO_X(D)$.  Then sections of $\CF\otimes\CL^n$ are identified with sections of $\CF$ vanishing to order $-n$ along $D$.  By supposition, any section of $\CF\otimes\CL^n$ has support in dimension greater than zero, intersecting $D$.  The final statement follows from analysing the long exact sequences for the derived global sections functor applied to the torsion filtration.
\end{proof}

\sssct
Let $H \in \NS(X)_{\BR}$ be an ample class.
The Hilbert polynomial of a sheaf $\CF \in \Ob(\Coh(X))$ is the integer-valued polynomial
\[
P_{\CF}(m)\coloneqq \chi\left(\CF \otimes \CO_X(mH)\right)\coloneqq \sum_{i\in \BZ}(-1)^i \ho^i\!\left(\CF \otimes \CO_X(mH) \right).
\]
We can write
\[
P_{\CF}(m) = \sum_{i = 0}^{\dim(\supp( \CF))} a_i(\CF) \frac{m^i}{i!},
\]
where $a_i(\CF) \in \BZ$.
We denote the reduced Hilbert polynomial of $\CF$ by
\[
p_{\CF}(m) = \frac{P_{\CF}(m)}{a_{d}(\CF)} ,
\]
where $d = \dim(\supp( \CF))$.  We will adopt the convention throughout \S \ref{CSsec} that whenever a polynomial in $\BQ[t]$ is represented by a lowercase letter like $p(t)$, it is normalised so that its leading term is $t^d/d!$ for some $d$.
\smallbreak
There is a total order on polynomials $P(t)\in \mathbb{Q}[t]$ defined by $P(t)\geq Q(t)$ if $P(N)\geq Q(N)$ for $N\gg 0$.  A sheaf $\CF$ is defined to be $H$-Gieseker semistable if it is pure and for every non-zero proper subsheaf $\CE \subsetneq \CF$ we have $p_{\CE}(t)\leq p_{\CF}(t)$, and $\CF$ is moreover stable if this inequality is strict.  Every pure $d$-dimensional sheaf $\mathcal{F}$ carries a unique Harder--Narasimhan (HN) filtration
\[
0=\CF_0\subset\ldots\subset \CF_r=\CF
\]
with each $\CF_{i}/\CF_{i-1}$ pure $d$-dimensional and $H$-Gieseker semistable, and with $p_{\CF_{i}/\CF_{i-1}}(t)>p_{\CF_{i+1}/\CF_{i}}(t)$ for every $i$.  In particular $\CF$ is semistable if and only if $r=1$ in the HN filtration of $\CF$.  The collection of polynomials $(P_{\CF_1}(t),\ldots,P_{\CF_{r}/\CF_{r-1}}(t))$ is called the HN type of $\CF$.  
\begin{lemma}\cite{Nit11} 
\label{Nit_lemma}
There is a total ordering $\preccurlyeq$ on the set of HN types $(P_1(t),\ldots,P_r(t))$ satisfying $\sum P_i(t)=P(t)$, stratifying the stack $\FN$ of pure-dimensional coherent sheaves on $X$ with Hilbert polynomial $P(t)$ into locally closed substacks $\FM^H_{\bm{\alpha}}(X)\subset \FN$ of sheaves with HN type $\bm{\alpha}$.  I.e. if we define
\[
\FM^H_{\preccurlyeq \bm{\alpha}}(X)=\bigcup_{\bm{\beta}\preccurlyeq \bm{\alpha}}\FM^H_{\bm{\beta}}(X)
\]
then $\FM^H_{\preccurlyeq \bm{\alpha}}(X)\subset \FN$ is a closed embedding, and $\FM^H_{\bm{\alpha}}(X)\subset \FM^H_{\preccurlyeq \bm{\alpha}}(X)$ is an open embedding.
\end{lemma}
\sssct
A $H$-Gieseker semistable sheaf $\CF$ admits an increasing Jordan--H\"{o}lder (JH) filtration,
\[
0 = \CF_0 \subset \CF_1 \subset \ldots \subset \CF_k = \CF,
\]
such that each subquotient $\CF_i/\CF_{i-1}$ is $H$-Gieseker stable with the same reduced Hilbert polynomial as $\CF$.
Unlike the torsion filtration and the HN filtration, the JH filtration is not unique.  However the direct sum of subquotients, $\gr_H(\CF) \coloneqq \bigoplus_{i=1}^{k} \CF_i/\CF_{i-1}$ is uniquely determined up to isomorphism, and is moreover the semisimplification of $\CF$ inside the category of semistable coherent sheaves with fixed (reduced) Hilbert polynomial $p_{\CF}(t)$.  Two $H$-Gieseker semistable sheaves $\CF$ and $\CF'$ are called S-equivalent if $\gr_H(\CF)$ and $\gr_H(\CF')$ are isomorphic.
\sssct
\label{MSCON}
Let $P(t) \in \BQ[t]$ be a polynomial.  For $T$ a scheme, a $T$-family of ($H$-Gieseker) semistable sheaves with Hilbert polynomial $P(t)$ is a coherent sheaf $\CF$ on $X \times T$, flat over $T$, such that for all geometric points $s \colon \Spec K \to T$, the sheaf $\CE_s = (\Id_X \times s)^*\CF$ is semistable on $X_K$ with Hilbert polynomial $P(t)$.  We denote by $\FM^{H\sstab}_{P(t)}(X)$ the moduli stack whose $T$-valued points are $T$-families of $H$-Gieseker semistable sheaves on $X$ with Hilbert polynomial $P(t)$.  Likewise we denote by $\FM^{H\st}_{P(t)}(X)$ the substack of stable coherent sheaves.
\smallbreak
We briefly recall the construction of this stack as a global quotient stack, along with the GIT construction of its coarse moduli space, in order to be able to work with the direct image mixed Hodge module complex along the morphism to the coarse moduli space, as defined in \S\ref{stack_di}. 
\smallbreak
First, semistable sheaves on $X$ with Hilbert polynomial $P(t)$ form a bounded family by \cite[Thm.3.3.7]{MR2665168}.
Hence there is an integer $m$ such that $\CE(m)$ is globally generated whenever $\CE$ is a semistable sheaf with fixed Hilbert polynomial $P(t)$, and $P(m) = \ho^0(\CE(m))$, while $\ho^i(\CE(m))=0$ for $i\neq 0$.
Let $V \coloneqq k^{\oplus P(m)}$ and $\CH = V \otimes \CO_X(-m)$.
An isomorphism $g \colon V \xrightarrow{\cong} H^{0}(\CE(m))$ yields a surjection $f \colon \CH \onto \CE$, by precomposing the evaluation morphism $H^0(\CE(m)) \otimes \CO_X(-m) \onto \CE$ with the isomorphism $g\otimes\id_{\mathcal{O}_X(-m)}$.
This yields a closed point,
\[
[f] \in \Quot_{X}(\CH,P(t)),
\]
in the corresponding Quot scheme.  Conversely, there is a subscheme $U \subset \Quot_{X}(\CH,P(t))$ parameterising those quotients $\CH \onto \CE$ such that $\CE$ is semistable and the induced morphism,
\[
g\colon V \to H^0(\CH(m)) \to H^0(\CE(m)),
\]
is an isomorphism.  The subscheme $U$ is open \cite[Prop.2.3.1]{MR2665168}, and invariant under the natural right action of $\GL_{V}$ on $\Quot_X(\CH,P(t))$ by precomposition.  If we define $U^s\subset U$ to be the subscheme parameterising those quotients $\CH \onto \CE$ such that $\CE$ is stable, then $U^s\subset U$ is open, again by \cite[Prop.2.3.1]{MR2665168}.
There are isomorphisms
\[
   \FM^{H\sstab}_{P(t)}(X) \cong U/\GL_{V};\quad\quad
   \FM^{H\st}_{P(t)}(X)\cong U^s/\GL_V.
\]
There is a natural linearisation $\chi$ of the $\Gl_V$-action on $\overline{U}$ such that points in $U$ are semistable in the sense of GIT and the corresponding quotient is a good GIT quotient $\CM^{H\sstab}_{P(t)}(X) = \overline{U} /\!\!/_{\chi} \GL_{V}$ \cite[Thm.4.3.3]{MR2665168}.  The morphism $\FM^{H\sstab}_{P(t)}(X) \longrightarrow \CM^{H\sstab}_{P(t)}(X)$ is a good moduli space as in Definition~\ref{def:good_moduli_space}.  Moreover the centre $\BG_m\subset \GL_V$ acts trivially on $U$, and the morphism $U^s\rightarrow \CM^{H\st}_{P(t)}(X)$ is a principal $\PGl_V$-bundle.
\smallbreak 
If $X$ is a quasiprojective variety, we fix a projective compactification $\overline{X}$, and realise $\FM^{H\sstab}_{P(t)}(X)\subset \FM^{H\sstab}_{P(t)}(\overline{X})$ (and $\CM^{H\sstab}_{P(t)}(X)\subset \CM^{H\sstab}_{P(t)}(\overline{X})$) as the open substack (respectively subscheme) parameterising sheaves $\CF$ such that $\supp(\CF)\subset X$.  By universality of the moduli stack (respectively coarse moduli scheme) the spaces $\FM^{H\sstab}_{P(t)}(X)$ and $\CM^{H\sstab}_{P(t)}(X)$ only depend on the choice of $\overline{X}$ up to canonical isomorphism.
\smallbreak
We summarise the construction in the following proposition.
\begin{proposition}
Let $X$ be, as above, a connected quasiprojective scheme, and fix an ample class $H \in \NS(\overline{X})_{\BR}$.  The moduli stack $\FM^{H\sstab}_{P(t)}(X)$ is an algebraic stack of finite type which admits a presentation as a global quotient stack.
It has a good moduli space 
\begin{equation}
    \pi^{H}_{P(t)} \colon \FM^{H\sstab}_{P(t)}(X) \longrightarrow \CM^{H\sstab}_{P(t)}(X),
\end{equation}
with $\CM^{H\sstab}_{P(t)}(X)$ quasiprojective, and moreover projective if $X$ is.
\end{proposition}
\sssct
Let $S$ be a smooth projective surface.  Fix a Mukai vector $\nu=v(\CF)\coloneqq \mathrm{ch}(\CF)\cdot \sqrt{\mathrm{td}(S)}\in\HO^{\mathrm{even}}(S,\mathbb{Z})$.  We denote by $\FM^H_{\nu}(S)$ the stack of $H$-Gieseker semistable coherent sheaves with Mukai vector $\nu$.  The Mukai vector of a coherent sheaf on $S$ determines the Poincar\'e polynomial $P(t)$ of $\CF$, and moreover the inclusion
\[
\FM^H_{\nu}(S)\hookrightarrow \FM^{H\sstab}_{P(t)}(S)
\]
is open and closed, so that $\HO^{\BM}\!\left(\FM^H_{\nu}(S),\BQ\right)$ appears as a direct summand of $\HO^{\BM}\!\left(\FM^{H\sstab}_{P(t)}(S),\BQ\right)$ in the category of mixed Hodge structures.  In particular, purity of $\HO^{\BM}\!\left(\FM^H_{\nu}(S),\BQ\right)$ will follow from purity of $\HO^{\BM}\!\left(\FM^{H\sstab}_{P(t)}(S),\BQ\right)$.
\sssct
Let $X$ be a separated scheme of finite type, and let $\QCoh_{\dg}(X)$ denote the dg category of unbounded complexes of quasi-coherent sheaves on $X$.  Recall that an h-injective object in a dg category $\mathscr{C}$ is an object that is injective in its homotopy category $\HO^0(\mathscr{C})$.  Consider the full dg subcategory $\Dub_{\dg}(\Coh(X)) \subset \QCoh_{\dg}(X)$ containing complexes whose objects are h-injective and whose cohomology sheaves are coherent.  This provides an explicit dg enhancement of $\Dub(\Coh X)$; see \cite[Sec.~3.1]{MR3421093}.  One obtains a dg enhancement of $\Db_{\dg}(\Coh(X))$ by taking the full dg subcategory of $\Dub_{\dg}(\Coh(X))$ containing those complexes with bounded cohomology.  Equivalently, we may consider the full subcategory of $\QCoh_{\dg}(X)$ containing those complexes of injective quasi-coherent sheaves, bounded on the left, whose total cohomology is coherent; see \cite{MR3421093} for proofs and details.
\sssct
The following result provides a left CY structure on $\Db_{\dg}(\Coh(X))$ in the cases that we will consider later.
\begin{proposition}\cite[Prop.~5.12]{BD19}
\label{prop:CY_on_K3}
Let $X$ be a separated scheme of finite type, and let $\CE_X$ be the dualizing complex of $X$. Then there is a natural bijection between the set of quasi-isomorphisms $\CO_X\rightarrow \CE_X[d]$ and the set of left dCY structures on $\Db_{\dg}(\Coh(X))$.  In particular, if $S$ is a surface, the category $\Db_{\dg}(\Coh(S))$ carries a left 2CY structure if and only if $S$ is Gorenstein and has trivial dualizing sheaf.
\end{proposition}
\sssct
The following is a special case of formality (Corollary \ref{form_cor}):
\begin{proposition}
Let $S$ be a smooth complex surface with $\omega_S\cong \CO_S$.  Let $\BS=\{\CF_1,\ldots,\CF_r\}$ be a $\Sigma$-collection of complexes in $\Db(\Coh(S))$ such that each $\CF_i$ has compactly supported cohomology.  Then the full subcategory of $\Db_{\dg}(\Coh(S))$ containing $\BS$ is formal.
\end{proposition}
\subsubsection{Relation with Toda's neighbourhood theorem}
In \cite{MR3811778} Toda proves an analytic local neighbourhood theorem for the morphism $\pi^H_{P(t)}$ (for general projective varieties $X$), which combined with the above formality result, tells us that for $S$ as in the proposition, and additionally quasiprojective, there is a Cartesian diagram as in \eqref{eq:neighbourhood_diagram}, with $p=\pi^H_{P(t)}$, and with the modification that $H$ is an analytic variety equipped with a $\Gl_{\mathbf{d}}$-action, $H/\Gl_{\mathbf{d}}$ is an analytic stack, and the horizontal morphisms are analytic open embeddings.  
\smallbreak
For Hodge-theoretic applications, it is more convenient to work with \'etale neighbourhoods.  This is because there is no six functor formalism for derived categories of mixed Hodge modules on analytic varieties, so it would be a challenge to translate the proof of Theorem \ref{loc_pur_thm} into the category of analytic mixed Hodge modules.  Note that $\ul{\BQ}_{\FM}$ (or its approximation via well-defined complexes of mixed Hodge modules as in \S \ref{heart_di}) is a true \textit{complex} of mixed Hodge modules, since $\FM$ is generally highly singular (see Example \ref{qex}).
\subsubsection{Results for semistable sheaves}
Fix a smooth complex surface $S$ satisfying $\omega_S\cong\mathcal{O}_S$.  We denote by $\chi_S(\cdot,\cdot)$ the Euler form for the category of (compactly supported) coherent sheaves on $S$.  By Proposition~\ref{BDthm}, if $\CF_1,\ldots,\CF_r$ is a collection of complexes of compactly supported coherent sheaves on $S$, the full dg subcategory of $\Db_{\dg}(\Coh(S))$ containing $\CF_1,\ldots,\CF_r$ carries a right 2CY structure.
\begin{proposition}
\label{rel_pur_S}
Let $S$ be a smooth complex quasiprojective surface, and fix a Hilbert polynomial $P(t)$, and ample class $H \in \NS(\overline{S})_{\BR}$.  We assume that this data satisfies the following support condition: 
\begin{itemize}
\item[*]
For every $H$-Gieseker semistable sheaf $\CF$ on $\overline{S}$ with Hilbert polynomial $P(t)$, and support contained in $S$, there exists an open $S'\subset S$, containing the support of $\CF$, such that $\CO_{S'}\cong \omega_{S'}$.  
\end{itemize}
Then the mixed Hodge module complex $\pi^H_{P(t),!}\ul{\BQ}_{\FM^{H\sstab}_{P(t)}(S)}\otimes \LLL^{\frac{1}{2}\chi_S(\cdot,\cdot)}$ is pure, and is concentrated in cohomological degrees $i\leq 0$.  Therefore the dual compactly supported cohomology
\[
\HO^{\BM}(\FM^{H\sstab}_{P(t)}(S),\BQ\otimes\LLL^{\frac{1}{2}\chi_S(\cdot,\cdot)})
\]
carries an ascending perverse filtration relative to $\pi_{P(t)}^H$ by mixed Hodge structures, beginning in perverse degree zero.
\end{proposition}
\begin{proof}
Let $\CF$ be a $H$-Gieseker polystable sheaf represented by a point $\iota\colon\{x\}\hookrightarrow \CM_{P(t)}^{H\sstab}(S)$.  Then by the support assumption, we can find an $S'\subset S$ with trivial canonical bundle such that $\iota$ factors through the open embedding $\CM_{P(t)}^{H\sstab}(S')\subset \CM_{P(t)}^{H\sstab}(S)$, and it is enough to prove  the proposition under the assumption $\omega_S\cong\CO_S$.  The result then follows from our local purity theorem (Theorem \ref{loc_pur_thm}), along with the GIT construction of the moduli spaces of semistable coherent sheaves recalled above.
\end{proof}
The support condition in Proposition \ref{rel_pur_S} is automatic if $P(t)$ is constant, or if $\omega_S\cong\CO_S$.  Halpern--Leistner's conjecture \cite[Conj.~4.4]{HL16} follows from Theorem \ref{glob_pur_thm} by considering the second of these cases.
\begin{theorem}
\label{HL_con}
Let $S$ be a smooth complex projective surface, let $H \in \NS(S)_{\BR}$ be an ample class, and let $P(t)$ be a Hilbert polynomial satisfying the condition (*) of Proposition \ref{rel_pur_S}.  Then
\[
\HO^{\BM}(\FM^{H\sstab}_{P(t)}(S),\BQ)
\]
is pure.  In particular, if $S$ is an Abelian surface or a K3 surfacethen the above Borel--Moore homology is pure for any choice of ample $H$ and Hilbert polynomial $P(t)$.
\end{theorem}
This deals with case (2) of Theorem \ref{thm:purity_absolute}.  By construction, the moduli stacks $\FM^{H\sstab}_{P(t)}(S)$ and $\FM^{H\st}_{P(t)}(S)$ satisfy the conditions of \S \ref{Tret}.  We thus deduce the following from Theorem \ref{ICin}:
\begin{proposition}
\label{HL_c_c}
Let $S$ be a smooth complex quasiprojective surface satisfying $\omega_S \cong \CO_S$.  Fix an ample class $H \in \NS(\overline{S})_{\BR}$, Hilbert polynomial $P(t)$, and determinant bundle \eqref{Det_choice}.  Then there are natural inclusions of cohomologically graded mixed Hodge structures
\[
\bigoplus_{\substack{\CN\in\pi_0\left(\CM^{H\st}_{P(t)}(S)\right)\\{n\geq 0}}}\ICA(\overline{\CN})\otimes\LLL^n\hookrightarrow \HO^{\BM}(\FM^{H\sstab}_{P(t)}(S),\BQ\otimes \LLL^{\frac{1}{2}\chi_{S}(\cdot,\cdot)})
\]
with $\ICA(\overline{\CN})\otimes\LLL^n$ including into the $(2n)$th piece of the perverse filtration on the target.  The same statement remains true under the weaker condition (*).
\end{proposition}
\begin{remark}
We expect both the above results to hold for more general stacks of Bridgeland-semistable complexes of coherent sheaves on K3 and Abelian surfaces, with the required finite type Artin stacks constructed as in \cite{MR2177199,MR2397465,Betal19,AHLH18}.  These papers use general representability results to obtain existence and further results for these stacks of semistable objects, in contrast to the more ``hands on'' construction recalled in \S \ref{CSsec}; in the absence of a result stating that these stacks are exhausted by global quotient stacks, we are limited to stating Theorem \ref{HL_con} and Proposition \ref{HL_c_c} for stacks of Gieseker semistable coherent sheaves (not Bridgeland-semistable complexes), for now.  
\end{remark}
\subsection{General coherent sheaves on surfaces}
\label{gen_sheaves}
In the previous subsection, we proved purity for the Borel--Moore homology of certain stacks admitting good moduli spaces (Theorem \ref{HL_con}).  We can in fact extend our results to more general stacks, including those that do not admit good moduli spaces, by decomposing stacks of coherent sheaves on $S$ according to the torsion and HN filtrations of the sheaves they parameterise.  Initially we work over an arbitrary field $k$.  
\sssct
Firstly, we consider moduli stacks of pure-dimensional sheaves, but we relax the condition that those sheaves be semistable.  However, stability conditions will still play a role for us, in ensuring that the Borel--Moore homology is finite-dimensional in each cohomological degree.  So we fix, as above, an ample class $H \in \NS(\overline{S})_{\BR}$, where $S$ is a quasiprojective scheme.  Fix a polynomial $q(t)\in\BQ[t]$ with leading coefficient $t^i/i!$ for $0\leq i\leq 2$.  We denote by 
\[
\FM^{H[\geq q(t)]}(S)\subset \FM(S)
\]
the open substack of compactly supported coherent sheaves $\CF\in\Ob(\Coh(S))$ for which the reduced Hilbert polynomials appearing in the HN type of $\CF$ are all greater than or equal to $q(t)$.  For $P(t)\in\BQ[t]$ of the same degree as $q(t)$, we denote by 
\[
\FM^{H[\geq q(t)]}_{P(t)}(S)\subset\FM^{H[\geq q(t)]}(S)
\]
the substack of such sheaves with Hilbert polynomial $P(t)$.  The following lemma is standard, we provide the proof for completeness:
\begin{lemma}
\label{fin_HN}
Let $S$ be a quasiprojective surface with fixed ample class $H\in\NS(\overline{S})_{\BR}$.  Fix polynomials $P(t),q(t)\in\BQ[t]$ of the same degree, where $q(t)$ has leading term $x^d/d!$ for $d\leq 2$.  Then the set $\BS$ of possible HN types of pure-dimensional coherent sheaves on $S$ corresponding to points of $\FM^{H[\geq q(t)]}_{P(t)}(S)$ is finite.
\end{lemma}
\begin{proof}
There are three cases corresponding to the value of $d$.  The case $d=0$ is trivial, while the proof for the case $d=1$ is an easier version of the proof for the $d=2$ case, which we now provide.

\smallbreak
Let $\CF$ satisfy $P_{\CF}(t)=P(t)=a_2t^2/2+a_1t+a_0$.  We may assume that $\CF$ is not semistable; semistable shaves contribute precisely one HN type to $\BS$ (if they exist).  We show that there is a finite collection of possibilities for the Hilbert polynomial of the destabilising submodule $\CF_1\subset \CF$.  Let $P_{\CF_1}(t)=b_2t^2/2 +b_1t+b_0$ be such a polynomial.  By purity of the dimensions of the subquotients in the HN filtration, we have $0<b_2<a_2$.  Since $\CF_1$ is destabilising we have
\[
\frac{P_{\CF_1}(t)}{b_2}>\frac{P_{\CF}(t)}{a_2}.
\]
By our assumption on the HN type of $\CF$ there is an inequality
\begin{equation*}
\frac{P_{\CF}(t)-P_{\CF_1}(t)}{a_2-b_2}\geq q(t).
\end{equation*}
The first inequality bounds $b_1$ below, while the second bounds it above.  Inducting on the leading coefficient $a_2$, we deduce that there is a finite set $L$ of tuples of degree one polynomials such that if $(P_1(t),\ldots,P_l(t))$ is a HN type for a sheaf $\CE$ corresponding to a point in $\FM^{H[\geq q(t)]}_{P(t)}(S)$, there is some tuple $(E_1(t),\ldots,E_l(t))\in L$ and some tuple $(c_1,\ldots,c_l)\in\BQ^l$ such that $P_i(t)=tE_i(t)+c_i$ for all $i$,  and
\begin{equation}
\label{sum_const}
\sum_{i\leq l} c_i= a_0.
\end{equation}
In words, there is a finite set of possibilities for the tuple of leading and linear terms in a HN type of a coherent sheaf in $\FM^{H[\geq q(t)]}_{P(t)}(X)$, and so we just need to show that for each such choice, there is a finite set of choices of constant terms.  
\smallbreak
Fixing the Hilbert polynomial $P_{\CF'}(t)$ of a semistable sheaf $\CF'$, there are a finite set of possibilities for the Chern classes $c_1(\CF'),c_2(\CF')$ by boundedness of the moduli space.  By the Bogomolov inequality \cite[Thm.4.3]{MR2665168}, if $\CE'$ is a pure semistable sheaf on $S$ we have
\[
2r(\CE')c_2(\CE')-(r'-1)c_1(\CE')^2\geq 0
\]
(where the Chern classes are taken in $\HO(\overline{S},\BZ)$) bounding the constant term of $p_{\CF'}(t)$ below.  The result then follows from \eqref{sum_const}.
\end{proof}
\sssct
We record the following corollary of Lemma \ref{fin_HN}, which one may use to define the mixed Hodge structure on $\FM_{P(t)}^{H[\geq q(t)]}(X)$ as in \S \ref{sec:algebraic_MHMs}.
\begin{corollary}
\label{fts_p}
The stack $\FM_{P(t)}^{H[\geq q(t)]}(S)$ is a global quotient stack.
\end{corollary}
\begin{proof}
We may use the same construction as in \S \ref{MSCON}, as long as we can show that there is some $m\gg 0$ such that for all sheaves $\CF$ with Hilbert polynomial $P(t)$ satisfying the condition that the subquotients $\CQ$ in the HN filtration of $\CF$ all satisfy $p_{\CQ}(t)\geq q(t)$, the sheaf $\CF(m)$ is generated by global sections and has vanishing higher cohomology.  By finiteness of the set $\BS$ in Lemma \ref{fin_HN}, it is sufficient to show that for all sheaves $\CF$ with fixed HN type $(P_1(t),\ldots,P_l(t))$, there exists such an $m$.  
\smallbreak

We prove this for $l=2$, and we write $0\subset \CF_1\subset \CF_2=\CF$ for the HN filtration of $\CF$; the general case is an easy induction.  There exists an $m$ such that for all semistable coherent sheaves $\CE$ with Hilbert polynomial $P_1(t)$ or $P_2(t)$, $\CE(m)$ is generated by global sections and $\HO^i(S,\CE(m))=0$ for $i>0$.  Then the morphism of distinguished triangles
\[
\xymatrix{
\HO(S,\CF_1(m))\otimes_k\CO_S\ar[r]\ar[d]&\HO(S,\CF(m))\otimes_k\CO_S\ar[r]\ar[d]&\HO(S,\CF/\CF_1(m))\otimes_k\CO_S\ar[r]\ar[d]&\\
\CF_1(m)\ar[r]&\CF(m) \ar[r]&\CF/\CF_1(m)\ar[r]&
}
\]
demonstrates that $\CF(m)$ is generated by global sections and has vanishing higher cohomology.
\end{proof}

Combining Lemmas \ref{Nit_lemma} and \ref{fin_HN} we deduce
\begin{corollary}
\label{S_intro}
Let $S$ be a quasiprojective surface, and fix an ample class $H\in\NS(\overline{S})_{\BR}$.  Fix polynomials $q(t),P(t)$ with the same degree.  The moduli stack $\FM_{P(t)}^{H[\geq q(t)]}(S)$ can be written as a union of stacks $\FM_{\bm{\alpha}}^H(S)$ corresponding to coherent sheaves of HN type $\bm{\alpha}$, with $\bm{\alpha}$ belonging to a finite set $\BS$.  Moreover if we define
\[
\FM^{H[\geq q(t)]}_{\preccurlyeq \bm{\alpha}}(S)=\bigcup_{\substack{\bm{\beta}\in \BS\\\bm{\beta}\preccurlyeq \bm{\alpha}}}\FM^H_{\bm{\beta}}(S)
\]
then $\FM^{H[\geq q(t)]}_{\preccurlyeq \bm{\alpha}}(S)\subset \FM_{P(t)}^{H[\geq q(t)]}(S)$ is a closed substack, and $\FM^H_{\bm{\alpha}}(S)\subset \FM^{H[\geq q(t)]}_{\preccurlyeq \bm{\alpha}}(S)$ is open.
\end{corollary}

\sssct For the rest of \S \ref{gen_sheaves} we set $k=\BC$.

\begin{lemma}
\label{HN_pure}
Let $S$ be a smooth complex quasiprojective surface, and fix an ample class $H\in\NS(\overline{S})_{\BR}$.  Fix a HN type $\bm{\alpha}=(P_1(t),\ldots,P_r(t))$, and assume that for each $P_i(t)$ the mixed Hodge structure on $\HO_c(\FM_{P_i(t)}^{H\sstab}(S),\BQ)$ is pure, and assume that $\omega_S\cong\CO_S$, or more generally that the pair $(H,P_i(t))$ satisfies condition (*) from Proposition \ref{rel_pur_S} for each $1\leq i\leq r$.  Then $\HO_c(\FM^H_{\bm{\alpha}}(S),\BQ)$ is pure.
\end{lemma}
\begin{proof}
As indicated in the proof of \cite[Cor.4.1]{HL16}, the idea is to use that the morphism $\FM^H_{\bm{\alpha}}(S)\rightarrow \FM_{P_1(t)}^{H\sstab}(S)\times\ldots\times \FM_{P_l(t)}^{H\sstab}(S)$ is quasi-smooth and induces an isomorphism in Borel--Moore homology (up to an overall Tate twist), by combining generalisations of \eqref{Aff_1}, \eqref{Aff_2} and \eqref{K_iso} below.  We spell out the details here.  For ease of presentation we assume that $\bm{\alpha}=(P_1(t),P_2(t))$ has length two; the proof for more general HN types is then a simple induction.  Also, for ease of presentation, we will assume that $S$ is projective, the general proof proceeds as in the definition of $\FM_{P(t)}^{H\sstab}(S)$ by working with coherent sheaves on the compactification $\overline{S}$ that have support contained in $S$.
\smallbreak
In the construction of the moduli stacks $\FM_{P_1(t)}^{H\sstab}(S)$ and $\FM_{P_2(t)}^{H\sstab}(S)$ we fix sufficiently large numbers $m_1$ and $m_2$ so that $\CF(m_i)$ is globally generated whenever $\CF$ is a semistable sheaf with Hilbert polynomial $P_i(t)$, and also satisfies $\ho^n(\CF)=0$ for $n>0$.  Replacing $m_1$ and $m_2$ by $\max(m_1,m_2)$, we may assume that $m=m_1=m_2$.  As in \S \ref{MSCON} we define
\[
V=\BC^{\oplus P_{2}(m)};\quad\quad
\CH=V\otimes\CO_S.  
\]
By construction, on the stack $\FM_{P_2(t)}^{H\sstab}(S)\times S$ we have a short exact sequence of coherent sheaves
\[
0\rightarrow \CK\rightarrow \overline{\CH}\xrightarrow{f} \CF_{P_2(t)}\rightarrow 0
\]
where $\overline{\CH}$ is induced by the $\Gl_V$-equivariant sheaf $\pi_S^*\CH$ on $\Quot_S(\CH,P_2(t))\times S$, and $f$ is the quotient morphism to the universal sheaf $\CF_{P_2(t)}$.  
For $i,j\in\{1,2,3\}$ we denote by $\pi_{ij}$ the projection to the $i$th and $j$th factors in the product
\[
\FM_{P_1(t)}^{H\sstab}(S)\times \FM_{P_2(t)}^{H\sstab}(S)\times S.
\]
We claim that both of the complexes
\[ 
\CE_1\coloneqq \pi_{12,*}\SRHom(\pi_{23}^*\overline{\CH}, \pi_{13}^*\CF_{P_1(t)});\quad\quad
\CE_0\coloneqq \pi_{12,*}\SRHom(\pi_{23}^*\CK, \pi_{13}^*\CF_{P_1(t)})
\]
are concentrated in degree zero, and are moreover locally free.  The first claim follows directly from our assumption on $m$.  For the second, consider a geometric point $s\colon\Spec(K)\rightarrow \FM_{P_1(t)}^{H\sstab}(S)\times \FM_{P_2(t)}^{H\sstab}(S)$ corresponding to a pair of semistable sheaves $\CG_1$ and $\CG_2$ on $S_K$ with Hilbert polynomials $P_1(t)$ and $P_2(t)$ respectively.  Then for $i\geq 0$
\[
\Ho^i(s^*\CE_0)\cong \Ext^{i+1}_{S_K}(\CG_2,\CG_1)\cong \Ext^{1-i}_{S_K}(\CG_1,\CG_2)^{\vee}
\]
by Serre duality.  Since $\CG_1$ and $\CG_2$ are semistable coherent sheaves, with $p_{\CG_1}(t)>p_{\CG_2}(t)$, we deduce that this final term is zero for $i>0$, as required.  In the diagram
\[
\xymatrix{
\FM^H_{\bm{\alpha}}(S)&\ar[l]_-{\cong}\Tot(\CE_1/\CE_0)& \ar[l]_-{a}\Tot(\CE_1)\ar[r]^-b &\FM^{H\sstab}_{P_1(t)}(S)\times \FM^{H\sstab}_{P_2(t)}(S)
}
\]
both of the morphisms $a$ and $b$ are affine fibrations, so there are isomorphisms of mixed Hodge structures
\begin{align}
\label{Aff_1}
\HO_c(\FM^H_{\bm{\alpha}}(S),\BQ)\otimes\LLL^{\dim(a)}\cong& \HO_c(\Tot(\CE_1),\BQ)\\
\label{Aff_2}
\HO_c(\FM^{H\sstab}_{P_1(t)}(S)\times \FM^{H\sstab}_{P_2(t)}(S),\BQ)\otimes\LLL^{\dim(b)}  \cong& \HO_c(\Tot(\CE_1),\BQ).
\end{align}
Via the Kunneth isomorphism
\begin{equation}
\label{K_iso}
\HO_c(\FM^{H\sstab}_{P_1(t)}(S)\times \FM^{H\sstab}_{P_2(t)}(S),\BQ)\cong \HO_c(\FM^{H\sstab}_{P_1(t)}(S),\BQ)\otimes \HO_c(\FM^{H\sstab}_{P_2(t)}(S),\BQ)
\end{equation}
and the isomorphisms \eqref{Aff_1} and \eqref{Aff_2} the mixed Hodge structure $\HO_c(\FM^H_{\bm{\alpha}}(S),\BQ)$ is pure if both $\HO_c(\FM^{H\sstab}_{P_1(t)}(S),\BQ)$ and $\HO_c(\FM^{H\sstab}_{P_2(t)}(S),\BQ)$ are.
\end{proof}
%The isomorphisms \eqref{Aff_1}, \eqref{Aff_2} and \eqref{K_iso} appear also in \cite{HL16}, along the way to proving there that the Borel--Moore Homology $\HO_c(\FM^H_{\bm{\alpha}}(S),\BQ)$ can be written as a direct sum of the Borel--Moore homology of the strata appearing in the decomposition into HN types.  One of the principal motivations for the present paper was to explore the Hodge theoretic speculations of \cite[\S 4.1.1]{HL16}; as envisaged there, we can use our purity results to upgrade this direct sum decomposition according to HN types to a decomposition of cohomologically graded mixed Hodge structures.
\sssct Putting everything together, we can finally prove purity for more general stacks of coherent sheaves:
\begin{theorem}
\label{purity_ohne}
Let $S$ be a smooth complex projective surface satisfying $\CO_S\cong\omega_S$, let $H\in\NS(S)_{\BR}$ be an ample class, and let $P(t)$ and $q(t)$ be polynomials of the same degree, where the leading term of $q(t)$ is $t^d/d!$ for some $d\leq 2$.  Then the mixed Hodge structure on $\HO^{\BM}(\FM^{H[\geq q(t)]}_{P(t)}(S),\BQ)$ is pure.
\end{theorem}
\begin{proof}
Define the set $\BS=\{\bm{\beta}_1,\ldots,\bm{\beta}_m\}$ as in Corollary \ref{S_intro}, which we label so that $\bm{\beta}_i\preccurlyeq\bm{\beta}_j$ if $i<j$.  Set
\[
\FN_i\coloneqq \bigcup_{j\leq i}\FM^{H[\geq q(t)]}_{\bm{\beta}_j}(S).
\]
Then by Corollary \ref{S_intro}, each $\FN_i\subset \FM^{H[\geq q(t)]}_{P(t)}(S)$ is a closed substack, and moreover the inclusion
\[
\FM_{\bm{\beta}_i}^{H}(S)=\FN_i\setminus \FN_{i-1}\hookrightarrow \FN_i
\]
is open.  We show by induction that each $\HO_c(\FN_i,\BQ)$ is pure.  For the base case note that $\FN_1=\FM_{\bm{\beta}_1}^{H}(S)$ has pure compactly supported cohomology by Lemma \ref{HN_pure} and Theorem \ref{HL_con}.  For the induction step, for each $i$ and for each $n\in\BZ$  there is a long exact sequence
\[
\rightarrow \HO_c^n(\FM_{\bm{\beta}_i}^{H}(S),\BQ)\rightarrow \HO_c^n(\FN_i,\BQ)\rightarrow \HO_c^n(\FN_{i-1},\BQ)\rightarrow
\]
The mixed Hodge structure $\HO_c^n(\FN_{i-1},\BQ)$ is pure of weight $n$ by the induction hypothesis, while $\HO_c^n(\FM_{\bm{\beta}_i}^{H}(S),\BQ)$ is pure of weight $n$ by another application of Lemma \ref{HN_pure} and Theorem \ref{HL_con}.  We deduce that $\HO_c^n(\FN_i,\BQ)$ is pure of weight $n$ as required.  In particular, $\HO_c(\FM^{H[\geq q(t)]}_{P(t)}(S),\BQ)=\HO_c(\FN_m,\BQ)$ is pure, as is its graded dual $\HO^{\BM}(\FM^{H[\geq q(t)]}_{P(t)}(S),\BQ)$.
\end{proof}
\begin{remark}
\label{thm_exten}
More generally, Theorem \ref{purity_ohne} remains true, with the same proof, under the weaker condition that $(H,P'(t))$ satisfies condition (*) of Proposition \ref{rel_pur_S} for every $P'(t)$ appearing in the HN types contained in $\BS$.
\end{remark}
\sssct
We now assume that $S$ is a smooth complex projective surface satisfying $\CO_S\cong\omega_S$.  Fix a degree $d\leq 2$ and let $q(t)<h(t)$ be a pair of degree $d$ normalised polynomials.  Then there is an open inclusion of stacks
\[
\FM^{H[\geq h(t)]}_{P(t)}(S)\hookrightarrow \FM^{H[\geq q(t)]}_{P(t)}(S)
\]
with complement $\FN$ a finite union of stacks $\FM^{H}_{\bm{\alpha}}$ for HN types $\bm{\alpha}=(E_1(t),\ldots,E_r(t))$ satisfying $h(t)>e_r(t)\geq q(t)$ (where we denote by $e_i(t)$ the normalisation of $E_i(t)$).  We have again used Lemma \ref{fin_HN} for the finiteness part of this statement.  
\smallbreak
Via the same inductive proof as Theorem \ref{purity_ohne}, the mixed Hodge structure on $\HO_c(\FN,\BQ)$ is pure, so that in the long exact sequence
\[
\rightarrow \HO^n_c\left(\FM^{H[\geq h(t)]}_{P(t)}(S),\BQ\right)\xrightarrow{\lambda_n} \HO^n_c\left(\FM^{H[\geq q(t)]}_{P(t)}(S),\BQ\right)\rightarrow \HO^n_c\left(\FN,\BQ\right)\xrightarrow{\psi}
\]
the morphism $\psi$ is a morphism from a pure weight $n$ Hodge structure to a pure weight $n+1$ Hodge structure.  It follows that $\psi=0$ and the long exact sequence breaks up into split short exact sequences.  The resulting direct sum decomposition according to HN strata itself can be derived as a special case of \cite[Cor.4.1]{HL16}.  In particular, $\lambda_n$ is an inclusion, and we have proved the following
\begin{proposition}
Let $S$ be a smooth complex projective surface satisfying $\CO_S\cong\omega_S$, let $d\leq 2$, and let $P(t)\in\BQ[t]$ be a degree $d$ polynomial.  We write $p(t)$ for the normalisation of $P(t)$.  Let $\FM^{\mathrm{pdim }\:d}_{P(t)}(S)$ denote the stack of pure $d$-dimensional coherent sheaves on $S$ with Hilbert polynomial $P(t)$.  Then 
\[
\HO_c\left(\FM^{\mathrm{pdim }\:d}_{P(t)}(S),\BQ\right)=\bigcup_{q(t)\leq p(t)}\HO_c\left(\FM^{H[\geq q(t)]}_{P(t)}(S),\BQ\right),
\]
and so for every $j\in\BZ$, $\HO^j_c\left(\FM^{\mathrm{pdim }\:d}_{P(t)}(S),\BQ\right)$ can be written as an infinite nested union of pure Hodge structures of weight $j$.
\end{proposition}
\sssct
Next, we relax the condition on the purity of the support of coherent sheaves on $S$.  For simplicity we consider sheaves in the category $\Coh_{\leq 1}(S)$, i.e. the category of coherent sheaves with one-dimensional support (which we do not assume is pure).  We fix an ample class $H$.  For $\CF\in\Ob(\Coh_{\leq 1}(S))$ we define the \textit{slope} $\mu(\CF)=a_0/a_1$, where $P_{\CF}(t)=a_1t+a_0$.  Since we do not assume purity of the support of $\CF$, the slope may be infinite.  For a coherent sheaf $\CF\in\Ob(\Coh_{\leq 1}(S))$, the condition of being semistable and pure is equivalent to slope stability, i.e. the condition that for all proper nonzero subsheaves $\CF'\subset \CF$ we have an inequality
\[
\mu(\CF')\leq \mu(S).
\]
Let $l\in\BQ$.  We define $\FM^{\mu\geq l}_{(\beta,n)}(S)\subset \FM_{(\beta,n)}(S)$ to be the stack of coherent sheaves $\CF$ with characteristic cycle $[\CF]=\beta\in\HO_2(S,\BZ)$ and with $\chi(\CF)=n$, such that there are no subquotients of slope less than $l$ in the (Gieseker) HN filtration of $\CF/\CF_{(0)}$, or equivalently the (slope) HN filtration of $\CF$.  Here we denote by $\CF_{(0)}\subset \CF$ the zero-dimensional sheaf appearing in the torsion filtration of $\CF$.  
\smallbreak
Via the same argument as the proof of Lemma \ref{fin_HN}, if $\CF$ is a coherent sheaf such that the slopes of the subquotients in the HN filtration of $\CF/\CF_{(0)}$ are bounded below, there is an upper bound $M$ on the length of $\CF_{(0)}$, and there is a finite list of possibilities for the HN type of $\CF/\CF_{(0)}$ (again by Lemma \ref{fin_HN}).  Arguing as in Corollary \ref{fts_p}, the stack $\FM^{\mu\geq l}_{(\beta,n)}(S)$ is a global quotient stack.
\smallbreak
As a final application of the purity theorem to coherent sheaves, we have
\begin{theorem}
\label{1dim_pure}
Let $S$ be a smooth complex projective surface with fixed ample class $H$.  Fix $\beta\in \HO_2(S,\BZ)$, and $n\in\BZ$, and assume that for all decompositions $\beta=\beta_1+\ldots+\beta_r$ into effective classes, $\beta_i\cdot \omega_S=0$ for all $1\leq i\leq r$.  Let $l\in\BQ$.  Then the mixed Hodge structure on $\HO^{\BM}(\FM^{\mu\geq l}_{(\beta,n)}(S),\BQ)$ is pure.
\end{theorem}
\begin{proof}
We define a function on the stack $\FM^{\mu\geq l}_{(\beta,n)}(S)$ that takes a point $\CF$ to the length of $\CF_{(0)}$.  We claim that this function is upper semicontinuous.  This follows from the fact that (by Lemma \ref{tors_grab}) we may define this function as $\ho^0(\CF\otimes \CO(-NH))$ for $N\gg 0$.  We define $\FT_m\subset \FM^{\mu\geq l}_{(\beta,n)}(S)$ to be the substack of sheaves for which the length of the torsion part is $m$.  Then we have shown that the inclusion
\[
\FT_{\geq m}\coloneqq \bigcup_{m'\geq m}\FT_{m'}\subset \FM^{\mu\geq l}_{(\beta,n)}(S)
\]
is closed, while the inclusion $\FT_m\subset \FT_{\geq m}$ is open.  Via the same argument as the proof of Lemma \ref{HN_pure}, there is an isomorphism of mixed Hodge structures
\begin{equation}
\label{tors_off}
\HO_c(\FT_m,\BQ)\cong \HO_c(\FM_{(0,m)}(S),\BQ)\otimes \HO_c(\FM_{(\beta,n-m)}^{H[\geq t+l]}(S),\BQ)\otimes\LLL^{\gamma}
\end{equation}
where $\gamma\in\BZ$ depends only on $P(t)$ and $m$\footnote{Although we do not use it, it is not hard to show that in fact $\gamma=0$.  This can be seen by considering the special case in which the support of $\CF_{(0)}$ is disjoint from the support of $\CF/\CF_{(0)}$ and the fact that $\gamma$ only depends on $P(t)$ and $m$.}.  By Theorem \ref{HL_con} the first factor on the right hand side of \eqref{tors_off} is pure, and by Theorem \ref{purity_ohne} (and Remark \ref{thm_exten}) the second factor is as well.  As in the proof of Theorem \ref{purity_ohne} we show by induction on $M-m$ that each $\HO_c(\FT_{\geq m},\BQ)$ is pure, and the result follows, since $\FT_{\geq 0}=\FM^{\mu\geq l}_{(\beta,n)}(S)$.
\end{proof}
\subsection{Higgs sheaves}
\label{Higgs_sec}
Let $C$ be a smooth projective curve over $k$, an algebraically closed field with $\chara(k)=0$.  We recall that a Higgs sheaf on $C$ is a pair 
\[
\overline{\CF}=\big(\mathcal{F}\in\Ob(\Coh(C)),\phi\in\Hom(\mathcal{F},\mathcal{F}\otimes\omega_C)\big).
\]
The element $\phi$ is called the \textit{Higgs field}.  The rank $r(\overline{\CF})$ and degree $d(\overline{\CF})$ of a Higgs sheaf are defined to be the rank and degree of the underlying coherent sheaf $\mathcal{F}$.  Note that we allow the rank to be zero, i.e. $\CF$ may be a torsion sheaf.  Higgs sheaves form an Abelian category $\CatHiggs(C)$.

\smallbreak
We define the slope of a Higgs sheaf by $\mu(\overline{\CF})=d(\overline{\CF})/r(\overline{\CF})$.  Since $\CF$ may be a torsion sheaf, the slope may be $\infty$.  We say that a Higgs sheaf $\overline{\CF}=(\CF,\phi)$ is stable (respectively semistable) if it has no proper nonzero Higgs subsheaf $\overline{\CG}$ with $\mu(\overline{\CG})\geq \mu(\overline{\CF})$ (respectively $\mu(\overline{\CG})> \mu(\overline{\CF})$).  Equivalently, a Higgs sheaf is semistable if there is no proper nonzero subsheaf $\CG\subset \CF$, preserved by $\phi$, with strictly larger slope than $\CF$.

\sssct
The coarse moduli space of Higgs bundles $\CM^{\Dol}_{r,d}(C)$ was constructed by Nitsure as a GIT quotient \cite{Nit91}, and via this construction one obtains a commutative diagram

\[
\xymatrix{
J/G\ar[d]_p\ar[r]^-{\cong}&\FM^{\Dol}_{r,d}(C)\ar[d]\\
J/\!\!/_{\chi}G\ar[r]^-{\cong}&\CM^{\Dol}_{r,d}(C)
}
\]
for $J$ a space of (semistable) pairs of a morphism $V\otimes\mathcal{O}_C(-m)\rightarrow \CE$ and a Higgs field on $\CE$.  The morphism $p$ is the morphism to a certain GIT quotient, defined with respect to a linearisation $\chi$ of the $G$-action.  

\subsubsection{BNR correspondence} It will be convenient for us to adopt a slightly different approach to the construction of the moduli stack of Higgs sheaves and its coarse moduli scheme, via the Beauville--Narasimhan--Ramanan correspondence (see \cite{BNR} for background).
\smallbreak

Set $S=\Tot(\omega_C)$, and $\ol{S}=\BP_C(\mathcal{O}_C\oplus \omega_C^{\vee})$.  We embed $S$ inside $\overline{S}$ in the natural way.  Write $\pi\colon S\rightarrow C$ and $\ol{\pi}\colon\ol{S}\rightarrow C$ for the projections, and $D_{\infty}\subset \ol{S}=\overline{S}\setminus S$ for the divisor at infinity.  There is an equivalence of categories between the category of coherent sheaves on $S$ for which $\Ho^0\!\pi_*S$ is coherent and the category of Higgs sheaves, and so in turn an equivalence of categories between the category of Higgs sheaves on $C$, and the category of coherent sheaves $\mathcal{F}$ on $\ol{S}$ with $\supp(\mathcal{F})\cap D_{\infty}=\emptyset$.  Coherency of $\overline{\pi}_*\CF\cong \Ho^0\ol{\pi}_*\CF$ is automatic for such $\CF$.
\smallbreak
Picking an ample class $H\in\NS(\ol{S})_{\BR}$, one may in turn identify the category of semistable Higgs sheaves on $C$ with the category of $H$-Gieseker semistable coherent sheaves on $\ol{S}$ with support avoiding $D_{\infty}$.  For the derived category of Higgs sheaves, we consider the dg enhancement $\Db_{\dg}(\CatHiggs(C))$ given by considering Higgs sheaves as a full subcategory of quasicoherent sheaves on $\ol{S}$.  In other words we define $\Db_{\dg}(\CatHiggs(C))$ to be the full subcategory of $\Db_{\dg}(\Coh(\ol{S}))$ containing those complexes for which the total cohomology has support avoiding $D_{\infty}$.  By \cite{LO10}, the dg enhancement of $\Db(\Coh(\ol{S}))$ is essentially unique (note that the perfect derived category and the bounded derived category of coherent sheaves on $\ol{S}$ coincide, since $\overline{S}$ is smooth).
\smallbreak
\subsubsection{Formality}
By Corollary \ref{form_cor} we deduce:
\begin{proposition}
Let $C$ be a smooth connected projective curve over $k$.  Let $\BS=\{\CF_1,\ldots,\CF_l\}$ be a $\Sigma$-collection inside $\Db(\CatHiggs(C))$, and let $\mathscr{C}$ be the full subcategory of $\Db_{\dg}(\CatHiggs(C))$ containing $\BS$.  Then $\mathscr{C}$ is formal.  In particular, if $\CF=\bigoplus_{i\in I}\CF_i^{\mathbf{d}_i}$ is a polystable Higgs sheaf, where each $\CF_i$ has injective resolution $\CI_i^{\bullet}$, then the dg Yoneda algebra 
\[
\REnd(\CF,\CF)\coloneqq \Hom\left(\bigoplus_{i\in I}(\CI_i^{\bullet})^{\mathbf{d}_i},\bigoplus_{i\in I}(\CI_i^{\bullet})^{\mathbf{d}_i}\right)
\]
is formal.
\end{proposition}
\subsubsection{Purity} For the rest of \S \ref{Higgs_sec} we assume that $k=\BC$.  As in the introduction, we write
\begin{equation}
\label{pHiggs}
p \colon \FM^{\Dol}_{r,d}(C) \to \CM^{\Dol}_{r,d}(C)
\end{equation}
for the morphism from the stack to the coarse moduli space, and $\Hit\colon \CM^{\Dol}_{r,d}(C)\rightarrow \Lambda_r$ for the Hitchin map.  We do \textit{not} assume that $(r,d)=1$.  Applying Theorem \ref{loc_pur_thm} to \eqref{pHiggs}, we deduce the following.
\begin{proposition}
\label{Higgs_loc_pur}
The complex $p_!\ul{\BQ}_{\FM^{\Dol}_{r,d}(C)}\in\Ob(\Dblf(\MHM(\CM^{\Dol}_{r,d}(C))))$ is pure, and satisfies
\[
\bm{\tau}^{\geq 1}\left(p_!\ul{\BQ}_{\FM^{\Dol}_{r,d}(C)}\otimes\LLL^{(1-g)r^2}\right)=0.
\]
\end{proposition}
Applying Proposition \ref{intr_pur} and the above proposition we deduce the following
\begin{proposition}
The complex $(\Hit \!p)_!\ul{\BQ}_{\FM^{\Dol}_{r,d}(C)}$ is pure.
\end{proposition}
The following is our main theorem on the mixed Hodge structure of moduli stacks of Higgs sheaves.
\begin{theorem}
\label{Higgs_pure}
Let $C$ be a smooth connected complex projective curve and let $\FM^{\Dol}_{r,d}(C)$ denote the stack of semistable rank $r$ and degree $d$ Higgs sheaves on $C$. Then the mixed Hodge structure on $\HO^{\BM}(\FM^{\Dol}_{r,d}(C),\BQ)$ is pure.
\end{theorem}
This deals with case (3) of Theorem \ref{thm:purity_absolute}, and thus completes the proof.  Note that this result does not follow directly from Theorem \ref{prop_cor}, since the coarse moduli space of Higgs sheaves $\CM^{\Dol}_{r,d}(C)$ is not projective.  

\begin{proof}[Proof of Theorem \ref{Higgs_pure}]
By Proposition \ref{Higgs_loc_pur}, the complex $p_!\ul{\BQ}_{\FM^{\Dol}_{r,d}(C)}$ is pure.  By \eqref{callback} it is enough to show that
\[
\HO_c\!\left(\CM^{\Dol}_{r,d}(C),\Ho^i(p_!\ul{\BQ}_{\FM^{\Dol}_{r,d}(C)})[-i]\right)
\]
is pure for every $i\in\mathbb{Z}$.
\smallbreak
We fix the class $H=3F+Z\in\NS(\overline{S})_{\BR}$, where $F=\overline{\pi}^{-1}(p)$ for some point $p\in C$, and $Z\subset\Tot(\omega_C)= S\subset \overline{S}$ is the zero section.  Set $Q(t)=(2g+1)rt+d+(1-g)r$.  Via the BNR correspondence we have the Cartesian diagram
\[
\xymatrix{
\FM^{\Dol}_{r,d}(C)\ar@{^{(}->}[d]\ar[r]^-p&\CM^{\Dol}_{r,d}(C)\ar@{^{(}->}[d]
\\
\FM_{Q(t)}^{H\sstab}(\ol{S})\ar[r]&\CM_{Q(t)}^{H\sstab}(\ol{S}).
}
\]
In the notation of \S \ref{MSCON}, we write $U'\subset \Quot_{\ol{S}}(\CH,Q(t))$ for the open subscheme parameterising those quotients $\CH\onto \CE$ such that $\CE$ is semistable and for which $\supp(\CE)\cap D_{\infty}=\emptyset$.  Let $\BC^*$ act on $\ol{S}=\BP_{C}(\mathcal{O}_C\oplus \omega_C^{\vee})$ by scaling the $\omega_C^{\vee}$-factor and leaving the $\CO_C$-factor invariant.  Then $U'$ is $\BC^*$-invariant for the induced $\BC^*$-action on $\Quot_{\ol{S}}(\CH,Q(t))$.  
\smallbreak
Fix $i\in\BZ$.  Now pick $M\gg 0$, so that we have 
\begin{equation}
\label{congy}
\CG\coloneqq \Ho^i(p_!\ul{\BQ}_{\FM^{\Dol}_{r,d}(C)})\cong\Ho^i(q_!\ul{\BQ}_{Y}\otimes\LLL^{-M\dim(V)})
\end{equation}
where $Y=U'\times_{\Gl_V} H'$ for $H'\subset \Hom(\BC^M,V)$ the subspace of surjective maps and
\[
q\colon Y\rightarrow \CM^{\Dol}_{r,d}(C)
\]
the projection.  We refer the reader to \S \ref{heart_di} for this construction, and the isomorphism \eqref{congy}.

\smallbreak

Letting $\BC^*$ act on $Y$ via the given action on $U'$ and the trivial action on $H'$, the morphism $q$ is $\BC^*$-equivariant, and so the underlying perverse sheaf of $\CG$ is $\BC^*$-equivariant.  Since purity is stable under Verdier duality, purity of $\HO_c(\CM^{\Dol}_{r,d}(C),\CG[-i])$ is equivalent to purity of $\HO(\CM^{\Dol}_{r,d}(C),(\BD\CG)[i])$.  The $\BC^*$-action contracts $\CM^{\Dol}_{r,d}(C)$ onto the projective subscheme $\ul{Z}\subset \CM_P^{H\sstab}(\ol{S})$ corresponding to coherent sheaves with set-theoretic support on $C\subset \ol{S}$.  It follows that the restriction map
\begin{equation}
\label{Higgs_sandwich}
\HO(\CM^{\Dol}_{r,d}(C),\BD\CG[i])\rightarrow \HO(\ul{Z},(\BD\CG[i])\lvert_{\ul{Z}})
\end{equation}
is an isomorphism.  Since $\CG$ is pure of weight $i$, $\BD\CG$ is pure of weight $-i$, i.e. $\BD\CG[i]$ is pure.  Since $(\CM^{\Dol}_{r,d}(C)\rightarrow \pt)_*$ increases weights \cite{Sai90}, we deduce that
\begin{align}
\label{bound1}
\Gr^W_a\!\left(\HO^b\!\left(\CM^{\Dol}_{r,d}(C),\BD\CG[i]\right)\right)=0 &\textrm{ if }a<b.
\end{align}
On the other hand, the restriction morphism $(\ul{Z}\rightarrow \CM^{\Dol}_{r,d}(C))^*$ decreases weights, and the morphism $\ul{Z}\rightarrow \pt$ is projective, and hence preserves weights.  So we deduce, again from purity of $\BD\CG[i]$, that 
\begin{align}
\label{bound2}
\Gr_a^W\!\left(\HO^b\!\left(\ul{Z},(\BD\CG[i])\lvert_{\ul{Z}}\right)\right)=0\textrm{ if } a>b.
\end{align}
Combining \eqref{bound1} and \eqref{bound2} and the fact that \eqref{Higgs_sandwich} is an isomorphism gives the result.
\end{proof}

\sssct
If $(r,d)=1$ then $\FM^{\Dol}_{r,d}(C)\cong \CM^{\Dol}_{r,d}(C)/\BC^*$ and purity of the Borel--Moore homology of $\FM^{\Dol}_{r,d}(C)$ follows from the isomorphism
\[
\HO^{\BM}\!\left(\FM^{\Dol}_{r,d}(C),\BQ\right)\cong \HO^{\BM}\!\left(\CM^{\Dol}_{r,d}(C),\BQ\right)\otimes\HO\left(\pt/\BC^{*},\BQ\right),
\]
purity of $\HO(\pt/\BC^{*},\BQ)$, and purity of $\HO^{\BM}(\CM^{\Dol}_{r,d}(C),\BQ)$, proved in \cite{HT03}.  So the advance here is in establishing purity for (possibly highly singular and stacky) cases in which $(r,d)\neq 1$.
\sssct
If $\FX$ is a stack, exhausted by global quotient stacks, and the mixed Hodge structure on $\HO^{\BM}(\FX,\BQ)$ is pure, then the Borel--Moore homology is completely determined by the class $[\HO^{\BM}(\FX,\BQ)]$ in the completed Grothendieck group $\widehat{\KK_0}(\MHS)$.  Here we have to take a completion with respect to weights, since $\HO^{\BM}(\FX,\BQ)$ will generally be infinite-dimensional, although $\Gr_n^W(\HO^{\BM}(\FX,\BQ))$ will be finite-dimensional for each $n$.  There is a homomorphism $\KK_0(\mathrm{Sta}_{\mathrm{Aff}})\rightarrow \widehat{\KK_0}(\MHS)$ from the naive Grothendieck group of finite type stacks with affine geometric stabilisers, to the Grothendieck ring of mixed Hodge structures, which sends the class $[\FX]$ to $[\HO^{\BM}(\FX,\BQ)]$.  Since in the case $\FX=\FM^{\Dol}_{r,d}(C)$, thanks to \cite{FSS17}, we have an effective means of calculating $[\FX]$, it follows that we can use Theorem \ref{Higgs_pure} to actually determine the Borel--Moore homology of $\FM^{\Dol}_{r,d}$, along with its mixed Hodge structure.

\subsubsection{The global nilpotent cone}
We let $\FM^{\Dol,\nil}_{r,d}(C)\subset \FM^{\Dol}_{r,d}(C)$ denote the substack of Higgs sheaves $(\CF,\phi)$ for which $\phi$ is nilpotent.  This is the \textit{semistable global nilpotent cone} \cite{La88}.  We note a corollary of the proof of Theorem \ref{Higgs_pure}:
\begin{proposition}
\label{GNC_prop}
The natural mixed Hodge structure on the semistable global nilpotent cone
\begin{equation}
    \label{nilp_pure}
\HO^{\BM}\!\left(\FM^{\Dol,\nil}_{r,d}(C),\BQ\right)
\end{equation}
is pure.
\end{proposition}
Strictly speaking, we have defined $\FM^{\Dol,\nil}_{r,d}(C)$ as a formal stack.  For instance for $r=1$, it is easy to see that $\FM^{\Dol,\nil}_{r,d}(C)$ is a trivial $\BC^*$-quotient of the formal neighbourhood of $\Jac_d(C)\times \{0\}$ in $\Jac_d(C)\times\HO^0(C,\omega_C)$.  This formal stack has the same set of closed points as the stack
\[
\FM^{\Dol}_{r,d}(C)\times_{\Lambda_r}\{0\},
\]
and we may take the Borel--Moore homology of this stack as our definition in \eqref{nilp_pure}.
\begin{proof}[Proof of Proposition \ref{GNC_prop}]
Let $\iota\colon \{0\}\hookrightarrow \Lambda_r$ be the inclusion, and let $\tau\colon \Lambda_r\rightarrow \pt$ be the structure morphism.  By base change it is enough to show that
\begin{equation}
    \label{nilp_sandw}
\iota^*(\Hit\!p)_!\ul{\BQ}_{\FM^{\Dol}_{r,d}(C)}\cong \tau_*(\Hit\!p)_!\ul{\BQ}_{\FM^{\Dol}_{r,d}(C)}
\end{equation}
is pure, where the isomorphism \eqref{nilp_sandw} follows from the fact that the $\BC^*$-action scaling the Higgs field contracts $\Lambda_r$ to $0$.  Purity of $(\Hit\!p)_!\ul{\BQ}_{\FM^{\Dol}_{r,d}(C)}$ follows from purity of $p_!\ul{\BQ}_{\FM^{\Dol}_{r,d}(C)}$, projectivity of $\Hit$, and Proposition \eqref{intr_pur}.  Now purity of \eqref{nilp_sandw} follows by the same argument as Theorem \ref{Higgs_pure}, since $\iota^*$ increases weights and $\tau_*$ decreases weights (see \cite[\S 4.5]{Sai90}).
\end{proof}
\subsubsection{Perverse filtrations}
Another application of the purity of $p_!\ul{\BQ}_{\FM^{\Dol}_{r,d}(C)}$ and $(\Hit\!p)_!\ul{\BQ}_{\FM^{\Dol}_{r,d}(C)}$ is the following consequence of Propositions \ref{int_filt} and \ref{PIS_prop}:
\begin{proposition}
For arbitrary $d,r$, the Borel--Moore homology
\[
\HO^{\BM}\!\left(\FM^{\Dol}_{r,d}(C),\BQ\otimes\LLL^{(1-g)r^2}\right)
\]
carries a perverse filtration, by pure Hodge structures, with respect to $p$, and with respect to $\Hit \!p$.  In the case $(r,d)=1$, the perverse filtration with respect to $\Hit \!p$ is the usual perverse filtration associated to the Hitchin system.  The perverse filtration with respect to $p$ begins in degree zero.
\smallbreak
Similarly, the Borel--Moore homology $\HO^{\BM}(\FM^{\Dol,\nil}_{r,d}(C),\BQ\otimes\LLL^{(1-g)r^2})$ carries a perverse filtration with respect to $p$ and with respect to $\Hit \!p$, and the perverse filtration with respect to $p$ begins in degree zero.
\end{proposition}
\begin{remark}
Since we work throughout with the perverse filtration by mixed Hodge structures induced by Saito's theory, in order to relate to the ``usual'' filtration by pure Hodge structures in the coprime case, we are tacitly invoking the comparison theorem \cite[Thm.4.3.5]{dC10}.
\end{remark}
\begin{remark}
The perverse filtration on $\HO^{\BM}(\FM^{\Dol,\nil}_{r,d}(C),\BQ)$ with respect to $\Hit \!p$ is highly nontrivial, despite the fact that $\Hit\! p$ maps all of $\FM^{\Dol,\nil}_{r,d}(C)$ to a point; see the warning in \S \ref{Serre_sec} for some explanation of this. 
\end{remark}
\sssct
The fact that the perverse filtration on $\HO^{\BM}\!\left(\FM^{\Dol}_{r,d}(C),\BQ\otimes\LLL^{(1-g)r^2}\right)$ is a filtration by pure Hodge structures is already known in the case $(r,d)=1$.  One may show this using the language of mixed Hodge modules (applying a very special case of the arguments above, but where purity of $p_!\underline{\BQ}_{\FM^{\Dol}_{r,d}(C)}$ is trivial) or by using the alternative definition of the perverse filtration with respect to the morphism $\Hit$ in terms of the flag filtration; see \cite[Sec.7]{dCMi10} for details, and \cite{dCMi10,dC10} for the equivalence between these two approaches.

\subsubsection{Unstable Higgs sheaves}
\label{UnstHS}
Arguing as in \S \ref{gen_sheaves}, we can modify Theorem \ref{Higgs_pure} to allow moduli stacks of possibly unstable Higgs bundles:
\begin{proposition}
\label{Higgs_nest}
Let $C$ be a smooth complex projective curve.  Fix a number $l\in\BQ$, and denote by $\CHiggs^{\mu\geq l}_{r,d}(C)$ the moduli stack of Higgs sheaves $\overline{\CF}$ satisfying the condition that all of the subquotients in the slope slope HN filtration\footnote{Note that since we work with $\mu$-stability here, the first term in the HN filtration is the torsion part of $\overline{\CF}$, if the torsion part is nonzero.} of $\overline{\CF}$ have slope at least $l$, and by $\CHiggs_{r,d}(C)$ the stack of all Higgs sheaves of rank $r$ and degree $d$.  We define
\[
\CHiggs^{\mu\geq l,\nil}_{r,d}(C)\subset \CHiggs^{\mu\geq l}_{r,d}(C);\quad \quad
\CHiggs^{\nil}_{r,d}(C)\subset \CHiggs_{r,d}(C)
\]
as the (reductions of the) substacks corresponding to Higgs sheaves with nilpotent Higgs field.  Then 
\[
\HO^{\BM}(\CHiggs^{\mu\geq l}_{r,d}(C),\BQ)\quad\textrm{and}\quad \HO^{\BM}(\CHiggs^{\mu\geq l,\nil}_{r,d}(C),\BQ)
\]
are pure.  Moreover for $l'>l$ the natural morphisms
\[
\HO_c(\CHiggs^{\mu\geq l'}_{r,d}(C),\BQ)\rightarrow\HO_c(\CHiggs^{\mu\geq l}_{r,d}(C),\BQ);\quad\quad
\HO_c(\CHiggs^{\mu\geq l',\nil}_{r,d}(C),\BQ)\rightarrow \HO_c(\CHiggs^{\mu\geq l,\nil}_{r,d}(C),\BQ)
\]
are injective, so that we may write
\[
\HO_c(\CHiggs_{r,d}(C),\BQ)\quad\textrm{and}\quad \HO_c(\CHiggs^{\nil}_{r,d}(C),\BQ)
\]
as an infinite nested union of pure Hodge structures.
\end{proposition}
\begin{remark}
The fact that $\HO_c(\CHiggs^{\nil}_{r,d}(C),\BQ)$ can be written as an infinite nested union of pure Hodge structures (but in a different way) follows from a more down-to-earth argument, found in \cite{SchMo20}.  In brief, the stack admits a stratification by Jordan type of the Higgs sheaves (as defined in \cite{Sch16}), with each stratum an affine fibration over Cartesian products of stacks $\Coh_{r',d'}(C)$ for various $r',d'$.  We refer to \cite{SchMo20} for more details.
\end{remark}
\begin{proof}[Proof of Proposition \ref{Higgs_nest}]
We start with the non-nilpotent case.  This is almost a special case of Theorem \ref{1dim_pure}, except of course in this case $S$ is not projective.  However the proof of Theorem \ref{1dim_pure} only uses Serre duality (for the vanishing of $\mathrm{Ext}^2$ in Lemma \ref{HN_pure}) and purity of $\HO_c(\FM^H_{Q(t)}(S),\BQ)$ for $\deg(Q(t))\leq 1$, which holds here by Theorem \ref{Higgs_pure}.  The arguments for the nilpotent versions of the statements are identical, instead taking the purity statement of Proposition \ref{GNC_prop} as the input.
\end{proof}

\subsubsection{Comparing perverse filtrations}
\label{cpf}
We finish this section with the discussion of the perverse filtration on the intersection cohomology of $\CM^{\Dol}_{r,d}(C)$ that was promised in \S \ref{2per}.  Let $\CN\subset \CM^{\Dol}_{r,d}(C)$ be an irreducible component, containing a stable Higgs sheaf.  Then by setting $n=0$ in Theorem \ref{ICin} there is a morphism of complexes of mixed Hodge modules
\begin{equation}
\label{KMN}
\ICSn_{\CN}\rightarrow \BD(p_!\underline{\BQ}_{\FM^{\Dol}_{r,d}(C)}\otimes\LLL^{(1-g)r^2})
\end{equation}
which is moreover the inclusion of a summand, so taking derived global sections, we obtain an inclusion
\[
\ICA(\CN)\hookrightarrow \HO^{\BM}(\FM^{\Dol}_{r,d}(C),\BQ\otimes \LLL^{(1-g)r^2}).
\]
As in the introduction we denote by $\Hit\colon \CM^{\Dol}_{r,d}(C)\rightarrow \Lambda_r$ the Hitchin morphism.  Applying $\Hit_*\!$ to \eqref{KMN} we obtain the morphism of complexes
\begin{equation}
\label{PKMN}
\Hit_*\!\ICSn_{\CN}\rightarrow \Hit_*\!\BD(p_!\underline{\BQ}_{\FM^{\Dol}_{r,d}(C)}\otimes\LLL^{(1-g)r^2}).
\end{equation}
Let $s\colon \Lambda_r\rightarrow \pt$ be the structure morphism.  Applying the natural transformation $s_*(\bm{\tau}^{\leq j}\rightarrow \id)$ to \eqref{PKMN} we obtain the commutative square
\[
\xymatrix{
\ICA(\CN)\ar[r]&\HO^{\BM}(\FM^{\Dol}_{r,d}(C),\BQ\otimes \LLL^{(1-g)r^2})\\
\FP_j^{\ICS}\ar@{^{(}->}[u]\ar[r]&\FP_j^{\stacky}\ar@{^{(}->}[u]
}
\]
where 
\[
\FP_j^{\ICS}\subset\ICA(\CN);\quad\quad
\FP_j^{\stacky}\subset \HO^{\BM}(\FM^{\Dol}_{r,d}(C),\BQ\otimes \LLL^{(1-g)r^2}) 
\]
are the $j$th pieces of the perverse filtrations from \S \ref{perv_intro}, but with normalised twists and shifts properly taken care of.  In sum we have the following
\begin{proposition}
Let $C$ be a smooth connected complex projective curve, and let $\CN_s$ for $s\in S$ be the irreducible components of $\CM^{\Dol}_{r,d}(C)$ containing a stable Higgs sheaf.  Then there is a canonical inclusion
\[
\bigoplus_{s\in S}\ICA(\CN_s)\hookrightarrow \HO^{\BM}(\FM^{\Dol}_{r,d}(C),\BQ\otimes \LLL^{(1-g)r^2})
\]
of pure Hodge structures, which moreover respects the natural perverse filtrations induced by the Hitchin maps on the source and the target.
\end{proposition}
\subsection{$k[\pi_1(\Sigma_g)]$-modules}
\label{Betti_sec}
To start with let $k$ be an algebraically closed field of characteristic zero.  As in the case of quiver algebras, if $A=k\langle x_1,\ldots,x_r\rangle/R$ is a finitely generated $k$-algebra, with relations $R$, we may describe the moduli stack of $d$-dimensional representations of $A$ as a finite type global quotient, as follows.  We consider the closed subscheme $Z\subset (\Mat_{d\times d}(k))^{\times r}$ cut out by the matrix-valued relations $R$.  This scheme is $\Gl_d(k)$-equivariant, where $\Gl_d(k)$ acts via the simultaneous conjugation action.  Then there is an isomorphism of stacks
\[
\FM_d(A)\cong Z/\Gl_d(k).
\]
The more intrinsic way to describe the left hand side is as follows: it is the stack sending a commutative $k$-algebra $B$ to the groupoid formed by removing all non-invertible morphisms from the full subcategory of all $B\otimes A$-modules that are flat over $B$ and of dimension $d$ at geometric points of $B$.
\smallbreak

Clearly $\FM_d(A)$ is a global quotient stack.  We denote by
\[
\JH_d\colon \FM_d(A)\rightarrow \CM_d(A)\coloneqq \Spec(\Gamma(\CO_Z)^{\Gl_d(k)})
\]
the affinization map.  The notation is explained by the following observations: for $K\supset k$ a field, the $K$-points of $\CM_d(A)$ correspond to semisimple $A\otimes K$-modules, and at the level of points, $\JH_d$ takes modules to their semisimplifications, i.e. the direct sum of the subquotients appearing in their Jordan-H\"older filtration.
\sssct
We consider the case $A=k[\pi_1(\Sigma_g)]$.  Via the standard group presentation
\[
\pi_1(\Sigma_g)=\left\langle a_1,\ldots,a_g,b_1,\ldots,b_g\;\lvert\; \prod_{i=1}^g(a_i,b_i)\right\rangle
\]
we obtain a finite presentation for $A$.  The Calabi--Yau structure on $\Ddg(A\Lmod)$ in some sense predates Calabi and Yau, coming as it does from Poincar\'e duality.  The fact that this category has a left 2CY structure in the sense that we need is implied by the following unpublished result of Kontsevich (see \cite[Thm.5.2.2 \& Prop.5.2.6]{Superpotentials} for a proof\footnote{Note that the proof works for general $k$.}):

\begin{proposition}
Let $\Sigma_g$ be a closed, orientable Riemann surface without boundary.  Then the algebra $k[\pi_1(\Sigma_g)]$ is homologically smooth and carries a left 2CY structure.
\end{proposition}

It follows from Proposition~\ref{BDthm} that for any collection $\CF_1,\ldots,\CF_r$ of complexes of $k[\pi_1(\Sigma_g)]$-modules with finite-dimensional total cohomology, the full dg subcategory of $\Ddg(\Mod^{k[\pi_1(\Sigma_g)]})$ containing $\CF_1,\ldots,\CF_r$ carries a right 2CY structure.  From Corollary \ref{form_cor} we deduce the following proposition.
\begin{proposition}
Let $\CF_1,\ldots,\CF_r\in\Ob(\Db(k[\pi_1(\Sigma_g)]))$ be a $\Sigma$-collection of complexes of $k[\pi_1(\Sigma_g)]$-modules.  Then the full subcategory of $\Ddg(\fdmod^{k[\pi_1(\Sigma_g)]})$ containing $\CF_1,\ldots,\CF_r$ is formal.
\end{proposition}
\sssct
For the rest of \S \ref{Betti_sec} we set $k=\BC$.  From Theorem \ref{loc_pur_thm} and Proposition \ref{int_filt} we deduce the following
\begin{theorem}
\label{pi1_thm}
The complex $\JH_!\ul{\BQ}_{\FM^{\Betti}_{g,r}}$ is pure, and the Borel--Moore homology
\[
\HO^{\BM}(\FM^{\Betti}_{g,r},\BQ\otimes\LLL^{(1-g)r^2})\coloneqq \HO(\FM^{\Betti}_{g,r},\BD(\BQ\otimes\LLL^{(1-g)r^2}))
\]
carries an ascending perverse filtration by mixed Hodge structures, with respect to $\JH$, starting in perverse degree $0$.
\end{theorem}
Applying Theorem \ref{ICin} again we obtain
\begin{proposition}
Let $\CN$ be a connected component of $\CM^{\Betti}_{g,r}$ containing a simple representation of $\pi_1(\Sigma_g)$.  Then there is a canonical inclusion
\[
\ICA(\CN)\hookrightarrow \HO^{\BM}(\FM^{\Betti}_{g,r},\BQ\otimes \LLL^{(1-g)r^2})
\]
of mixed Hodge structures.
\end{proposition}
\sssct
Note that in contrast with some of our other examples, the Borel--Moore homology of $\FM^{\Betti}_{g,r}$ is never pure, unless $rg=0$.  This follows, for example, from the fact that the cohomology sheaves of the direct image $D_!\ul{\BQ}_{\FM^{\Betti}_{g,r}}$ are locally constant, where we define
\[
D:\FM^{\Betti}_{g,r}\rightarrow \BC^* ;\quad\quad
(A_1,\ldots,A_g,B_1,\ldots,B_g)\mapsto \det(A_1).
\]

Nonetheless, Theorem \ref{pi1_thm} enables us to write
\begin{equation}
\JH_!\ul{\BQ}_{\FM_{g,r}^{\Betti}}\cong\bigoplus_{n\in\BZ}\Ho^n\!\left(\JH_!\ul{\BQ}_{\FM_{g,r}^{\Betti}}\right)[-n]
\end{equation}
and 
\begin{equation}
\label{dbet}
\Ho^n\!\left(\JH_!\ul{\BQ}_{\FM_{g,r}^{\Betti}}\right)\cong \bigoplus_{j\in J_{g,r,n}}\IC_{Z_j}(\mathcal{L}_j)
\end{equation}
for (canonically defined) simple variations of Hodge structure $\mathcal{L}_j$ on locally closed subvarieties $Z_j\subset \CM_{g,r}^{\Betti}$. 
\subsubsection{Strict supports and nonabelian Hodge theory}
It is an interesting question to work out what these simple objects $\IC_{Z_j}(\mathcal{L}_j)$ in \eqref{dbet} are, or even just to calculate the supports $Z_j$.  Using results of Simpson \cite{Si92} it should be possible to show that the supports $Z_j$ are so-called (B,B,B)-branes: that is, their images under the nonabelian Hodge diffeomorphism 
\[
\Phi\colon \FM_{g,r}^{\Betti}\rightarrow \CM_{r,0}^{\Dol}(C)
\]
are the underlying topological spaces of the analytification of complex subvarieties of the Higgs moduli space.  We conjecture that something stronger is true:
\begin{conjecture}
\label{PWorder}
For every $n$ there is a natural isomorphism
\begin{equation}
\label{PerFil}
\Psi\colon\Phi_*\!{}^{\mathfrak{p}}\!\Ho^n\!\left(\JH_!\BQ_{\FM_{g,r}^{\Betti}}\right)\rightarrow{}^{\mathfrak{p}}\!\Ho^n\!\left( p_!\BQ_{\FM_{r,0}^{\Dol}(C)}\right)
\end{equation}
of perverse sheaves.
\end{conjecture}
Note that here we are obliged to work within the category of perverse sheaves, not mixed Hodge modules, since $\Phi$ is not an analytic morphism, and so does not induce a direct image functor at the level of mixed Hodge modules.  Note also that there is no known isomorphism of topological stacks between $\FM_{g,r}^{\Betti}$ and $\FM_{r,0}^{\Dol}(C)$, so this conjecture is saying something beyond the range of classical nonabelian Hodge theory.  In particular, while we do not conjecture that there is a natural isomorphism (of underlying graded vector spaces)
\[
\HO^{\BM}(\FM_{g,r}^{\Betti},\BQ)\cong \HO^{\BM}(\FM_{r,0}^{\Dol}(C),\BQ),
\]
the above conjecture implies that there is a natural isomorphism after passing to the associated graded object of the perverse filtration with respect to the morphisms to the respective coarse moduli spaces.  See \cite{PS19} for related work, and \cite[Thm.1.7]{DHSM22} for a recent proof of this conjecture that was completed as this paper was being revised.  See also \cite{Henn23} for a proof that the isomorphism in \cite{DHSM22} respects the Hall algebras on both sides, as well as the analogous result for the de Rham moduli stack. %\smallbreak
%By two applications of Theorem \ref{ICin}, Conjecture \ref{PWorder} is true when restricted to summands on the left and right hand side of \eqref{PerFil} that have full support on components of the coarse moduli spaces that intersect the stable locus of $\CM^{\Betti}_{g,r}$ and $\CM_{r,0}^{\Dol}(C)$ respectively.

\subsection{Cuspidal cohomology}
\label{cu_sec}
Theorems \ref{loc_pur_thm} and \ref{ICin} lead naturally to the definition of \textit{cuspidal cohomology}.  We start by explaining this in the case of semistable coherent sheaves on surfaces.
\subsubsection{Cohomological Hall algebra for surfaces}
Let $S$ be a smooth quasiprojective complex surface satisfying $\CO_S\cong \omega_S$.  Fix a polarisation $H\in\NS(\overline{S})_{\BR}$ and a polynomial $p(t)\in \BQ[t]$.  We define $\FM_{\sim p(t)}^{H\sstab}(S)$ to be the stack of compactly supported $H$-Gieseker semistable coherent sheaves on $S$ with reduced Hilbert polynomial $p(t)$.  So if we write $a_dt^d+\ldots a_0 \sim b_d t^d+\ldots+b_0$ to mean that $t^d+a_{d-1}/a_d t^{d-1}+\ldots+a_0/a_d=t^d+b_{d-1}/b_d t^{d-1}+\ldots+b_0/b_d$ then
\[
\FM_{\sim p(t)}^{H\sstab}(S)=\coprod_{P(t)\sim p(t)}\FM_{P(t)}^{H\sstab}(S).
\]
We define $\CM_{\sim p(t)}^{H\sstab}(S)$ similarly.  We denote by $\Ex_{\sim p(t)}^{H\sstab}(S)$ the stack of short exact sequences of $H$-Gieseker semistable sheaves with reduced Hilbert polynomials $p(t)$.  We denote by $\pi_i$ for $i=1,2,3$ the morphism sending a short exact sequence to its $i$th entry, and we will consider the commutative diagram
\[
\xymatrix{
\ar[d]^{\pi^H \times \pi^H}\FM_{\sim p(t)}^{H\sstab}(S)^{\times 2}&&\ar[ll]_-{\pi_1\times \pi_3} \Ex_{\sim p(t)}^{H\sstab}(S)\ar[r]^{\pi_2}&\FM_{\sim p(t)}^{H\sstab}(S)\ar[d]^{\pi^H}\\
\CM_{\sim p(t)}^{H\sstab}(S)^{\times 2}\ar[rrr]^{\oplus}&&&\CM_{\sim p(t)}^{H\sstab}(S).
}
\]
One may show (as in \cite[Lem.2.1]{MR19}) that the morphism $\oplus$ taking a pair of polystable coherent sheaves to their direct sum is finite.  The category $\Perv(\CM_{\sim p(t)}^{H\sstab}(S)^{\times 2})$ thus carries a natural convolution tensor product structure, defined by
\[
\CF\otimes_{\oplus} \CG\coloneqq \oplus_*(\CF\boxtimes\CG).
\]
The constant sheaf $\BQ_{\CM_0^{H\sstab}(S)}$ with support the zero (coherent) sheaf on $S$ provides a monoidal unit.  We can thus talk about algebras in $\Perv(\CM_{\sim p(t)}^{H\sstab}(S)^{\times 2})$ or $\Dub^+(\Perv(\CM_{\sim p(t)}^{H\sstab}(S)^{\times 2}))$, which are objects $\CF$ equipped with morphisms $\CF\otimes_{\oplus}\CF\rightarrow \CF$ and $\BQ_{\CM_0^{H\sstab}(S)}\rightarrow \CF$ satisfying the usual axioms for a monoid. 
\smallbreak
The morphism $\pi_2$ is proper, and we define
\[
\alpha\colon \BQ_{\FM_{\sim p(t)}^{H\sstab}(S)}\rightarrow \pi_{2,!}\BQ_{\Ex_{\sim p(t)}^{H\sstab}(S)}
\]
via adjunction.  We would like to be able to define a morphism
\[
\beta\colon (\pi_1\times \pi_3)_!\BQ_{\Ex_{\sim p(t)}^{H\sstab}(S)}\rightarrow \BQ_{\FM_{\sim p(t)}^{H\sstab}(S)^{\times 2}}[?]
\]
by taking the Verdier dual of the natural morphism 
\[
\BQ_{\FM_{\sim p(t)}^{H\sstab}(S)^{\times 2}}\rightarrow (\pi_1\times \pi_3)_*\BQ_{\Ex_{\sim p(t)}^{H\sstab}(S)},
\]
but this is not so simple, as the stacks in question are generally not smooth.  Nonetheless using virtual smoothness of $(\pi_1\times \pi_3)$, the morphism $\beta$ is defined as in \cite{KV19,DHSM22}, with the result that the symbol $?$ is replaced by a function of the Hilbert polynomials of the coherent sheaves that we are considering extensions of.  Now taking $\HO_c(\beta\circ\alpha)^{\vee}$, we obtain an associative multiplication $\mathbf{m}$ on the graded algebra
\[
\Coha=\bigoplus_{P(t)\sim p(t)}\HO^{\BM}(\FM_{P(t)}^{H\sstab}(S),\BQ[-\chi_S(\cdot,\cdot)])
\]
for which we refer to \cite{KV19,DHSM22} for further details.  The grading here is by the semigroup $L$ of polynomials $P(t)$ satisfying $P(t)\sim p(t)$.  We can consider $L$ as a reduced zero-dimensional scheme, with one $\BC$-point for each element of $L$.
\sssct

Consider the complex $\RCoha=\BD \pi^H_! \BQ_{\FM_{\sim p(t)}^{H\sstab}(S)}[-\chi_S(\cdot,\cdot)]$.  By Proposition \ref{rel_pur_S}, this is a semisimple complex of perverse sheaves situated in non-negative cohomological degrees.  This complex acquires an algebra structure in the category $\Dulf(\Perv(\CM_{\sim p(t)}^{H\sstab}(S)))$, by composing $\BD\pi^H_*\alpha$ with $\oplus_*\BD(\pi^H\times \pi^H)_! \beta$.  Moreover by definition there is an isomorphism of $L$-graded algebras $\Coha\cong r_*\RCoha$, where $r\colon \CM_{\sim p(t)}^{H\sstab}(S)\rightarrow L$ is the morphism taking $\CM_{P(t)}^{H\sstab}(S)$ to the point representing $P(t)$.
\smallbreak
In particular, the CoHA multiplication $\mathbf{m}= r_*(\RCoha\otimes_{\oplus}\RCoha\rightarrow \RCoha)$ respects the perverse filtration on $\Coha\otimes\Coha$ and $\Coha$, since it is obtained by taking derived global sections of a morphism of perverse complexes on $\CM_{\sim p(t)}^{H\sstab}(S)$ (see \S \ref{FPR}).  By Proposition \ref{rel_pur_S}, this perverse filtration begins in degree zero, and we define the \textit{BPS algebra} to be the subalgebra
\[
\FU_{\BPS}\coloneqq \FL_{\leq 0}\Coha\coloneqq\HO\left(\CM^{H\sstab}_{\sim p(t)},{}^{\mathfrak{p}}\!\bm{\tau}^{\leq 0}\BD\pi_!^H\BQ_{\FM_{\sim p(t)}^{H\sstab}(S)}[-\chi_S(\cdot,\cdot)]\right)\subset \Coha
\]
given by the zeroth piece of the perverse filtration\footnote{This is an analogue of the ``less'' perverse filtration from \cite{preproj3}, which is denoted in [ibid] by $\FL_{\leq \bullet}$ (hence our notation here).  This is not to be confused with the perverse filtration on critical cohomology (that $\Coha$ also carries, using Kinjo's dimensional reduction theorem \cite{Kin21}), which is quite different, and is denoted by $\FP_{\leq \bullet}$ in the case of preprojective algebras considered in \cite{preproj3}.}.  
\subsubsection{The BPS algebra}
More generally, assume that the category $\mathscr{C}$ and the lattice $L$ are obtained by one of the following constructions:
\begin{enumerate}
\item
Fix $S$ a quasiprojective smooth complex surface with $\mathcal{O}_S\cong \omega_S$, fix a polarisation $H\in\NS(\overline{S})_{\BR}$, and a reduced Hilbert polynomial $p(t)$.  Then take the category of $H$-Gieseker semistable coherent sheaves on $S$ with Hilbert polynomial $P(t)$ such that the normalisation of $P(t)$ is $p(t)$ (as above).  Set $L$ to be the semigroup of such polynomials.
\item
Fix $C$ a smooth projective complex curve, fix a slope $\theta$ (possibly infinite), and consider the category of slope-semistable Higgs sheaves on $C$ with slope $\theta$.  Let $L\in\BZ_{\geq 0}\oplus\BZ$ be the semigroup of pairs $(r,d)$ such that $r/d=\theta$.
\item
Fix the ground field $k=\BC$.  Fix $Q$ a finite quiver, $\zeta\in \BQ^{Q_0}$ a King stability condition, and fix a slope $\theta$.  Then take the category of $\zeta$-semistable $\Pi_Q$-modules with slope $\theta$ (with respect to $\zeta$).  Let $L$ be the semigroup of dimension vectors with slope $L$.
\item
Fix a genus $g\geq 0$.  Then consider the category of finite-dimensional $\BC[\pi_1(\Sigma_g)]$-modules.  We denote by $L=\BZ_{\geq 0}$ the semigroup of possible dimensions of such a module.
\end{enumerate}
In each case we denote by $\JH\colon\FM_{\mathscr{C}}\rightarrow \CM_{\mathscr{C}}$ the morphism from the stack to the good moduli space.  In case (1), we have just seen that $\HO^{\BM}(\FM_{\mathscr{C}},\BQ\otimes\LLL^{\chi_{\mathscr{C}}(\cdot,\cdot)/2})$ carries a $L$-graded Hall algebra structure, which is moreover induced by an algebra structure on the MHM complex $\BD(\JH_!\underline{\BQ}_{\FM_{\mathscr{C}}} \otimes\LLL^{\chi_{\mathscr{C}}(\cdot,\cdot)/2})$.  Since case (2) is a special case of (1) via the BNR correspondence, the same remains true there; this Hall algebra structure was first introduced and studied by Schiffmann and Sala \cite{SaSc20}, following Minets \cite{Mi20} in the infinite slope case.  In case (3) we may instead use the Hall algebra construction of Schiffmann and Vasserot \cite{ScVa13}, and in this case it is well-known (e.g. by dimensional reduction and the work of Kontsevich and Soibelman \cite{KSCoha}) how to lift everything to the category of mixed Hodge modules.  Similarly, in case (4), there is an algebra structure on $\JH_!\underline{\BQ}_{\FM_{\mathscr{C}}}\otimes\LLL^{\chi_{\mathscr{C}}(\cdot,\cdot)}$, constructed via dimensional reduction, inducing the Hall algebra structure on the graded mixed Hodge structure 
\[
\bigoplus_{r\geq 0}\HO^{\BM}(\FM^{\Betti}_{g,r},\BQ\otimes\LLL^{(1-g)r^2}).
\]
See \cite{Dav16} for details, or \cite{DHSM22} for a unifying treatment.
\smallbreak

Summarising, in each of the cases (1)--(4), there is a $L$-graded Hall algebra structure on 
\[
\RCoha_{\mathscr{C}}\coloneqq \BD(p_!\BQ_{\FM_{\mathscr{C}}}[-\chi_{\mathscr{C}}(\cdot,\cdot)])
\]
inducing the Hall algebra structure on the Borel--Moore homology
\[
\Coha_{\mathscr{C}}\coloneqq \HO^{\BM}\!(\FM_{\mathscr{C}},\BQ[-\chi_{\mathscr{C}}(\cdot,\cdot)])
\]
by taking derived global sections.  In each case we may upgrade these statements to the level of mixed Hodge modules on the coarse moduli space: the algebra structure is induced from one on
\[
\underline{\RCoha}_{\mathscr{C}}\coloneqq \BD(p_!\underline{\BQ}_{\FM_{\mathscr{C}}}\otimes\LLL^{-\chi_{\mathscr{C}}(\cdot,\cdot)/2}).
\]
In each case, we define the \textit{BPS algebra}
\[
\FU_{\BPS}(\mathscr{C})\coloneqq \bigoplus_{\gamma\in L}\HO\!\left(\CM_{\mathscr{C},\gamma},{}^{\mathfrak{p}}\!\bm{\tau}^{\leq 0}\!\RCoha_{\mathscr{C},\gamma}\right)\subset \Coha_{\mathscr{C}},
\]
which inherits the algebra multiplication $\mathbf{m}$ of $\Coha_{\mathscr{C}}$, since $\mathbf{m}$ preserves the perverse filtration.
\subsubsection{Cuspidal cohomology}
Following the most well-understood example (3) (worked out in \cite{preproj3}), the expectation is that $\FU_{\BPS}(\mathscr{C})$ is the universal enveloping algebra of the BPS Lie algebra $\Fg$ associated to the 3CY category of pairs $(\rho,f)$, where $\rho\in\Ob(\mathscr{C})$ and $f$ is an endomorphism of $\rho$.  We conjecture that $\Fg$ is moreover a generalised Kac--Moody Lie algebra (this would in particular imply the conjectures of \cite{BoSc19} in the case of preprojective algebras).
\smallbreak
Fleshing out these expectations and conjectures would take us a long way from the main thread of this paper; see \cite{DHSM23} for a fuller exposition.  If instead we assume that they hold, it becomes crucial to identify the imaginary simple roots of $\Fg$.  In any case, we wish to identify generators of $\FU_{\BPS}(\mathscr{C})$.  Theorem \ref{ICin} leads the way to the following result in this direction.
\begin{theorem}
\label{cu_thm}
Let $\mathscr{C}$ be one of the categories occurring in examples (1)--(4) above.  For $\gamma\in L$ we define
\begin{equation}
\label{cu_eq}
\Fcu_{\gamma}(\mathscr{C})\coloneqq \bigoplus_{\FN\in\pi_0(\CM^{\stab}_{\mathscr{C},\gamma})}\ICA(\overline{\FN}).
\end{equation}
Then there is a canonical decomposition of vector spaces
\[
\FU_{\BPS,\gamma}(\mathscr{C})\cong \Fcu_{\gamma}(\mathscr{C})\oplus \Fl_{\gamma}
\]
such that for all pairs of nonzero $\gamma',\gamma''\in L$ with $\gamma'+\gamma''=\gamma$, the CoHA multiplication map
\begin{equation}
\label{torros}
\mathbf{m}\colon \FU_{\BPS,\gamma'}(\mathscr{C})\otimes \FU_{\BPS,\gamma''}(\mathscr{C})\rightarrow \FU_{\BPS,\gamma}(\mathscr{C}).
\end{equation}
factors through the inclusion of $\Fl_{\gamma}$.
\end{theorem}
Using the lift of the algebra object $\RCoha_{\mathscr{C}}$ to an algebra object in the category of complexes of MHMs, the above result holds at the level of mixed Hodge structures.
\begin{proof}
The morphism \eqref{torros} is obtained by applying the derived global sections functor to the morphism of MHMs
\begin{equation}
\label{La}
\Ho^0\!\left(\oplus_*(\underline{\RCoha}_{\mathscr{C},\gamma'}\boxtimes\underline{\RCoha}_{\mathscr{C},\gamma''})\right)\rightarrow \Ho^0\!\underline{\RCoha}_{\mathscr{C},\gamma}.
\end{equation}
Moreover, by Theorem \ref{loc_pur_thm}, this is a morphism between semisimple perverse sheaves, after applying $\rat$.  By Theorem \ref{ICin}, there is a canonical decomposition of the target of \eqref{La}
\begin{equation}
\label{Lb}
\Ho^0\!\underline{\RCoha}_{\mathscr{C},\gamma} \cong \CG\oplus \left( \bigoplus_{\FN\in \pi_0\left(\CM^{\stab}_{\mathscr{C},\gamma}\right)}\ICSn_{\overline{\FN}}\right)
\end{equation}
where $\CG$ is a semisimple mixed Hodge module with support entirely contained in the strictly semistable locus of $\CM_{\mathscr{C}}$.  Since the morphism $\oplus\colon \CM_{\mathscr{C}}^{\times 2}\rightarrow \CM_{\mathscr{C}}$ factors through this locus, we deduce that the morphism 
\[
\Ho^0\!\left(\oplus_*(\underline{\RCoha}_{\mathscr{C},\gamma'}\boxtimes\underline{\RCoha}_{\mathscr{C},\gamma''})\right)\rightarrow \left( \bigoplus_{\FN\in \pi_0\left(\CM^{\stab}_{\mathscr{C},\gamma}\right)}\ICSn_{\overline{\FN}}\right)
\]
obtained by composing \eqref{La} with the projection onto the second factor of \eqref{Lb} is zero.  The theorem then follows, after taking derived global sections.
\end{proof}
\begin{definition}
For $\mathscr{C}$ any one of the four types of categories considered in Theorem \ref{cu_thm}, the $L$-graded object
\[
\mathfrak{cu}(\mathscr{C})\coloneqq \bigoplus_{\gamma\in L}\Fcu_{\gamma}(\mathscr{C})
\]
is the \textit{cuspidal cohomology} of $\mathscr{C}$, where the summands in the right hand side are defined in \eqref{cu_eq}.
\end{definition}
Theorem \ref{cu_thm} states that $\Fcu(\mathscr{C})$ can be extended to a minimal set of generators of $\FU_{\BPS}(\mathscr{C})$.  As in \cite{preproj3}, we conjecture that there is a minimal set of generators $\Fcu(\mathscr{C})\oplus \Fiso(\mathscr{C})$, where $\Fiso(\mathscr{C})$ is $L$ graded, and $\Fiso_{\gamma}(\mathscr{C})=0$ if $\chi_{\mathscr{C}}(\gamma,\gamma)\neq  0$.  I.e. in the language of generalised Kac--Moody algebras, we conjecture that $\Fcu(\mathscr{C})$ is the direct sum of the hyperbolic imaginary simple roots and the real simple roots.  See \cite[Sec.7.2,Sec.7.3]{preproj3} for an idea of how to define $\Fiso(\mathscr{C})$, and the rest of \cite[Sec.7]{preproj3} for evidence towards this conjecture, and see \cite{DHSM22} for a recent proof of this conjecture in the absence of real or isotropic simple roots, and \cite{DHSM23} for an even more recent proof of the conjecture in full generality.

\subsection{Multiplicative preprojective algebras}
\label{sec:mult_pp}
Let $k$ be a field satisfying $\mathrm{char}(k)=0$.  Let $Q$ be a quiver and let $q\in(k^*)^{Q_0}$.
We briefly recall the construction of the multiplicative preprojective algebra, denoted by $\Lambda^q(Q)$, which was originally constructed by Crawley-Boevey and Shaw in~\cite{CBS06}.  Let $\ol{Q}$ be the doubled quiver of $Q$ as in \S\ref{preproj_sec}.
We extend the operation $a \mapsto a^*$ to an involution on $\ol{Q}_1$ by setting $(a^*)^* = a$, and for $a\in \overline{Q}_1$ we define $\epsilon(a) = 1$ if $a \in Q_1$ and $\epsilon(a) = -1$ if $a^* \in Q_1$.
\smallbreak
Fix a total ordering $<$ on the set of arrows in $\ol{Q}$.  As in \cite{BCS21} we denote by $k\overline{Q}_{\loc}$ the algebra obtained by formally inverting each of the elements $1+aa^*$ for $a\in\overline{Q}_1$.  Then 
\[
\Lambda^q(Q)\coloneqq k\overline{Q}_{\loc}/\langle \prod_{a \in \ol{Q}_1} (1+aa^{*})^{\epsilon(a)} -\sum_{i \in Q_0} q_i e_i \rangle,
\]
where the product on the right hand side is ordered via the ordering on the arrows of $\overline{Q}_1$.  Up to isomorphism, the resulting algebra does not depend on this ordering \cite[Thm.~1.4]{CBS06}.
\smallbreak
In~\cite{SchKa19} it is proved that the multiplicative preprojective algebra is a left 2CY algebra, as long as the underlying graph of $Q$ contains a cycle.  In particular, the category $\Perf(\Lambda^q(Q))$ is smooth under this condition.
\sssct
As in the case of ordinary preprojective algebras, from the point of view of derived moduli stacks (pursued in \cite{BCS21}) it is more natural to consider the derived multiplicative preprojective algebra, defined in the natural way, which we briefly recall from \cite{SchKa19,BCS21}.  
\smallbreak
As in \S \ref{dppa} we start with the quiver $Q'$ obtained by adding a degree $-1$ loop $u_i$ to the doubled quiver $\overline{Q}$ at every vertex $i\in Q_0$.  We place the original arrows $a\in Q_1$ in degree zero in $Q'$.  For $a$ an arrow in this graded quiver, we denote by $\lvert a\lvert$ the degree of $a$.  We define $kQ'_{\loc}$ as before by formally inverting the elements $(1+aa^*)\in kQ'$ for each $a\in\overline{Q}_1$.  Then we define
\[
\Upsilon^q(Q)\coloneqq (kQ'_{\loc},d)
\]
where 
\[
du_i=\prod_{a \in \overline{Q}_1\cap s^{-1}(t)} (1+a^{*}a)^{\epsilon(a)} - q_i e_i
\]
and the ordering of the product is induced by the fixed ordering on $\overline{Q}_1$.  As with the construction of the derived preprojective algebra, the algebra $\Upsilon^q(Q)$ is concentrated in nonpositive degrees, and satisfies
\begin{equation}
\label{mult_down}
\HO^0(\Upsilon^q(Q))\cong\Lambda^q(Q).
\end{equation}
It follows that the natural heart of the category $\Dub(\Mod^{\Upsilon^q(Q)})$ may be identified with the Abelian category of $\Lambda^q(Q)$-modules, which we denote $\mathscr{A}$.  We recall the following standard lemma.
\begin{lemma}
Let $M,N\in\Ob(\mathscr{A})$.  Then for $l=0,1$ there are natural isomorphisms
\[
\Ext^l_{\mathscr{A}}(M,N)\cong \Ext^l_{\Mod^{\Upsilon^q(Q)}}(M,N).
\]
\end{lemma}
\begin{proof}
The claim is obvious for $l=0$, and so we concentrate on the $l=1$ case.  Set $\Lambda=\Lambda^q(Q)$, and for $i\in Q_0$ set $\Lambda_i\coloneqq \Lambda\cdot e_i$ and ${}_i\Lambda\coloneqq e_i\cdot \Lambda$.  By \cite[Lem.3.1]{CBS06} there is an exact sequence of $\Lambda$-bimodules
\[
\underbrace{\bigoplus_{i\in Q_0} \Lambda_i\otimes_k{}_i\Lambda\rightarrow \bigoplus_{a\in\overline{Q}_1}\lambda_{t(a)}\otimes_k {}_{s(a)}\Lambda\rightarrow \bigoplus_{i\in Q_0} \Lambda_i \otimes_k {}_i \Lambda}_{P^{\bullet}\coloneqq}\rightarrow \lambda
\]
which we may use to calculate $\Ext^{1}(M,N)$ for $l=1$, by computing $\Hom(P^{\bullet}\otimes_{\Lambda} M,N)$.  Then we obtain the complex
\[
\bigoplus_{i\in Q_0}\Hom_k(e_i\cdot M,e_i\cdot N)\xrightarrow{\alpha} \bigoplus_{a\in \overline{Q}_1}\Hom_k(e_{t(a)}\cdot M,e_{s(a)}\cdot N)\xrightarrow{\beta} \bigoplus_{i\in Q_0}\Hom_k(e_i\cdot M,e_i\cdot N)
\]
where $\alpha$ and $\beta$ are determined by the morphisms in \cite{CBS06}, and the cohomology in the middle position calculates $\Ext^1_{\mathscr{A}}(M,N)$.
\smallbreak
Set $\Upsilon=\Upsilon^q(Q)$, and define $\Upsilon_i$ and ${}_i\Upsilon$ as above.  The morphism of $\Upsilon$-bimodules
\[
\bigoplus_{a\in Q'_1}(\Upsilon_{t(a)})\otimes_k ({}_{s(a)}\!\Upsilon)[\lvert -a\lvert]\xrightarrow{\psi} \bigoplus_{i\in Q_0}\Upsilon_i\otimes_k {}_i\Upsilon
\]
satisfies $\cone(\psi)\simeq\Upsilon$, so that we may use it to calculate $\Ext^l_{\Mod^{\Upsilon^q(Q)}}(M,N)$, by replacing $M$ with $\cone(\psi)\otimes_{\Upsilon}M$.  The result is the double complex
\[
\xymatrix{
\bigoplus_{a\in \overline{Q}_1}\Hom_k(e_{t(a)}\cdot M,e_{s(a)}\cdot N)\ar[r]^-{\beta'}& \bigoplus_{i\in Q_0}\Hom_k(e_i\cdot M,e_i\cdot N)\\
\bigoplus_{i\in Q_0}\Hom_k(e_{t(u_i)}\cdot M,e_{s(u_i)}\cdot N)\ar[u]^{\alpha'}\ar[r]&0\ar[u]
}
\]
and it is easy to check the equalities $\alpha=\alpha'$ and $\beta=\beta'$.
\end{proof}
We also will need the following theorem.
\begin{theorem}\cite[Thm.~4.11]{BCS21}
    For every $q \in (k^*)^{Q_0}$ there is a left 2CY structure on $\Upsilon^{q}(Q)$, and so the category $\Perf_{\dg}(\Upsilon^{q}(Q))$ is smooth and carries a left 2CY structure.
\end{theorem}
Since the diagonal $\Upsilon^q(Q)$-bimodule is perfect, it follows that if $M$ is a perfect $\Upsilon^q(Q)$-module and $N$ is a $\Upsilon^q(Q)$-module with finite-dimensional total cohomology, the total cohomology of $\RHom(M,N)$ is finite-dimensional.  It follows in addition that all $\Upsilon^q(Q)$-modules with finite-dimensional total cohomology are perfect.  In particular, by the results of Brav and Dyckerhoff (i.e. applying Proposition~\ref{BDthm}) we deduce
\begin{corollary}
    Let $Q$ be a finite quiver.
    The left 2CY structure on $\Perf_{\dg}(\Upsilon^{q}(Q)^{\opp})$ induces a right 2CY structure on any full subcategory of $\Ddg(\fdmod^{\Upsilon^{q}(Q)})$ containing finitely many objects.
\end{corollary}

We deduce the following proposition from Corollary \ref{form_cor}:
\begin{proposition}
\label{mult_pp_formality}
Let $\CF_1,\ldots,\CF_r$ be a $\Sigma$-collection of complexes of $\Upsilon^{q}(Q)$-modules with finite-dimensional total cohomology.  Then the full dg subcategory of $\Ddg(\fdmod^{\Upsilon^{q}(Q)})$ containing $\CF_1,\ldots,\CF_r$ is formal.
\end{proposition}

\sssct
For the rest of \S \ref{sec:mult_pp} we set $k=\BC$.  Fix a stability condition $\zeta\in\BQ^{Q_0}$, and denote by
\[
\JH^{\zeta}_{\dd}\colon\FM_{\dd}^{\zeta\sstab}(\Lambda^q(Q)) \to \CM^{\zeta\sstab}_{\dd}(\Lambda^q(Q)),
\]
the morphism to the GIT quotient.  From Theorem \ref{loc_pur_thm} and Proposition \ref{int_filt} we deduce the following
\begin{theorem}
The complex $\JH^{\zeta}_{\dd,!} \ul{\BQ}_{\FM_{\dd}^{\zeta\sstab}(\Lambda^{q}(Q))}$ is pure, and the Borel-Moore homology
\[
\HO^{\BM}(\FM_{\dd}^{\zeta\sstab}(\Lambda^{q}(Q)),\BQ\otimes\LLL^{\chi_Q(\mathbf{d},\mathbf{d})})
\]
carries an ascending perverse filtration by mixed Hodge structures, relative to $\JH^{\zeta}_{\dd}$, beginning in degree zero.
\end{theorem}
\sssct
Note that since the coarse moduli space $\CM^{\zeta\sstab}_{\dd}(\Lambda^q(Q))$ is not projective, we cannot deduce purity of $\HO^{\BM}(\FM_{\dd}^{\zeta\sstab}(\Lambda^{q}(Q),\BQ)$ from Theorem \ref{glob_pur_thm}.  Moreover, in contrast with the ordinary preprojective algebra, one can see in simple examples that the Borel--Moore homology is indeed impure.  For instance consider the Jordan quiver $Q$ with one vertex and one loop, and pick any stability condition $\zeta$.  Set $q=1$.  Then $\FM^{\zeta\sstab}_1(\Lambda^q(Q))\cong U/\BC^*$, where $U=\BA^2\setminus \BC^*$, and $\BC^*$ is embedded via $t\mapsto (t,-t^{-1})$.  From the long exact sequence in compactly supported cohomology
\[
\HO_c^i(U,\BQ)\rightarrow \HO_c^i(\BA^2,\BQ)\rightarrow \HO_c^i(\BC^*,\BQ)\rightarrow 
\]
one finds that $\HO^3_c(U,\BQ)$ is one-dimensional and has weight two, and impurity for $\HO_c\!\left(\FM^{\zeta\sstab}_1(\Lambda^q(Q)),\BQ\right)$ follows, along with impurity of its dual $\HO^{\BM}\!\left(\FM^{\zeta\sstab}_1(\Lambda^q(Q)),\BQ\right)$.

\subsection{Kuznetsov components}
\label{sec:Kuznetsov_components}
The formality result (Corollary \ref{form_cor}) and Theorem \ref{thm:purity_relative} are sufficiently general to apply to 2CY categories outside of the motivating examples introduced in \S \ref{glob_intro}.  In particular we can apply them to ``noncommutative 2CY varieties'', or Kuznetsov components, i.e. admissible linear subcategories $\mathscr{C}$ of the derived category $\Db_{\dg}(\Coh(X))$ of a smooth proper scheme $X$ which carry a left 2CY structure.
\smallbreak
We recall the definition of a 2CY category over a field $k$ from~\cite{Per20}.
\begin{definition}\cite[Def.6.1]{Per20}\label{def:2CY_category}
    A \emph{(geometric)\footnote{The definition in \cite{Per20} omits the word ``geometric''.  We add it here, since we have a number of different definitions of 2CY categories already in play in this paper.  See \cite[Rem.6.21]{Per20} for a remark that justifies this terminology.} 2CY category (over $k$)} is a $k$-linear dg category $\mathscr{C}$ such that:
    \begin{enumerate}
        \item There exists an admissible $k$-linear embedding $\mathscr{C} \into \Perf_{\dg}(X)$ into the category of perfect complexes on a smooth proper variety $X$.
        \item The shift functor $[2]$ is a Serre functor for $\mathscr{C}$.
        \item The Hochschild cohomology of $\mathscr{C}$ satisfies $\HH^0(\mathscr{C}) = k$.
    \end{enumerate}
\end{definition}

\begin{remark}
\label{SmP}
If $X$ is a smooth and proper variety then $\Perf_{\dg}(X)$ is smooth and proper.  If $\mathscr{C}$ is an admissible $k$-linear subcategory of $\Perf_{\dg}(X)$, then $\mathscr{C}$ is automatically smooth and proper.
\end{remark}

The first condition of Definition \ref{def:2CY_category} says that $\mathscr{C}$ is of geometric origin.  The second condition states that $\mathscr{C}$ behaves like the derived category of a smooth proper surface with trivial canonical bundle, i.e. like the dg category of coherent sheaves over a smooth projective surface with $\CO_S\cong\omega_S$.  The third condition is a noncommutative version of connectivity, since for a smooth proper variety $X$ there is a natural isomorphism $\HH^0(\Perf_{\dg}(X))\cong \HO^0(X,\mathcal{O}_X)$.
\smallbreak
To relate back to left/right 2CY structures, condition (2) states that there is a bimodule isomorphism $\mathscr{C}[2]\rightarrow \mathscr{C}^{\vee}$, i.e. a nondegenerate cocycle $\CHC_{\bullet}(\mathscr{C})\rightarrow k[-2]$.  Arguing as in Lemma \ref{lem:left_right_CY} (but ignoring cyclicity), this provides a nondegenerate cycle 
\begin{equation}
    \label{etaK_def}
\eta\colon k[2]\rightarrow \CHC_{\bullet}(\mathscr{C})
\end{equation}
in light of Remark \ref{SmP}.  Condition (3) then guarantees that this nondegenerate cycle is cohomologically unique, up to scalar multiplication (see below).

\sssct
Let $\mathscr{C}$ be a geometric 2CY category.  The Mukai pairing \cite[Lem.~3.6]{Per20} provides a nondegenerate pairing
\[
\HH_{i}(\mathscr{C}) \otimes \HH_{-i}(\mathscr{C}) \to k.
\]
By property (2) and \cite[Lem.~3.7]{Per20}, there are canonical duality isomorphisms $\HH^i(\mathscr{C}) \cong \HH_{2-i}(\mathscr{C})$,
and so by property (3) we find $\HH_2(\mathscr{C}) \cong k$. 

\begin{proposition}\label{prop:2CYcat_is_2CY}
    Let $\mathscr{C}$ be a geometric 2CY category and let $\mathscr{C} \into \Perf_{\dg}(X)$ be an admissible $k$-linear embedding for a smooth proper variety $X$.
    Assume that $\HH_{i}(X) = 0$ for $i > 2$.
    There is an induced cycle $\eta \colon k[2] \to \HC^{-}(\mathscr{C})$, unique up to scaling, which defines a left 2CY structure on $\mathscr{C}$.
\end{proposition}
\begin{proof}
Since Hochschild homology is additive on semi-orthogonal decompositions~\cite[Thm.~7.3]{Kuz09}, the natural morphism $\HH_{i}(\mathscr{C}) \rightarrow \HH_{i}(X)$ is injective, and so $\HH(\mathscr{C})=0$ for $i>2$.
Via the spectral sequence argument of~\cite[Prop.~5.7]{MR3896226} (see also \cite[Lem.~5.10(1)]{BD19}), we deduce that $\HC^{-}_{i}(\mathscr{C}) = 0$ for $i > 2$ and that the morphism $\CHC_{\bullet}(\mathscr{C})^{S^1} \to \CHC_{\bullet}(\mathscr{C})$ induces an isomorphism
\begin{equation}
    \HC^{-}_{2}(\mathscr{C}) \xrightarrow{\cong} \HH_{2}(\mathscr{C}).
\end{equation}
We deduce that the nondegenerate cycle $\eta$ from \eqref{etaK_def} has a cyclic lift.
\end{proof}

\sssct
The following is now a special case of Corollary \ref{form_cor}.
\begin{proposition}[Noncommutative Kaledin--Lehn conjecture]
\label{Kuz_formality}
Let $\mathscr{C}$ be a geometric 2CY category and let $\mathscr{C} \into \Perf_{\dg}(X)$ be an admissible $k$-linear embedding for a smooth proper variety $X$.
    Assume that $\HH_{i}(X) = 0$ for $i > 2$.  Let $\BS$ be a $\Sigma$-collection of objects in $\mathscr{C}$.  Then the full subcategory of $\mathscr{C}$ containing $\CF_1,\ldots,\CF_r$ is formal.  In particular, the dg Yoneda algebra of $\bigoplus_{i}\mathcal{F}_i^{\mathbf{d}_i}$ is formal for any $\mathbf{d}\in\BN^r$.
\end{proposition}

As in the rest of the paper, we can now apply Theorem \ref{loc_str_thm}.  I.e. if $\FM$ is (the classical underlying Artin stack of) an open substack of the moduli stack of objects in a geometric 2CY category $\mathscr{C}$, and $p\colon\FM\rightarrow \CM$ is a good moduli space, then \'etale locally the morphism $p$ is modelled by the semisimplification morphism from the moduli stack of representations of a preprojective algebra to the coarse moduli space, as in \eqref{eq:neighbourhood_diagram}.  
\sssct
Many interesting examples are known of moduli stacks of objects in geometric 2CY categories with good moduli spaces.  For example consider as in \cite{MR2605171} a Kuznetsov component $\Ku(X)$, which forms part of a semi-orthogonal decomposition
\[
\langle \Ku(X),\CO_X,\CO_X\otimes\iota^*\CO_{\BP^5}(1),\CO_X\otimes\iota^*\CO(2) \rangle=\Db_{\dg}(X)
\]
where $\iota\colon X\rightarrow \BP^5$ is an embedding of a smooth cubic fourfold.  The vanishing of $\HH_{\geq 3}(X)$ follows from inspection of the Hodge diamond of $X$, and thus we deduce formality for $\Sigma$-collections in $\Ku(X)$ from Proposition \ref{Kuz_formality}.  Bridgeland stability conditions are constructed for $\Ku(X)$ in \cite{Betal17}.  Moreover by \cite[Sec.21]{Betal19}, which in turn relies on fundamental results of \cite{AHLH18}, the moduli stack of semistable objects $\FM^{\zeta\sstab}(\Ku(X))$ is locally finite type and admits a good moduli space $p\colon \FM^{\zeta\sstab}(\Ku(X))\rightarrow \CM^{\zeta\sstab}(\Ku(X))$.  We write $\FM=\FM^{\zeta\sstab}(\Ku(X))$ and $\CM=\CM^{\zeta\sstab}(\Ku(X))$.
\smallbreak
By Lemma \ref{techy} and Proposition \ref{Sclosed} the morphism $p$ satisfies the conditions of Definition/Proposition \ref{dfp} and Theorem \ref{loc_pur_thm}, and so we deduce that for $i\in\BZ$ there is a canonically defined mixed Hodge module $\Ho^i\!p_!\underline{\BQ}_{\FM}$, pure of weight $i$, which vanishes at points of $\CM$ representing polystable objects $\CF$ satisfying $\chi_{\Ku(X)}(\CF,\CF)>i$.
\smallbreak
At this point we would like to be able to deduce purity of $p_!\underline{\BQ}_{\FM}$, purity of the mixed Hodge structure on Borel--Moore homology, and existence of a perverse filtration on $\HO^{\BM}(\FM,\BQ)$.  However we are impeded by the fact that we have only defined the direct image $p_!\underline{\BQ}_{\FM'}$ in the derived category of mixed Hodge modules on $\CM'$ for morphisms $p'\colon \FM'\rightarrow \CM'$, where $\FM'$ is a stack that at least admits an exhaustion by global quotient stacks.
\smallbreak
The same situation is repeated for Kuznetsov components $\Ku(X)$ appearing in semi-orthogonal decompositions (as introduced in \cite{MR3734109}) of Gushel--Mukai varieties \cite{MR682488,MR995400}, and the associated stacks $\FM$ of Bridgeland-semistable objects in $\Ku(X)$: formality of endomorphism algebras of semisimple objects is a special case of Proposition \ref{Kuz_formality}, and (weight $i$) purity of $\Ho^i\!p_!\underline{\BQ}_{\FM}$ follows as above, but we are prevented, for now, from using this result to say anything about the Hodge theory of $\HO^{\BM}(\FM,\BQ)$ by the fact that these stacks are not constructed as nested unions of global quotient stacks.

\appendix
\section{Coherent completions around semisimple objects}
\label{Appendix}
The goal of this section is to prove Proposition \ref{app_prop}, via some recollections of derived deformation theory.  This material is all well-known to the experts, although I couldn't find Proposition \ref{minversal} in exactly the form given here in the literature.
\smallbreak
In this section we fix a pre-triangulated dg category $\mathscr{C}$.  We work with derived stacks, as presented in \cite{HAGII}, since this is the natural setting for constructing moduli of objects in dg categories as in \cite{To07}.  We are only using a tiny slither of the theory: that part which covers formal neighbourhoods of points in the derived stack of objects, which really predates the general machinery of derived algebraic geometry.  For this reason we will refrain from giving a great number of definitions and details regarding derived algebraic geometry.  
\smallbreak
We denote by $\Calg_k$ the category of commutative $k$-linear dg algebras.  A derived prestack is a functor 
\[
\bm{\FM}\colon \Calg_k^{\leq 0}\rightarrow \Sset.
\]
where $\Calg_k^{\leq 0}$ is the category of commutative $k$-linear dg algebras with cohomology concentrated in nonpositive degrees, and $\Sset$ is the category of simplicial sets.  The prestack $\bm{\FM}$ is a stack if it satisfies descent with respect to hypercovers (see \cite{HAGII}).  
\smallbreak
Equipping $\Calg_k$ with the class of weak equivalences given by quasi-isomorphisms, we obtain a simplicial enrichment $\mathscr{D}$.  I.e. we define $\Map_{\Calg_k}$ to be the right derived functor of $\Hom_{\Calg_k}\colon \Calg_k\times (\Calg_k)^{\op}\rightarrow \Sset$ given by considering Hom sets as constant simplicial sets.  Somewhat more explicitly, for cofibrant $A$ we can define
\[
\Map_{\Calg_k}(A,B)_n=\Hom_{\Calg_k}(A,B\otimes \Omega^{\bullet}(\Delta^n))
\]
where $\Omega^{\bullet}(\Delta^n)=k[x_0,\ldots,x_n,dx_0,\ldots,dx_n] /(\sum_i x_i-1,\sum dx_i)$
\smallbreak
Given $A\in\Ob(\Calg_k)$ we define the derived stack
\[
\Spec(A)(B)=\Map_{\Calg_k}(A,B).
\]
We call a derived stack an affine derived stack if it is equivalent to $\Spec(A)$ for $A\in\Ob(\Calg_k)$, and a affine derived scheme if it is equivalent to $\Spec(A)$ for $A\in\Ob(\Calg_k^{\leq 0})$.  The theory of affine derived stacks is developed in \cite{To06}.  The paper \cite{BZN12} is a beautiful demonstration of affine derived stacks in action.
\smallbreak
The derived moduli stack $\bm{\FM}_{\mathscr{C}}$ is defined in \cite{To07} by
\[
\bm{\FM}_{\mathscr{C}}(A)=\Map_{\dgCat}(\mathscr{C},\Perf_{\dg}(A)).
\]
\begin{remark}
To obtain a well-behaved derived stack (i.e. a geometric stack in the terminology of \cite{HAGII}) we should demand that $\mathscr{C}$ is finite type.  Ensuring the good behaviour of $\bm{\FM}_{\mathscr{C}}$ will not be a concern in this paper, since we will restrict attention to a substack $\bm{\FM}$, or rather its classical truncation, which is defined as an Artin stack.  In any case we point out that the (derived) preprojective algebra is manifestly finite type, as it is given by a semifree presentation.  The dg category of coherent sheaves on a finite type scheme is also finite type by \cite{Ef20}.  The derived multiplicative algebra and the fundamental group algebra almost admit semifree presentations as dg algebras, the first by definition, the second by \cite{Superpotentials}, except in both cases we take a localisation.  That they are both finite type is then a consequence of \cite[Prop.2.8]{Ef20}.
\end{remark}

A formal prestack is a functor $\Cart^{\leq 0}_k\rightarrow \Sset$, where $\Cart^{\leq 0}_k$ is the full dg subcategory of $\Calg_k^{\leq 0}$ containing commutative dg algebras $A$ for which $\HO(A)$ is finite-dimensional and local.  Such algebras come equipped with a canonical augmentation $\iota_A\colon A\rightarrow k$.  A formal prestack is a stack if it satisfies hyperdescent with respect to hypercovers in $\Cart_k^{\leq 0}$.

\sssct

Let $A$ be an $A_{\infty}$-algebra with operations $\Ab_n\colon (A[1])^{\otimes n}\rightarrow A[2]$, and assume that $\HO(A)$ is finite-dimensional.  We define the Maurer--Cartan stack
\[
\dMC(A)=\Spf\left((\widehat{\Sym}(A[1]^{\vee}),d_A)\right)\coloneqq \colim_{n\mapsto \infty}\left(\Spec(\Sym(A[1]^{\vee}/\Sym^{\geq n}(A[1]^{\vee}),d_A)\right)
\]
where $d_A$ is the unique continuous differential on $\widehat{\Sym}(A[1]^{\vee})$ satisfying the Leibniz rule, and such that $d_A$ restricted to $A[1]^{\vee}$ is the sum $\sum_{n\geq 1}\Ab_n^{\vee}$.  We refer the reader to \cite{Ge17} for a fuller account of this and related derived stacks.  We consider $\dMC(A)$ as a formal stack by restricting to $\Cart^{\leq 0}_k$.
\smallbreak
If $B\rightarrow A$ is a morphism of $A_{\infty}$-algebras, defined by operators $f_n\colon A[1]^{\otimes n}\rightarrow B[1]$ we denote by 
\[
\tilde{f}\colon (\widehat{\Sym}(B[1]^{\vee}),d_B)\rightarrow (\widehat{\Sym}(A[1]^{\vee}),d_A)
\]
the continuous morphism of algebras restricting to $\sum f_n^{\vee}$ on $B[1]^{\vee}$.  The algebra morphism $\tilde{f}$ commutes with the differential since the operators $f_n$ define an $A_{\infty}$-morphism.  If $f$ is a quasi-isomorphism, so is the induced morphism $\tilde{f}$.
\smallbreak
Given an object $\CF$ of a dg category $\mathscr{C}$, let $E$ be the dg endomorphism algebra of $\CF$.  We denote by 
\[
\dDef_{\mathscr{C}}(\CF)\colon \Cart^{\leq 0}_k\rightarrow \Sset
\]
the functor sending $B$ to the homotopy fibre $X$ in the diagram
\[
\xymatrix{
X\ar[d]\ar[r]&\Spec(k)\ar[d]^{\CF}\\
\bm{\FM}_{\mathscr{C}}(B)\ar[r]^{\iota^*}&\bm{\FM}_{\mathscr{C}}(k).
}
\]
Then the fundamental theorem, which in the form in which we are giving it here is a special case of the result due to Lurie \cite{LuDAGX} and Pridham \cite{Prid10}, is that there is an equivalence of derived stacks
\[
\Psi\colon \dMC(E)\simeq \dDef_{\mathscr{C}}(\CF).
\]
\begin{remark}
We can make this equivalence a little more concrete by providing an explicit family of objects of $\mathscr{C}$ over $\dMC(E)$.  By our finiteness assumption on $\HO(E)$, there is a sequence of natural morphisms
\begin{align*}
\Hom_{\Calg_k}((\widehat{\Sym}(E[1]^{\vee}),d),(C,d_C))\simeq&\Hom_{\Cocoal_k}((\mathfrak{m}_C^{\vee},d_C^{\vee}),(F(E[1]^{\vee\vee}),d'))\\
\simeq &\Hom_{\Cocoal_k}((\mathfrak{m}_C^{\vee},d_C^{\vee}),(F(E[1]),d''))\\
\rightarrow & E[1]\otimes \mathfrak{m}_C
\end{align*}
where $F(V)$ is the cofree non-counital cocommutative coalgbra generated by $V$.  The differential $d'$ is induced by the differential $d$.
If $C$ is an ordinary (ungraded) Artinian algebra, the final morphism is the inclusion of the locus satisfying the classical Maurer--Cartan equation.  We denote by $\phi$ the composition of these morphisms.  Write $d_{\CF}$ for the differential on the dg $\mathscr{C}$-module $\CF$.  For $C$ an Artinian dg algebra equipped with a morphism 
\[
g\colon (\widehat{\Sym}(E[1]^{\vee}),d)\rightarrow C
\]
we consider the $\mathscr{C}\otimes C$-module $(\CF\otimes C,d_{\CF}\otimes \id_{C}+\phi(g))$.
\end{remark}
\sssct
\label{BSsetup}
Now let $\BS=\{\CF_1,\ldots,\CF_r\}$ be a simple collection in $\mathscr{C}$.  This means that there are no negative Exts between the $\CF_i$, they are pairwise non-isomorphic, and the only $\Ext^0$s are given by scalar multiples of identity morphisms.  We fix a unital minimal model for the subcategory of $\mathscr{C}$ containing $\BS$, inducing a $\Gl_{\mathbf{d}}$-equivariant minimal model $E$ for the endomorphism algebra of $\CF=\bigoplus_{i\leq r}\CF_i^{\oplus\bf{d}_i}$.  Let $\BS'=\{\CF'_1,\ldots,\CF'_r\}$ be a simple collection in a dg category $\mathscr{D}$, and assume that there is an $A_{\infty}$-isomorphism $G$ from the full subcategory of $\mathscr{C}$ containing $\BS$ to the full subcategory of $\mathscr{D}$ containing $\BS'$, sending $\CF_i$ to $\CF'_i$.  Let $E'$ be the minimal model for $\End_{\mathscr{D}}(\bigoplus_{i\leq r}\CF_i'^{\oplus\mathbf{d}_i})$ induced by a minimal model for $\BS'$.  Then the induced morphism $E\rightarrow E'$ of $A_{\infty}$-algebras is $\Gl_{\mathbf{d}}$-equivariant.
\smallbreak
The natural morphism $E^{\geq 1}\rightarrow E$ of (non-unital!) $A_{\infty}$-algebras induces a morphism of derived stacks
\[
\pi\colon \dMC(E^{\geq 1})\rightarrow \dMC(E)
\]
which is moreover formally smooth, since it is a $\widehat{(\mathfrak{gl}_{\mathbf{d}})}_0$ torsor.  We refer to \cite{ELO09} for details.  By construction, $\dMC(E^{\geq 1})$ is a formal derived scheme.  The classical scheme $\MC(E^{\geq 1})\coloneqq t_0(\dMC(E^{\geq 1}))$ is known as the Maurer--Cartan locus, while the classical stack $\MC(E)\coloneqq t_0(\dMC(E))$ is known as the Deligne groupoid.  
\begin{remark}
Although we do not use the fact, it may help to note that the derived scheme $\dMC(E^{\geq 1})$ represents the functor $\overline{\dDef}_{\mathscr{C}}(\CF)$, which sends $B\in\Cart^{\leq 0}_k$ to the homotopy pullback $X$ in the diagram
\[
\xymatrix{
X\ar[d]\ar[r]&\Spec(B)\ar[d]^{\tau^*\CF}\\
\bm{\FM}_{\mathscr{C}}(B)\ar[r]^{\tau^*\iota^*}&\bm{\FM}_{\mathscr{C}}(B)
}
\]
where $\tau\colon\Spec(B)\rightarrow \Spec(k)$ is the structure morphism.
\end{remark}
\sssct
We recall from \cite{MR4088350} that if $x\in\FX$ is a point of a classical stack, a \textit{formal miniversal deformation space} of $x$ is a formal affine scheme $\widehat{\Def}(x)$, with one closed point, along with a formally smooth morphism $\widehat{\Def}(x)\rightarrow X$ inducing an isomorphism of tangent spaces at $x$.  The morphism
\[
e=t_0(\Psi\circ \pi)\colon \MC(E^{\geq 1})\rightarrow t_0(\bm{\FM}_{\mathscr{C}})
\]
is formally smooth since $\Psi$ and $\pi$ are.  The morphism $e$ induces the identity on tangent spaces, after they are both identified with $\Ext^1(\CF,\CF)$.  So in particular $e\colon \MC(E^{\geq 1})\rightarrow t_0(\bm{\FM}_{\mathscr{C}})$ is a formal miniversal deformation space of $\CF$.  Moreover by construction, the given family of $\mathscr{C}$-modules on $\MC(E^{\geq 1})$ is $\Gl_{\mathbf{d}}$-equivariant, so that we obtain a morphism
\[
\Gamma\colon\MC(E^{\geq 1})/\Gl_{\mathbf{d}}\rightarrow t_0(\bm{\FM}_{\mathscr{C}}).
\]
Note that $\Gamma$ induces an isomorphism of stabiliser groups at $0\in\MC(E^{\geq 1})$.  Since the construction of $\MC(E^{\geq 1})/\Gl_{\mathbf{d}}$ only depends, up to equivalence, on the $A_{\infty}$-isomorphism class of the minimal model for the full subcategory of $\mathscr{C}$ containing $\BS$, we deduce that the following:
\begin{proposition}
\label{minversal}
Let $\BS$, $\BS'$, $\CF$, $\CF'$, $E$ and $E'$ be as introduced in \S \ref{BSsetup}.  Then there is a $\Gl_{\dd}$-equivariant isomorphism of miniversal deformation spaces $\MC(E^{\geq 1})\cong \MC(E'^{\geq 1})$, where $\MC(E^{\geq 1})$ is a miniversal deformation space to the point in $t_0(\bm{\FM}_{\mathscr{C}})$ representing $\CF$ and $\MC(E'^{\geq 1})$ is a miniversal deformation space to the point in $t_0(\bm{\FM}_{\mathscr{D}})$ representing $\CF'$.
\end{proposition}

Proposition \ref{app_prop} then follows from the equivalence of (1) and (2) in \cite[Thm.4.19]{MR4088350}.

\bibliographystyle{alpha}
\bibliography{biblio}

\end{document}

%% file: macros.tex
\renewcommand{\theta}{\uptheta}
\renewcommand{\alpha}{\upalpha}
\renewcommand{\beta}{\upbeta}
\renewcommand{\gamma}{\upgamma}
\renewcommand{\delta}{\updelta}
\renewcommand{\zeta}{\upzeta}
\renewcommand{\pi}{\uppi\hspace{0.05em}}
\renewcommand{\rho}{\uprho}
\renewcommand{\xi}{\upxi}
\renewcommand{\chi}{\upchi}
\renewcommand{\sigma}{\upsigma}
\renewcommand{\Lambda}{\Uplambda}
\renewcommand{\Gamma}{\Upgamma}
\renewcommand{\phi}{\upphi}
\renewcommand{\psi}{\uppsi}
\renewcommand{\nu}{\upnu}
\renewcommand{\tau}{\uptau}
\renewcommand{\mu}{\upmu}
\renewcommand{\eta}{\upeta}
\newtheorem{theorem}{Theorem}[section]
\newtheorem{thmx}{Theorem}

\newtheorem{proposition}[theorem]{Proposition}
\newtheorem{defproposition}[theorem]{Definition/Proposition}
\newtheorem{lemma}[theorem]{Lemma}
\newtheorem{conjecture}[theorem]{Conjecture}

\newtheorem{corollary}[theorem]{Corollary}

\theoremstyle{remark}
\newtheorem{remark}[theorem]{Remark}

\theoremstyle{definition}
\newtheorem{definition}[theorem]{Definition}

\newtheorem{example}[theorem]{Example}

\newcommand{\BA}{{\mathbb{A}}}

\newcommand{\BC}{{\mathbb{C}}}

\DeclareMathOperator{\Def}{Def}
\DeclareMathOperator{\dDef}{\mathbf{Def}}

\newcommand{\BD}{{\mathbb{D}}}

\newcommand{\BG}{{\mathbb{G}}}

\DeclareMathOperator{\HO}{\mathbf{H}}
\DeclareMathOperator{\ho}{\mathbf{h}}
\DeclareMathOperator{\BM}{BM}
\DeclareMathOperator{\Hit}{\mathtt{h}}
\newcommand{\HOBM}{\HO^{\BM}}
\DeclareMathOperator{\Isom}{\mathcal{I}som}

\newcommand{\BN}{{\mathbb{N}}}

\newcommand{\BP}{{\mathbb{P}}}
\newcommand{\BQ}{{\mathbb{Q}}}
\newcommand{\BR}{{\mathbb{R}}}
\newcommand{\BS}{{\mathbb{S}}}

\newcommand{\BZ}{{\mathbb{Z}}}

\newcommand{\CC}{{\mathcal C}}

\newcommand{\CE}{{\mathcal E}}
\newcommand{\CF}{{\mathcal F}}
\newcommand{\CG}{{\mathcal G}}
\newcommand{\CH}{{\mathcal H}}
\newcommand{\CI}{{\mathcal I}}

\newcommand{\CK}{{\mathcal K}}
\newcommand{\CL}{{\mathcal L}}
\newcommand{\CM}{{\mathcal M}}
\newcommand{\CN}{{\mathcal N}}
\newcommand{\CO}{{\mathcal O}}
\newcommand{\CP}{{\mathcal P}}
\newcommand{\CQ}{{\mathcal Q}}

\newcommand{\CS}{{\mathcal S}}

\newcommand{\CX}{{\mathcal X}}
\newcommand{\CY}{{\mathcal Y}}

\newcommand{\Fcu}{\mathfrak{scu}}

\newcommand{\Fg}{{\mathfrak{g}}}

\newcommand{\Fiso}{{\mathfrak{iso}}}

\newcommand{\Fl}{{\mathfrak{l}}}
\newcommand{\Fm}{{\mathfrak{m}}}

\newcommand{\Fp}{{\mathfrak{p}}}

\newcommand{\FL}{{\mathfrak{L}}}
\newcommand{\FM}{{\mathfrak{M}}}
\newcommand{\FN}{{\mathfrak{N}}}

\newcommand{\FP}{{\mathfrak{P}}}

\newcommand{\FS}{{\mathfrak{S}}}
\newcommand{\FT}{{\mathfrak{T}}}
\newcommand{\FU}{{\mathfrak{U}}}

\newcommand{\FX}{{\mathfrak{X}}}
\newcommand{\FY}{{\mathfrak{Y}}}

\let\ol=\overline
\let\ul=\underline
\DeclareMathOperator{\opp}{op}
\newcommand{\udim}{\underline{\dim}}

\DeclareMathOperator{\JH}{\mathtt{JH}}

\newcommand{\unit}{\mathds{1}}
\newcommand{\ICS}{\mathcal{IC}}
\newcommand{\ICSn}{\underline{\mathcal{IC}}}

\DeclareMathOperator{\pvrs}{{}^\mathfrak{p}\!}

\DeclareMathOperator{\sstab}{-sst}
\DeclareMathOperator{\st}{-st}
\DeclareMathOperator{\stab}{st}
\DeclareMathOperator{\coker}{coker}
\DeclareMathOperator{\ConB}{\mathbf{B}}

\DeclareMathOperator{\DL}{L}

\DeclareMathOperator{\Hom}{Hom}
\DeclareMathOperator{\chara}{char}

\DeclareMathOperator{\End}{End}
\DeclareMathOperator{\Ext}{Ext}
\DeclareMathOperator{\MHS}{MHS}

\DeclareMathOperator{\Gr}{Gr}
\DeclareMathOperator{\Vect}{\mathbf{Vect}}

\DeclareMathOperator{\SRHom}{\mathcal{RH}om}
\DeclareMathOperator{\Id}{Id}
\DeclareMathOperator{\rat}{rat}

\DeclareMathOperator{\cdotb}{\!\cdot\!}

\DeclareMathOperator{\dd}{\mathbf{d}}

\DeclareMathOperator{\Perf}{Perf}
\DeclareMathOperator{\Ex}{\mathfrak{E}x}

\DeclareMathOperator{\PGl}{PGl}

\DeclareMathOperator{\Perv}{Perv}
\DeclareMathOperator{\twist}{tw}
\DeclareMathOperator{\MHM}{MHM}
\DeclareMathOperator{\IC}{\mathcal{IC}}
\DeclareMathOperator{\ICA}{IH^*}

\DeclareMathOperator{\Sym}{Sym}

\DeclareMathOperator{\nil}{nil}

\DeclareMathOperator{\op}{op}
\DeclareMathOperator{\Spec}{Spec}
\DeclareMathOperator{\Spf}{Spf}
\DeclareMathOperator{\Gl}{GL}
\DeclareMathOperator{\id}{id}
\DeclareMathOperator{\fdmod}{\mathbf{mod}}
\DeclareMathOperator{\FreeT}{\widehat{T}}
\DeclareMathOperator{\Jac}{Jac}

\DeclareMathOperator{\Mod}{\mathbf{Mod}}
\DeclareMathOperator{\Tr}{Tr}

\DeclareMathOperator{\pt}{pt}
\DeclareMathOperator{\Tot}{Tot}

\DeclareMathOperator{\colim}{colim}
\DeclareMathOperator{\KK}{K}

\DeclareMathOperator{\topo}{top}

\DeclareMathOperator{\dR}{dR}
\DeclareMathOperator{\cone}{\mathbf{cone}}

\DeclareMathOperator{\vir}{vir}
\DeclareMathOperator{\Ob}{Ob}
\DeclareMathOperator{\cyc}{cyc}

\DeclareMathOperator{\BPS}{BPS}
\DeclareMathOperator{\LLL}{\mathscr{T}}
\DeclareMathOperator{\Ho}{\mathcal{H}}
\DeclareMathOperator{\HH}{HH}

\DeclareMathOperator{\B}{B\!}
\DeclareMathOperator{\MC}{MC}
\DeclareMathOperator{\dMC}{\mathbf{MC}}

\DeclareMathOperator{\Det}{Det}
\DeclareMathOperator{\Dol}{\scriptscriptstyle{Dol}}
\DeclareMathOperator{\stacky}{St}
\DeclareMathOperator{\Coha}{\mathcal{A}}
\DeclareMathOperator{\RCoha}{\widetilde{\mathcal{A}}}

\newcommand{\Dulf}{\mathcal{D}^{\geq, \mathrm{lf}}}
\newcommand{\Dblf}{\mathcal{D}^{\leq, \mathrm{lf}}}
\newcommand{\Db}{\mathcal{D}^{\bdd}}
\newcommand{\Dub}{\mathcal{D}}
\newcommand{\Ddg}{\mathcal{D}_{\dg}}

\DeclareMathOperator{\Coh}{\mathbf{Coh}}
\DeclareMathOperator{\QCoh}{QCoh}
\DeclareMathOperator{\NS}{NS}
\DeclareMathOperator{\bdd}{b}

\newcommand{\into}{\hookrightarrow}
\newcommand{\onto}{\twoheadrightarrow}
\newcommand{\Ab}{\mathbf{b}}
\newcommand{\Am}{\mathbf{m}}
\DeclareMathOperator{\Hoch}{Hoch}

\DeclareMathOperator{\CHiggs}{{\CH}iggs}
\DeclareMathOperator{\Calg}{\mathbf{Calg}}
\DeclareMathOperator{\dgCat}{\mathbf{dgCat}}
\DeclareMathOperator{\Cocoal}{\mathbf{Coalg}}
\DeclareMathOperator{\Cart}{\mathbf{CArt}}
\DeclareMathOperator{\Sset}{\mathbf{SSet}}
\DeclareMathOperator{\Betti}{\scriptscriptstyle{Betti}}
\let \ol=\overline
\let \ul=\underline
\newcommand{\sssct}{\subsubsection{}}
\DeclareMathOperator{\gr}{gr}
\DeclareMathOperator{\Mat}{Mat}
\DeclareMathOperator{\CatHiggs}{\mathbf{Higgs}}
\DeclareMathOperator{\HC}{HC}
\DeclareMathOperator{\Map}{Map}
\DeclareMathOperator{\RHom}{\mathbf{R}Hom}
\DeclareMathOperator{\REnd}{\mathbf{R}End}
\DeclareMathOperator{\CHC}{CC}

\DeclareMathOperator{\Lmod}{-Mod}
\DeclareMathOperator{\dg}{dg}

\newcommand{\Msp}{\mathcal{M}}
\newcommand{\Mst}{\mathfrak{M}}

\DeclareMathOperator{\Quot}{Quot}
\DeclareMathOperator{\supp}{supp}

\DeclareMathOperator{\GL}{GL}

\DeclareMathOperator{\Ku}{Ku}

\newcommand{\loc}{\textrm{loc}}